\documentclass[10pt, a4paper]{amsart}
\usepackage{amsmath,amssymb, amsthm} 
\usepackage[varg]{pxfonts}   
\usepackage[T1]{fontenc}        
\usepackage[latin1]{inputenc}
\usepackage[pdftex]{hyperref}  

\theoremstyle{plain}
    \newtheorem{thm}{Theorem}
    
    \newtheorem{corol}[thm]{Corollary}
    \newtheorem{prop}{Proposition}[section]
    \newtheorem{otherthm}[prop]{Theorem}
    \newtheorem{lemma}[prop]{Lemma}

\theoremstyle{definition}
    
\theoremstyle{remark}
    \newtheorem{rem}[prop]{Remark}
    \newtheorem{question}{Question}
    \newtheorem*{claim}{Claim}
    \newtheorem*{ack}{Acknowledgements}

\numberwithin{equation}{section}

\newcommand{\R}{\mathbb{R}}\newcommand{\Z}{\mathbb{Z}}\newcommand{\N}{\mathbb{N}}

\newcommand{\D}{\mathbb{D}}

\newcommand{\cD}{\mathcal{D}}
\newcommand{\cE}{\mathcal{E}}\newcommand{\cF}{\mathcal{F}}\newcommand{\cG}{\mathcal{G}}

\newcommand{\cM}{\mathcal{M}}\newcommand{\cN}{\mathcal{N}}\newcommand{\cO}{\mathcal{O}}
\newcommand{\cQ}{\mathcal{Q}}\newcommand{\cR}{\mathcal{R}}\newcommand{\cS}{\mathcal{S}}
\newcommand{\cU}{\mathcal{U}}\newcommand{\cV}{\mathcal{V}}\newcommand{\cW}{\mathcal{W}}\newcommand{\cX}{\mathcal{X}}
\newcommand{\cZ}{\mathcal{Z}}
\newcommand{\eps}{\varepsilon}

\renewcommand{\epsilon}{\varepsilon}
\renewcommand{\setminus}{\smallsetminus}
\renewcommand{\emptyset}{\varnothing}

\newcommand{\GL}{\mathrm{GL}}

\newcommand{\CB}{\mathit{CB}}

\newcommand{\Id}{\mathit{Id}}
\newcommand{\Diff}{\mathrm{Diff}}
\newcommand{\PH}{\mathit{PH}}

\DeclareMathOperator*{\supp}{supp}
\DeclareMathOperator{\interior}{int}

\renewcommand{\angle}{\measuredangle}
\newcommand{\m}{\mathbf{m}}

\newcommand{\hm}{\widehat{m}}
\newcommand{\hJ}{\widehat{J}}
\newcommand{\hW}{\widehat{\cW}}
\newcommand{\hE}{\widehat{E}}
\newcommand{\hB}{\widehat{B}}

\def\CB{{\mathit{CB}}}

\keywords{Partial hyperbolicity, center bunching, ergodicity,
symplectic diffeomorphisms}


\begin{document}

\title[Nonuniform Center Bunching and Generic Ergodicity]
{Nonuniform Center Bunching and
the Genericity of Ergodicity among $C^1$ Partially Hyperbolic Symplectomorphisms}

\author[A.~Avila]{Artur Avila}
\address{CNRS UMR 7599, Laboratoire de Probabilit\'es et Mod\`eles al\'eatoires,
Universit\'e de Paris VI \\ FRANCE}
\curraddr{IMPA \\ Estrada Dona Castorina, 110, 22460-320, Rio de Janeiro, RJ \\ BRAZIL}
\urladdr{www.impa.br/$\sim$avila}
\email{artur@math.sunysb.edu}

\author[J.~Bochi]{Jairo Bochi}
\address{Departamento de Matem\'atica, Pontif\'{\i}cia Universidade Cat\'olica do Rio de Janeiro \\
Rua Mq.\ S.\ Vicente, 225,
22453-900, Rio de Janeiro, RJ \\ BRAZIL}
\urladdr{www.mat.puc-rio.br/$\sim$jairo}
\email{jairo@mat.puc-rio.br}

\author[A.~Wilkinson]{Amie Wilkinson}
\address{Department of Mathematics, Northwestern University\\2033 Sheridan Rd., Evanston, IL, 60208 \\ USA}
\urladdr{www.math.northwestern.edu/$\sim$wilkinso}
\email{wilkinso@math.northwestern.edu}

\begin{abstract}
We introduce the notion of nonuniform center bunching for partially hyperbolic diffeomorphims,
and extend previous results by Burns--Wilkinson and Avila--Santamaria--Viana.
Combining this new technique with other constructions
we prove that $C^1$-generic partially hyperbolic symplectomorphisms are ergodic.
We also construct new examples of  stably ergodic partially hyperbolic diffeomorphisms.
\end{abstract}

\date{December, 2008. Revised May, 2009.}

\maketitle

\begin{center}
\textsc{Resserrement central non-uniforme et la généricité de l'ergodicité parmi les $C^1$-symplectomorphismes partiellement hyperboliques}
\end{center}

\medskip

{\footnotesize
\begin{quote}
\textsc{Résumé.} 
Nous introduisons une notion non-uniforme de resserrement central pour les difféomorphismes partiellement hyperboliques qui nous permet de généraliser quelques résultats de Burns - Wilkinson et Avila - Santamaria - Viana.
Cette nouvelle technique est utilisée, en combinaison avec d'autres constructions, pour démontrer
la généricité de l'ergodicité parmi les difféomorphismes  symplectiques partiellement hyperboliques de classe $C^1$.
De plus, nous obtenons de nouveaux exemples de dynamiques stablement ergodiques.
\end{quote}
}





\section{Introduction}

\subsection{Abundance of Ergodicity}

Let $(M,\omega)$ be a closed (ie, compact without boundary)
symplectic $C^\infty$ manifold of dimension $2N$.
Let $\Diff^1_\omega(M)$ be the space of $\omega$-preserving $C^1$ diffeomorphisms,
endowed with the $C^1$ topology.
Let $m$ be the measure induced by the volume form $\omega^{\wedge N}$,
normalized so that $m(M)=1$.

\smallskip

Let $\PH^1_\omega(M)$ be the
set of diffeomorphisms $f \in \Diff^1_\omega(M)$ that are partially hyperbolic, i.e.,
there exists an invariant splitting $T_x M = E^u(x) \oplus E^c(x) \oplus E^s(x)$, 
into nonzero bundles, 
and a positive integer $k$ such that for every $x \in M$,
\begin{equation}\label{e.def ph}
\begin{gathered}
\|(Df^k |E^u(x))^{-1}\|^{-1} > 1 > \|Df^k|E^s(x)\| \, , \\
\|(Df^k |E^u(x))^{-1}\|^{-1} > \|Df^k|E^c(x)\| \geq
\|(Df^k |E^c(x))^{-1}\|^{-1} > \|Df^k|E^s(x)\| \, .
\end{gathered}
\end{equation}
Such a splitting is automatically continuous.

\begin{thm}\label{t.main}
The set of ergodic diffeomorphisms is residual in $\PH^1_\omega(M)$.
\end{thm}


Our result is motivated by the following well-known conjecture of
Pugh and Shub \cite{Pugh Shub JEMS}: \emph {There is a $C^2$ open and dense
subset of the space of $C^2$ volume-preserving partially hyperbolic
diffeomorphisms formed by ergodic maps}.
Among the known results in this direction, we have:
\begin{itemize}
\item F.\ and M.\ A.\ Rodriguez-Hertz, and Ures~\cite{RRU 1d center}
proved that $C^r$-stable ergodicity is dense among
$C^r$ volume-preserving partially hyperbolic diffeomorphisms with one-dimensional center
bundle, for all $r\ge 2$.  
(See also \cite{BMVW} for an earlier result.)

\item F.\ and M.\ A.\ Rodriguez-Hertz, Tahzibi, and Ures~\cite{RRTU 2d center} proved 
that ergodicity holds on a $C^1$ open and dense subset of the
$C^2$ volume-preserving partially hyperbolic diffeomorphisms with two-dimensional center
bundle.
\end{itemize}
Together with the result from Avila~\cite{Avila}, it follows that ergodicity is $C^1$ generic
among volume-preserving partially hyperbolic diffeomorphisms with center dimension
at most~$2$.
On the other hand, the techniques yielding the results above seem less effective 
for the understanding of the case of symplectic maps, and indeed
Theorem~\ref{t.main} is the first result on denseness of ergodicity for non-Anosov
partially hyperbolic symplectomorphisms, even allowing for constraints on the center dimension.
Our approach develops some new tools of independent interest, as we explain next.

\subsection{Center Bunching Properties}   

To support their conjecture,
Pugh and Shub~\cite{Pugh Shub JEMS} provided a criterion 
for a volume-preserving partially hyperbolic map to be ergodic, 
based on the property of \emph{accessibility}, together with some technical hypotheses.
A significantly improved version of this criterion was
obtained by Burns and Wilkinson~\cite {Burns Wilk}: accessibility and \emph{center bunching} imply ergodicity.
Dolgopyat and Wilkinson~\cite{Dolgo Wilk} showed that
accessibility is open and dense in the $C^1$ topology, but center bunching is not
a dense condition unless the center dimension is~$1$ (which cannot happen for symplectic maps).
In this paper we introduce and exploit a weaker condition, called \emph{nonuniform center bunching}.

In the context of general (not necessarily volume-preserving)
partially hyperbolic diffeomorphisms,
the center bunching hypothesis in \cite{Burns Wilk}
is a global, uniform property, requiring that at every point in the
manifold, the nonconformality of the action on the center bundle
be dominated by the hyperbolicity in both the stable and unstable bundles.
By contrast, the nonuniform center bunching
property introduced here is a property of asymptotic nature
about the orbit of a single point; 
it is the intersection of a forward bunching
property of the forward orbit and a backward bunching property
of the backward orbit. 
The precise definitions are slightly technical (see  Section~\ref{s.nonunif cb}).
However, for Lyapunov regular points (which by Oseledets' theorem have full probability),
forward (resp.\ backwards) center bunching means that the biggest difference 
between the Lyapunov exponents in the center bundle is
smaller than the absolute value of the exponents in the stable (resp.\ unstable) bundle.
The set $CB^+$ of forward center bunched points for a partially hyperbolic diffeomorphism $f$
has the useful property of being {\em $\cW^s$-saturated}, meaning that
it is a union of entire stable manifolds of $f$; similarly the set $CB^{-}$ of backward
center bunched points is $\cW^u$-saturated, i.e.\ a union of unstable manifolds.

Our next main result, Theorem~\ref{t.nonunifbw},
generalizes the core result of \cite{Burns Wilk} (Theorem~5.1 of that paper).
It states that for any $C^2$ partially hyperbolic diffeomorphism, 
{\em the set of Lebesgue density points of any bi essentially saturated set meets $CB^+$
in a $\cW^s$-saturated set and $CB^-$ in a $\cW^u$-saturated set.} 
(A bi essentially saturated set is one that
coincides mod~$0$ with  a $\cW^s$-saturated set and mod~$0$ with a $\cW^u$-saturated set.)

Burns and Wilkinson \cite{Burns Wilk} obtain their ergodicity criterion 
as a simple consequence of their technical core result.  Indeed, assuming 
accessibility (or even essential accessibility), ergodicity in \cite{Burns Wilk} follows
in one step from the core result, using a Hopf argument; it is not necessary 
to establish local ergodicity first (as one does in proving ergodicity for hyperbolic systems). 
It is unclear to us whether the the Burns--Wilkinson criterion for ergodicity can be improved by
replacing uniform center bunching by almost everywhere nonuniform center bunching,
in part because the uniform version in \cite{Burns Wilk}
is by nature {\em not} a ``local ergodicity'' result.
In reality, it is possible to deduce a new ergodicity criterion 
(Corollary~\ref{c.weird}) from Theorem~\ref{t.nonunifbw}.
Namely, ergodicity follows from almost everywhere nonuniform center bunching
together with a stronger form of essential accessibility, where we only
allow $su$-paths whose corners are center-bunched points.
While this accessibility condition is far from automatic,
it can be verified in some interesting classes of examples:
see \S\ref{ss.applications} below.


\subsection{Outline of the Proof of Theorem~\ref{t.main}}\label{ss.outline}

Let us explain how nonuniform center bunching combines with other ingredients to yield Theorem~\ref{t.main}. 
Take a symplectomorphism with the following $C^1$ generic properties:
\begin{itemize}
\item[(a)] it is stably accessible, by Dolgopyat and Wilkinson~\cite{Dolgo Wilk};
\item[(b)] all central Lyapunov exponents vanish at almost every point, by Bochi \cite{Bochi sympl}.
\end{itemize}
%
Notice that property~(b) implies almost every point is center bunched.
But Theorem~\ref{t.nonunifbw} requires $C^2$ regularity.
This is achieved by taking a perturbation, which still has property~(a), but loses property~(b).
What happens is that 
each point in some set of measure close to $1$ 
has small center Lyapunov exponents
and thus is center bunched.

Before getting useful consequences from Theorem~\ref{t.nonunifbw}, we need to provide a local source of ergodicity.
This is achieved through a novel application of the Anosov--Katok \cite{Anosov Katok} examples.
(By comparison, \cite{RRTU 2d center} uses Bonatti--D\'{\i}az blenders.) 
We  proceed as follows.
By perturbing, we find a periodic point whose center eigenvalues have unit modulus.
Perturbing again, we create a disk tangent to the center direction that is invariant by a power of the map.
We can choose any dynamics close to the identity on this disk,
so we select an ergodic Anosov--Katok map.
Ergodicity is spread from the center disk to a ball around the periodic point using Theorem~\ref{t.nonunifbw},
and then to the whole manifold by accessibility.
(In fact, since the set of center bunched points is not of full measure,
a $G_\delta$ argument is necessary to conclude ergodicity --
see Section~\ref{s.proof main} for the precise procedure.)


\subsection{Further Applications of Nonuniform Center Bunching}\label{ss.applications}

By means of our ergodicity criterion (Corollary~\ref{c.weird}) 
we construct an example of a stably ergodic partially hyperbolic diffeomorphism 
that is almost everywhere nonuniformly center bunched (but not center bunched
in the sense of \cite{Burns Wilk}) in a robust way.

We also prove in this paper an extension of Theorem~\ref{t.nonunifbw}
to sections of bundles over partially hyperbolic diffeomorphisms.
This result, Theorem~\ref{t.nonunifjimmy}, brings into the nonuniform
setting a recent result of Avila, Santamaria and Viana \cite{ASV},
which they use to show that the generic bunched $SL(n,\R)$ cocycle
over an accessible, center bunched, volume-preserving partially hyperbolic diffeomorphism
has a nonvanishing exponent. 
The result from \cite{ASV} has also been used 
in establishing measurable rigidity of solutions to
the cohomological equation over center-bunched systems; see \cite{Wi}.  
Theorem~\ref{t.nonunifjimmy} has similar
applications in the setting where nonuniform center bunching
holds, and we detail some of them in Section~\ref{s.jimmy}.


We conceive that our methods
may be further extended to apply in certain ``singular
partially hyperbolic'' contexts where
partial hyperbolicity holds on an open, noncompact subset of the manifold $M$
but decays in strength near the boundary.  Such conditions hold, for
example, for geodesic flows on certain nonpositively curved manifolds.
Under suitable accessibility hypotheses, these systems should be
ergodic with respect to volume.

\subsection{Questions}
Combining results  of  \cite{Dolgo Wilk} and Brin \cite{Brin},
one obtains that topological transitivity holds for a $C^1$ open and
dense set of partially hyperbolic symplectomorphisms.  On the other hand,
the $C^1$-interior of the ergodic symplectomorphisms is contained in the
partially hyperbolic diffeomorphisms \cite{Horita Tah, SX robust trans}.
This suggests the following natural question.
\begin{question}
Can Theorem~\ref{t.main} be improved to an open (and dense) instead of residual set?
\end{question}
Notice that it is not known even whether the set of
$C^1$ Anosov ergodic maps has non-empty interior.

\medskip

Dropping partial hyperbolicity, 
recall that $C^1$ generic symplectic and volume-preserving diffeomorphisms
are transitive by \cite{Arn Bon Cro} and \cite{Bon Cro}, while
ergodicity is known to be $C^0$-generic among volume-preserving homeomorphisms
by \cite{OU}.
So the following well-known question arises:
\begin{question}
Is ergodicity generic among $C^1$ symplectic and volume-preserving
diffeomorphisms?
\end{question}

\subsection{Organization of the Paper}

In Section~\ref{s.nonunif cb} 
we define nonuniform center bunching, state Theorem~\ref{t.nonunifbw},
and derive Corollary~\ref{c.weird} from it.

In Section~\ref{s.proof main} we prove Theorem~\ref{t.main} following the outline 
given in \S\ref{ss.outline}.
As we have explained, the proof uses the existence (after perturbation)
of a periodic point with elliptic central behavior.
Such a result goes along the lines of \cite {Bon Diaz Pujals, Horita Tah, SX robust trans},
but we have not been able to find a precise reference.  
In Section~\ref{s.elliptic}, which can be read independently from the rest of the paper, 
we provide a proof of this result by reducing it to its ergodic counterpart and
applying the Ergodic Closing Lemma.
This approach is different from the one taken in the literature.
For this reason, we included an appendix explaining how to use it to reobtain 
some results from \cite{Bon Diaz Pujals}.

The proof of Theorem~\ref{t.nonunifbw}, despite having much in common with \cite{Burns Wilk},
is given here in full detail in  Section~\ref{s.amie}.
In Section~\ref{s.jimmy} we formulate and prove the more general Theorem~\ref{t.nonunifjimmy}.
The new examples of stably ergodic maps are constructed in
Section~\ref{s.examples}.

\begin{ack}
We would like to thank N.~Gourmelon for the idea used in the proof of Lemma~\ref{l.nicolas}.
This research was partially conducted during the period A.~A.\ served as a Clay Research Fellow.
J.~B.\ is partially supported by CNPq, and A.~W.\ is partially supported by the NSF.
\end{ack}

\section{Nonuniform Center Bunching and Consequences} \label{s.nonunif cb}

Throughout this section, $f$ denotes a fixed $C^2$ partially hyperbolic of a closed manifold $M$ of dimension~$d$.
(We do not require $f$ to be symplectic or even volume-preserving.)
Using a result of Gourmelon~\cite{Gourmelon}, 
we take a Riemannian metric $\| \mathord{\cdot}\|$ on $M$
for which relations~\eqref{e.def ph} hold with $k=1$.

\begin{rem}\label{r.defs ph}
The notion of partial hyperbolicity we use in this paper is called \emph{relative}.
There is a stronger form of partial hyperbolicity, called \emph{absolute}, 
which asks for the existence of a Riemannian metric such that
$\|(Df |E^u(x))^{-1}\|^{-1}>\max(1,\|Df|E^c(y)\|)$ and
$\min(1,\|(Df |E^c(y))^{-1}\|^{-1})>\|Df|E^s(z)\|$
for every $x$, $y$, $z\in M$; 
see~\cite{AV flavors}.
\end{rem}

\subsection{Saturated Sets}
If $\cF$ is a foliation with smooth leaves,
a set $X \subseteq M$ is said to be \emph{$\cF$-saturated} if it is a union of entire leaves of $\cF$.
We say that a measurable set $X$ is \emph{essentially $\cF$-saturated} if it
coincides Lebesgue mod $0$ with a $\cF$-saturated set.

We also say that a set $X$ is \emph{$\cF$-saturated at a point $x$} if
there exist $0<\delta_0<\delta_1$ such that for any $z\in X \cap B(x,\delta_0)$, we have $\cF(z,\delta_1)\subset X$.
(Here $\cF(z,\delta_1)$ denotes the
connected component of $\cF(z)\cap B(z,\delta_1)$ containing $z$.)

A measurable set $X$ is called \emph{bi essentially saturated}
if it is both essentially $\cW^u$-saturated and essentially $\cW^s$-saturated.
(Here $\cW^u$ and $\cW^s$ are the unstable and stable foliations of the partially
hyperbolic diffeomorphism $f$.)

\subsection{Nonuniform Center Bunching}

If $A:V\to W$ is a linear transformation between Banach spaces, we denote
by ${\bf m}(A)$ the {\em conorm} of $A$, defined by
$${\bf m}(A) = \inf_{v\in V, \, \|v\|=1} \|A(v)\|.$$
If $A$ is invertible, then ${\bf m}(A) = \|A^{-1}\|^{-1}$.

We say that a point $p\in M$ is \emph{forward center bunched} if
there exist $\theta>1$ and
a sequence $0=i_0<i_1<\cdots$ such that
$i_{k+1}/i_k \to 1$ and for every $k \geq 0$,
$$
\| D_{f^{i_k}(p)} f^{i_{k+1}-i_k}|E^s \|^{-1}
\geq \theta^{i_{k+1}-i_k} \cdot
\frac {\|D_{f^{i_k}(p)} f^{i_{k+1}-i_k}|E^c\|} {\m \big(D_{f^{i_k}(p)} f^{i_{k+1}-i_k}|E^c \big)} \, .
$$
The point $p$ is called \emph{backwards center bunched} if it is forward center bunched with respect to $f^{-1}$.
The set of forward, resp.\ backwards, center bunched points is denoted by $\mathit{CB}^+$, resp.\ $\mathit{CB}^-$.
Also denote $\CB = \CB^+ \cap \CB^-$.
It is easy to see that these sets are $f$-invariant.
Moreover, in Section~\ref{s.amie} we show:

\begin{prop}\label{p.CB is sat}
$\mathit{CB}^+$ is $\cW^s$-saturated and
$\mathit{CB}^-$ is $\cW^u$-saturated.
\end{prop}

A much deeper property is:

\begin{thm}\label{t.nonunifbw}
Let $f$ be a $C^{2}$ partially hyperbolic diffeomorphism.
Let $X$ be a bi essentially saturated set, and let
$\hat{X}$ denote the set of Lebesgue density points of $X$.
Then $\hat{X} \cap \mathit{CB}^+$ is $\cW^s$-saturated
and $\hat{X} \cap \mathit{CB}^-$ is $\cW^u$-saturated.
\end{thm}

We remark that the hypotheses of Theorem~\ref{t.nonunifbw} are weaker
than the center bunching hypothesis in \cite{Burns Wilk}.
In the setting of \cite{Burns Wilk},
$\mathit{CB}^+=\mathit{CB}^- = M$ and one takes $i_k=k$ in the definition of forward center bunching.
(In fact, the center bunching hypothesis in  \cite{Burns Wilk}
is equivalent to the condition $\mathit{CB}^+=\mathit{CB}^- = M$
see Remark~\ref{r.unif CB} below.)

Another remark is that, as in~\cite{Burns Wilk}, it is essential that
$X$ is \emph{both} essentially $\cW^u$-saturated and  essentially $\cW^s$-saturated
in order to conclude anything.

\subsection{Relation with Lyapunov Spectrum}

Let us formulate sufficient conditions for center bunching in terms of Lyapunov exponents.


Oseledets' Theorem asserts that there exists a set of full probability
(that is, a Borel set of full measure with respect to any $f$-invariant probability)
where Lyapunov exponents and Oseledets' splitting are defined
(see for example \cite[Theorem~3.4.11 and Remark~4.2.8]{LArnold}).
The elements of this set are called \emph{Lyapunov regular points}.

If $p\in M$ is a Lyapunov regular point,
we write the Lyapunov exponents (with multiplicity) of $f$ at $p$ as:
$$
\underbrace{\lambda_1        \ge \cdots \ge \lambda_k}_{E^u} >
\underbrace{\lambda_{k+1}    \ge \cdots \ge \lambda_\ell}_{E^c} >
\underbrace{\lambda_{\ell+1} \ge \cdots \ge \lambda_d}_{E^s} \, .
$$
(The braces are shorthands meaning that $\dim E^u = k$, $\dim E^c = \ell - k$, $\dim E^s = d - \ell$.)
We say the Lyapunov spectrum of $f$ at $p$ satisfies
the \emph{forward center bunched condition} if
$$
\lambda_{k+1} - \lambda_\ell < - \lambda_{\ell+1} \, ,
$$
and the \emph{backwards center bunched condition} in the case that
$$
\lambda_{k+1} - \lambda_\ell <  \lambda_k \, .
$$

Notice that if $f$ is symplectic
then, by the symmetry between the exponents,
the forward and the backwards center bunching conditions are
equivalent to:
$$
2 \lambda_{k+1} <  \lambda_k \, .
$$

\begin{prop}\label{p.spectrum}
A Lyapunov regular point is forward (resp.\ backwards) center bunched
if and only if  its spectrum satisfies the forward (resp.\ backwards) center bunched condition.
\end{prop}

\begin{proof}
We only need to prove the forward part of the proposition,
and the backwards part will follow by symmetry.

Fix a point $p$ and define
\begin{equation}\label{e.Theta}
\Theta(j,n) = \| D_{f^{j}(p)} f^{n}|E^s \|^{-1}
\cdot \m \big(D_{f^{j}(p)} f^{n}|E^c \big) \cdot \|D_{f^{j}(p)} f^{n}|E^c\|^{-1}  \quad j, n \ge 0 .
\end{equation}
Assume that $p$ is forward center bunched.
Let $\theta$ and $i_k$ be as in the definition of forward center bunching;
then $\Theta(i_k, i_{k+1}-i_k) > \theta^{i_{k+1}-i_k}$.
We have
$$
\Theta(0,i_k) \ge \Theta(0, i_1) \Theta(i_1, i_2 - i_1) \cdots \Theta(i_{k-1}, i_k - i_{k-1}) \ge \theta^{i_k} \, ,
$$
and in particular
\begin{equation}\label{e.non criterion}
\limsup_{n \to +\infty} \frac1n \log \Theta(0,n)> 0.
\end{equation}
If $p$ is Lyapunov regular then the $\limsup$ above equals $-\lambda_{\ell+1} +\lambda_\ell - \lambda_{k+1}$.
Thus $p$ has center bunched Lyapunov spectrum.

Conversely, assume that the point $p$ is Lyapunov regular and has center bunched Lyapunov spectrum.
Fix some $\tau$ with $0 < \tau < - \lambda_{\ell+1} - \lambda_{k+1} + \lambda_\ell$.
We claim that
\begin{equation} \label{e.zz}
\text{for every $\delta>0$ there exists $c_\delta > 0$ such that
$\Theta(j,n) > c_\delta e^{-\delta j} e^{\tau n}$ for all $j$, $n \ge 0$.}
\end{equation}
Before giving the proof, let us see how to conclude from here.
Let $i_0 = 0$.
Inductively define $i_{k+1}$ as the least $i>i_k$ such that
$\Theta(i_k, i-i_k) > e^{(\tau/2)(i-i_k)}$.  
Let us see that this sequence of times satisfies the requirements of the
definition of forward center bunching, with $\theta = \tau/2$.
For any $\delta>0$, we have
$$
c_\delta e^{-\delta i_k} e^{\tau (i_{k+1} - i_k - 1)} < \Theta(i_k, i_{k+1} - i_k - 1)
\le e^{(\tau/2)(i_{k+1}-i_k-1)} \, .
$$
It follows that if $i_k$ is sufficiently large (depending on $\delta$) then $(i_{k+1}-i_k)/i_k < 3\delta/\tau$.
This proves that $i_{k+1}/i_k \to 1$
and hence that $p \in \mathit{CB}^+$.

We are left to prove~\eqref{e.zz}.
For $1 \le i \le d = \dim M$, let $E^i(p)$ be the Oseledets space
corresponding to the Lyapunov exponent $\lambda_i(p)$.
(This notation is not standard because those spaces are not necessarily different.)
A consequence of the Lyapunov regularity 
of $p$ is that,
for each $i = 1, \ldots, d$,
the quotient $n^{-1} \log\|D_p f^n (v)\|$ converges to $\lambda_i$
uniformly over unit vectors $v \in E^i(p)$.
Thus for every $\delta>0$ there exists $K_\delta>1$ such that
$$
K_\delta^{-1} e^{(\lambda_i-\delta)n} \le  \|D_p f^n (v)\|
\le K_\delta e^{(\lambda_i+\delta)n} \, ,
\text{for all unit vectors $v\in E^i(p)$ and $n \ge 0$.}
$$
Hence, for each $n$, $j \ge 0$, we have
\begin{equation}\label{e.zzz}
\|D_{f^j(p)} f^n|_{E^i}\| \le \|D_p f^{n+j}|_{E^i}\|  / \m(D_p f^j|_{E^i})
                     \le K_\delta^2 e^{2\delta j} e^{(\lambda_i + \delta)n} \, .
\end{equation}
Another consequence of Lyapunov regularity (see \cite[Corollary~5.3.10]{LArnold})
is that
the angles between (sums of different) Oseledets spaces along the orbit of $p$ are subexponential.
In particular, for each $\delta >0$ we can find $K'_\delta>1$ such that
$$
(K'_\delta)^{-1} e^{-\delta (j+n)} \le
\frac{\|D_{f^j(p)} f^n|_{E^s}\|}{\max_{i \in [\ell+1,d]} \|D_{f^j(p)} f^n|_{E^i}\|}
\le K_\delta' e^{\delta (j+n)} \, ,
\quad \text{for each $n$, $j \ge 0$.}
$$
It follows from \eqref{e.zzz} that there exists $K_\delta''>1$ such that
$$
\|D_{f^j(p)} f^n|_{E^s}\| \le K''_\delta e^{3\delta j} e^{(\lambda_{\ell+1} + 2 \delta)n} \, ,
\quad \text{for each $n$, $j \ge 0$.}
$$
This controls the first term in~\eqref{e.Theta}.
The other two are dealt with in an analogous way, and \eqref{e.zz} follows.
\end{proof}

\begin{rem}
If $p \in \mathit{CB}^+$ then we have seen that~\eqref{e.non criterion} holds,
where $\Theta$ is defined by~\eqref{e.Theta}.
Let us show that condition~\eqref{e.non criterion} alone does not imply forward center bunching.
First notice that if $p \in \mathit{CB}^+$ then
\begin{equation}\label{e.absurd}
\liminf_{m \to \infty} \frac{1}{n_m} \log \Theta(j_m, n_m) > 0 
\text{ for any sequences $j_m$, $n_m$ with $n_m > \frac{1}{10} j_m \to \infty$.}
\end{equation}
Now let
$$
A = \begin{pmatrix} e^{-2} & 0   \\ 0 & e^{-1} \end{pmatrix} , \quad
B = \begin{pmatrix} e^{-1/2} & 0 \\ 0 & e^{-1} \end{pmatrix} , \quad
C = \begin{pmatrix} 1 & 0 \\ 0 & e^{3/4} \end{pmatrix}.
$$
Assume that $D_{f^j(p)} f|_{E^c}$ equals $C$ for every $j \ge 0$,
while the sequence $D_{f^j(p)} f|_{E^s}$, $j \ge 0$ is given by:
$$
A, B, A, A, B, B, A \text{ ($4$ times)}, B \text{ ($4$ times)}, A \text{ ($8$ times)}, B \text{ ($8$ times)},  \ldots
$$
Notice that for every $n \ge 0$, we have
$\|D_p f^n|_{E^s}\| = e^{-n}$, $\m(D_p f^n|_{E^c}) = 1$, and
$\|D_p f^n|_{E^c}\| = e^{(3/4)n}$,
so condition~\eqref{e.non criterion} is satisfied.
On the other hand, if $j = 2^{m+1} + 2^{m} - 2$ and $n=2^{m}$
then $D_{f^j(p)} f^n  = B^n$
and therefore $\Theta(j,n) = e^{(-1/4)n}$.
Hence \eqref{e.absurd} does not hold and so $p$ is not forward center bunched.
\end{rem}

\begin{rem}\label{r.unif CB}
If $\mathit{CB}^+ = \mathit{CB}^- = M$ then $f$ is center bunched in the sense of \cite{Burns Wilk}.
Indeed, let $\Theta_p(j,n)$ be as in \eqref{e.Theta}, with a subscript to indicate dependence on the point.
Assuming $\mathit{CB}^+ = M$, compactness implies that there exists $\theta>1$
and $m$ such that for every $p \in M$ there exists $i$ with $1 \le i \le m$ such that
$\Theta_p (0, i) > \theta$.
It follows that there is $c>0$ such that
$\Theta_p (0, n) > c \theta^{n/m}$.
We reason analogously for $f^{-1}$.
The conclusion follows from an adapted metric argument along the lines of \cite{Gourmelon}.
\end{rem}

\subsection{An Ergodicity Criterion}

Let us extract a criterion for ergodicity from Theorem~\ref{t.nonunifbw}.
(It is not used in the proof of Theorem~\ref{t.main},
so the reader can skip the rest of this section.)

\begin{corol}\label{c.weird}
Let $f$ be a $C^2$ partially hyperbolic volume-preserving diffeomorphism.
Let $\mathit{CB} = \mathit{CB}^+ \cap \mathit{CB}^-$ be the set of center bunched points.
Assume that almost every pair of points $x$, $y \in CB$
can be connected by an $su$-path whose corners are in $CB$.

Let $X$ be a bi~essentially saturated set such that $X\cap CB$ has positive measure.
Then $X$ has full measure in $CB$.
If $CB$ has full measure, then $f$ is ergodic, and in fact a $K$-system.	
\end{corol}


In Section~\ref{s.examples} we give applications of Corollary~\ref{c.weird}
to prove stable ergodicity of certain partially hyperbolic diffeomorphisms
that are {\em not} center bunched.

\begin{proof}[Proof of  Corollary~\ref{c.weird}]
Let $f$ and $X$ satisfy the hypotheses of Corollary~\ref{c.weird} and
let $\hat X$ be the set of Lebesgue density points of $X$.  Then
for almost every $x\in \hat X\cap CB$ and almost every $y\in CB$,
there is an $su$-path from $x$ to $y$ with corners $x_0=x,x_1,\ldots,x_k=y$ all lying
in $CB$ (that is, so that $x_i$ lies in $CB\cap (\cW^s(x_{i+1})\cup \cW^u(x_{i+1})$),
for $i=0,\ldots, k-1$).  Fix such an $x$ and $y$ and such an $su$-path.
Applying Theorem~\ref{t.nonunifbw} inductively to each pair $x_{i},x_{i+1}$,
we obtain that $x_i$ lies in $\hat X$, for $i=1,\ldots k$, and
so $y\in \hat X$.  This implies that almost every $y\in CB$ lies
in $\hat X$, and hence $X$ has full measure in $CB$.

A standard argument shows that a volume-preserving
partially hyperbolic diffeomorphism is ergodic if and only if every
bi essentially saturated, invariant set has measure
$0$ or $1$.
Moreover, $f$ is a $K$-system if every
bi essentially saturated set, invariant or not, has measure
$0$ or $1$ (see \cite{Burns Wilk}, Section 5).
If $CB$ has full measure, then any bi~essentially saturated
set has $0$ or full measure in $CB$, and hence has measure $0$ or $1$.
It follows that $f$ is ergodic, and in fact a $K$-system.
\end{proof}

\section{Proof of Theorem~\ref{t.main}}\label{s.proof main}



For $\eps>0$,
let us call a diffeomorphism $f \in \PH_\omega^1(M)$ \emph{$\eps$-nearly ergodic}
if for any bi~essentially saturated and mod~$0$ invariant set $X$, 
either $m(X) < \eps$ or $m(X) > 1- \eps$.\footnote{A related notion, $\eps$-ergodicity,
was considered by \cite{Tah}.}
The bulk of the proof of Theorem~\ref{t.main} consists in showing the following:
\begin{prop}\label{p.near ergodic}
For any  $\eps>0$, the $\eps$-nearly ergodic diffeomorphisms form a
dense subset of $\PH_\omega^1(M)$.
\end{prop}
In \S\S\ref{ss.zero}, \ref{ss.access}, and \ref{ss.c disk}
we review some results from the literature, which are used to prove
the proposition in \S\ref{ss.near erg}.
Then in \S\ref{ss.G delta} we explain how Proposition~\ref{p.near ergodic}
implies Theorem~\ref{t.main}.

\subsection{Zero Center Exponents} \label{ss.zero}

Given $f \in \PH^1_\omega(M)$,
the partially hyperbolic splitting $TM = E^u \oplus E^c \oplus E^s$
is not necessarily unique.
\emph{We consider from now on only the unique splitting of minimal center dimension.}
If this center dimension is constant on a $C^1$-neighborhood of $f$, we say that
$f$ has \emph{unbreakable} center bundle.
Such $f$'s form an open dense subset of $\PH^1_\omega(M)$ (by upper-semicontinuity of
the center dimension).

\medskip

To get center bunching, we will use the following:

\begin{otherthm}[Bochi~\cite{Bochi sympl}, Theorem C]\label{t.zero center}
There is a residual set $\cR \subset \PH^1_\omega(M)$
such that if $f \in \cR$ then all Lyapunov exponents in the center bundle vanish for a.e.~point.
\end{otherthm}

In other words, $\lambda^c(f)=0$ for generic $f$, where
$$
\lambda^c (f) = \lim_{n \to +\infty} \frac{1}{n} \int_M \log \|Df^n | E^c_f\| \; dm
              = \inf_{n} \frac{1}{n} \int_M \log \|Df^n | E^c_f\| \; dm.
$$
Notice that $\lambda^c(f)$ is an upper semicontinuous function of $f$.
Therefore, for any $\delta>0$, the set of $f \in \PH^1_\omega(M)$ with $\lambda^c(f)<\delta$
is open and dense (and thus, by~\cite{Zehnder}, it contains $C^2$ maps).

\subsection{Accessibility} \label{ss.access}

There are two results about accessibility that we will need:
one says that it is frequent, and the other gives a useful consequence.

\begin{otherthm}[Dolgopyat and Wilkinson \cite{Dolgo Wilk}] \label{t.DW}
There is an open and dense subset of $\PH^1_\omega (M)$
formed by accessible symplectomorphisms.
\end{otherthm}

\begin{otherthm}[Brin~\cite{Brin}]\label{t.Brin}
If $f$ is a $C^2$ volume-preserving partially hyperbolic diffeomorphism
with the accessibility property then almost every point has a dense orbit.
\end{otherthm}

In fact, Brin proved the result above for absolute\footnote{See Remark~\ref{r.defs ph}} partially hyperbolic maps.
Another proof was given by Burns, Dolgopyat, and Pesin, see \cite[Lemma~5]{Burns Dolgo Pesin}.
Their proof applies to relative partially hyperbolic maps (the weaker definition taken in this paper):
the only necessary modification is to use
the property of absolute continuity of stable and unstable foliations in the relative case,
which is proven by Abdenur and Viana in \cite{AV flavors}.

\subsection{Creating an Ergodic Center Disk} \label{ss.c disk}
The last ingredient we will need in the proof of Proposition~\ref{p.near ergodic}
is Lemma~\ref{l.c disk} below,
whose proof needs its own preparations.
We begin finding a suitable periodic point:

\begin{otherthm}\label{t.BDP}
Let $f$ have unbreakable center.
There exists a $C^1$-perturbation $\tilde f$
that has a periodic point with $\dim E^c$ eigenvalues of modulus $1$.
\end{otherthm}

This result can be obtained along the lines of
\cite{SX robust trans} or \cite{Horita Tah} (which prove symplectic versions of
the results of \cite{Bon Diaz Pujals}).
In Section~\ref{s.elliptic} we give a different proof, relying on \cite{Bochi sympl}
and the Ergodic Closing Lemma~\cite{Mane ergodic}.

The following symplectic pasting lemma is established using generating functions,
see \cite{Arbieto Matheus}:

\begin{lemma}\label{l.pasting}
Let $f \in \Diff^r_\omega(M)$.
Given $\eps > 0$ there is $\delta>0$ such that
if $U \subset M$ is an open set of diameter less than $\delta$,
and $g:U \to M$ is a $C^r$-symplectic map that is $\delta$-$C^1$-close to $f|U$,
then $g$ can be extended to some $\hat{g} \in \Diff^r_\omega(M)$ that is $\eps$-$C^1$-close to $f$.
\end{lemma}

The Anosov--Katok constructions enter here:

\begin{otherthm}\label{t.Katok}
Let $L:\R^{2N} \to \R^{2N}$ be a symplectic linear map with all
eigenvalues of modulus $1$.
Then there exist an arbitrarily
small neighborhood $U$ of $0$ in $\R^{2N}$ and an \emph{ergodic} symplectic
diffeomorphism $g:U \to U$ that is $C^\infty$-close to $L|U$.

\end{otherthm}

\begin{proof}
We may assume that $L$ has only simple eigenvalues
$\lambda_1^{\pm 1}, \ldots, \lambda_N^{\pm 1}$, all in the unit circle.
For $1 \leq i \leq N$, let $E^i$
be the $L$-invariant two-dimensional subspaces
associated to the eigenvalues $\lambda_i$, $\lambda_i^{-1}$.
Let $A_i:E^i \to \R^2$ be linear maps conjugating
$L|E^i$ to rigid rotations $R_i$.
Fix $\epsilon>0$.
If $g_i:\D \to \D$ are area-preserving maps of the unit disk $\D \subset \R^2$
that are $C^\infty$-close to $R_i | \D$, then the formula
$$
g(x) = \epsilon A_1^{-1} g_1(\epsilon^{-1} A_1 x_1) + \cdots + \epsilon A_N^{-1} g_N(\epsilon^{-1} A_N x_N),
\text{ where $x=x_1+ \cdots +x_N$ with $x_i \in E^i$,}
$$
defines, on a  small neighborhood $U$ of $0$ in $\R^{2N}$,
a symplectic map $g:U \to U$ that is $C^\infty$-close to $L|U$.
Now using a well known result of Anosov and Katok~\cite{Anosov Katok},
we choose maps $g_i$ as above that are weakly mixing.
It follows (see~\cite{Petersen}, Theorem~2.6.1)
that $g$ is weakly mixing (and hence ergodic) as well.
\end{proof}

\begin{lemma}\label{l.c disk}
For all $f$ in a $C^1$ dense subset of $PH_\omega^1(M)$, the following properties hold:
The map $f$ is $C^2$, and
there is a immersed closed disk $D^c$
such that:
\begin{enumerate}
\item the tangent space $T_x D^c$ coincides with $E^c(x)$ at each $x \in D^c$;
\item there is some $\ell$ such that $D^c$ is $f^\ell$-invariant, and moreover $D^c \cap f^i(D^c) = \emptyset$ for $0< i<\ell$;
\item the restriction of $f^\ell$ to $D^c$ is ergodic (with respect to the Riemannian volume $m_c$);
\item the disk is center bunched in the sense that
$$
\frac{\|Df|E^c(x)\|}{\m(Df|E^c(x))} < \min \big( \m(Df|E^u(x)) ,  \,  \|Df|E^s(x)\|^{-1}  \big)
\quad \text{for all $x\in D^c$};
$$
\item $f$ is dynamically coherent in a box neighborhood $B$ of $D^c$
(that is, there are foliations $\cW^c$, $\cW^{uc}$, $\cW^{cs}$ in the box $B$
that integrate the distributions $E^c$, $E^u \oplus E^c$, $E^c \oplus E^s$).
\end{enumerate}
\end{lemma}

\begin{proof}
We will explain how to perturb
a given $f$ in order to obtain the desired properties.

First use Theorem~\ref{t.BDP} to perturb $f$ and
find a periodic point $p$ of period $\ell$ such that all eigenvalues of $Df^\ell|E^c(p)$ have modulus $1$.
Also assume that these eigenvalues are distinct and their arguments are rational mod $2\pi$,
so that $Df^\ell|E^c(p)$ is diagonalizable and a power of it is the identity.

Take a neighborhood $U$ of $p$ that is disjoint from $f^i(U)$ for $1 \le i \le \ell-1$,
and such that there is a symplectic chart $\phi: U \to \R^{2N}$,
(that is, the form $\phi_* \omega$ coincides with $\sum_{i=1}^N dp_i \wedge dq_i$,
where $p_1$, \ldots, $p_N$, $q_1$, \ldots, $q_N$ are coordinates in $\R^{2N}$.)
We can also assume that $\phi(p) =0$ and $D\phi(p)$ sends
the spaces $E^u(p)$, $E^c(p)$, and $E^s(p)$
to the planes $p_1 \cdots p_u$, $p_{u+1} \cdots p_N q_{u+1} \cdots q_N$, and $q_1 \cdots q_u$,
respectively (where $u = \dim E^u$.)

Using Lemma~\ref{l.pasting}, we can perturb $f$ so that
$\phi \circ f^\ell \circ \phi^{-1}$ coincides with the linear map
$D\phi(p) \circ Df^\ell(p) \circ D\phi^{-1}(0)$ on a neighborhood of $p$.
For simplicity, we omit the chart in the writing, thus
$$
f^\ell(x_u,x_c,x_s) = (L_u(x_u), L_c(x_c), L_s(x_s)).
$$
Recall that there is a power of $L_c$ that is the identity.
So, if necessary changing the point~$p$ and the period~$\ell$,
we can assume $L_c$ is the identity.

Next we use Theorem~\ref{t.Katok}.
Let $g: D^c \to D^c$ be an ergodic symplectic diffeomorphism, where
the disk $D^c \subset \R^{\dim E^c}$ is contained in the chart domain.
Consider the (symplectic) map
$$
G(x_u,x_c,x_s) = (L_u(x_u), g(x_c), L_s(x_s)),
\text{ defined in a neighborhood of $(0,0,0)$.}
$$
Now use Lemma~\ref{l.pasting} again to find a global $\tilde f:M \to M$ close to $f$,
such that (still in charts)
$$
f^\ell(x_u,x_c,x_s) = G(x_u,x_c,x_s)
\text{ in a neighborhood of $(0,0,0)$.}
$$
Rename $\tilde f$ to $f$.
Then $f$ has all the desired properties.
\end{proof}

\subsection{Getting Near-Ergodicity} \label{ss.near erg}

\begin{proof}[Proof of Proposition~\ref{p.near ergodic}]
Fix an open set $\cU \subseteq \PH^1_\omega(M)$ and $\eps>0$.
Let $\delta>0$ be small.
Using Theorems~\ref{t.DW} and \ref{t.zero center},
we can assume that the set $\cU$ is composed of maps $f$ that are accessible and
satisfy $\lambda^c(f) < \delta$.
With a good choice of $\delta$, the latter property implies
that for any $f\in\cU$,
the measure of the set 
of Lyapunov regular points 
whose Lyapunov spectrum satisfies the center bunching condition is at least $1-\eps$.
Thus, by Proposition~\ref{p.spectrum},
$$
m(\mathit{CB}^+) > 1 - \eps.
$$

Now take $f\in \cU$ given by Lemma~\ref{l.c disk}.
Thus we have a center bunched disk $D^c$
that is ergodic (w.r.t.\ the measure $m_c$) by a power $f^\ell$,
disjoint from its first $\ell-1$ iterates,
and has a dynamically coherent box neighborhood $B$.

We will prove that $f$ is $\eps$-nearly ergodic.
So take any bi~essentially saturated set mod~$0$ invariant set $X$.
Let $X_1$ be the (invariant) set of its Lebesgue density points, and $X_0 = M \setminus X_1$.
By Proposition~\ref{p.CB is sat} and Theorem~\ref{t.nonunifbw},
$X_j \cap \mathit{CB}^+$ is $\cW^s$-saturated and 
$X_j \cap \mathit{CB}^-$ is $\cW^u$-saturated for both $j=0,1$. 

The map $f^\ell$ has the invariant ergodic measure $m_c$, supported on $D^c$.
Thus, for some $i \in \{0,1\}$ (that will be kept fixed in the sequel),
$$
m_c(X_i) = 0.
$$

By Oseledets' Theorem, $m_c$-almost every point is Lyapunov regular for~$f^\ell$,
and hence for $f$ as well.
By Property~4 in Lemma~\ref{l.c disk}, all these points have center bunched Lyapunov spectrum,
and thus are (forward and backwards) center bunched, by Proposition~\ref{p.spectrum}.
Hence for $m_c$-almost every $x\in D^c$, the unstable manifold $\cW^u(x)$
is contained in $X_{1-i}$.
Dynamical coherence gives a foliation $\cW^{uc}$ in the box $B$
(which integrates $E^u \oplus E^c$);
let $D^{uc}$ be the leaf that contains $D^c$,
with an induced  Riemannian volume measure $m_{uc}$.
It follows from the absolute continuity of the $\cW^u$ foliation that
$$
m_{uc} (X_i \cap D^{uc}) = 0.
$$

Since the set $Y_i  = X_i \cap \mathit{CB}^+$ is $\cW^s$-saturated and $m_{uc} (Y_i \cap D^{uc}) = 0$,
absolute continuity gives $m (Y_i \cap B) = 0$.
It follows that $m(Y_i) = 0$;
indeed if the invariant set $Y_i$ 
had positive measure then, by Theorem~\ref{t.Brin},
it would have a positive measure intersection with
every set of nonempty interior, for example the box $B$.
Recalling that $m(\mathit{CB}^+) > 1 - \eps$,
we get that $m (X_i) < \eps$.
This means that either $m(X) < \eps$ or $m(X) > 1-\eps$,
as we wanted to prove.
%
\end{proof}

\subsection{The $G_\delta$ Argument}\label{ss.G delta}

We now explain how Proposition~\ref{p.near ergodic} implies Theorem~\ref{t.main}.

Given $f \in \Diff^1_\omega(M)$ and a continuous function $\varphi: M \to \R$,
define functions:
$$
\varphi_{f,n} (x) =  \frac{1}{n} \sum_{i=0}^{n-1} \varphi(f^j(x)) \, , \qquad
\varphi_f (x) = \lim_{n \to +\infty} \varphi_{f,n} (x)  \quad \text{(defined a.e.).}
$$

For $\varphi \in C^0(M,\R)$, $a\in \R$, and $\eps>0$, let
$\cG (\varphi, a, \eps)$ be the set of $f$ such that
$m[\varphi_f \ge a] \ge 1-\eps$ or $m[\varphi_f \le a] \ge 1-\eps$.
(Here $[\varphi_f \ge a]$ is a shorthand for the set of $x\in M$
where $\varphi_f(x)$ exists and is greater than or equal to $a$.)

\begin{lemma}\label{l.G delta}
$\cG(\varphi, a, \eps)$ is a $G_\delta$ subset of $\Diff^1_\omega(M)$.
\end{lemma}

\begin{proof}
Define
$$
\cF(\varphi, a, \alpha) = \big\{ f ; \; m [\varphi_f \ge a] \ge \alpha \big\} .
$$
So we have
$$
\cG (\varphi, a, \eps) = \cF(\varphi, a, 1-\eps) \cup \cF(-\varphi, -a, 1-\eps),
$$
We are going to prove that $\cF(\varphi, a, \alpha)$ is a $G_\delta$.
Since the finite union of $G_\delta$'s is a $G_\delta$, \footnote{Proof:
$\bigcap A_n \cup \bigcap B_n = \bigcap (A_1 \cap \cdots \cap A_n) \cup (B_1 \cap \cdots \cap B_n)$.}
the lemma will follow.

Let $\varphi$, $a$, $\alpha$ be fixed.
Given $b< a$, $\beta < \alpha$, and $n_0$, $n_1\in \N$
with $n_0 \le n_1$, let $\cU(b, \beta, n_0, n_1)$ be the set of $f$ such that
the set $\left[ \max_{n \in [n_0, n_1]} \varphi_{f,n} > b \right]$ has measure $>\beta$.
Then $\cU(b, \beta, n_0, n_1)$ is open.

We will check that:
\begin{equation}\label{e.g delta}
\cF(\varphi,a,\alpha) =
\bigcap_{b<a} \bigcap_{\beta<\alpha} \bigcap_{n_0} \bigcup_{n_1>n_0} \cU(b, \beta, n_0, n_1) \, ,
\end{equation}
where $b$ and $\beta$ are take rational values.
First, we have:
$$
[\varphi_f \ge a] = [\limsup \varphi_{f,n} \ge a] =
\bigcap_{b<a} \bigcap_{n_0} \bigcup_{n \ge n_0} [\varphi_{f,n} > b] \quad \text{mod $0$.}
$$
Then we have the following equivalences:
\begin{align*}
m[\varphi_f \ge a] \ge \alpha
\ &\Longleftrightarrow \
\forall b<a, \ \forall n_0, \ m\left(\bigcup_{n \ge n_0} [\varphi_{f,n} > b] \right) \ge  \alpha
\\ &\Longleftrightarrow \
\forall b<a, \ \forall n_0, \ \forall \beta<\alpha, \ \exists n_1 > n_0 \text{ s.t.} \
\ m\left(\bigcup_{n=n_0}^{n_1} [\varphi_{f,n} > b] \right) \ge \alpha \, .
\end{align*}
This proves \eqref{e.g delta}, and hence that $\cF(\varphi,a,\alpha)$ is a $G_\delta$.
\end{proof}

\begin{proof}[Proof of Theorem~\ref{t.main}]
First, we claim that if $f$ is a $\eps$-nearly ergodic map, 
then $f \in \cG(\varphi, a, \eps)$ for any
$\varphi\in C^0(M,\R)$ and $a\in \R$.
Indeed, let $X$ be the (invariant) set of points $x\in M$ where $\limsup \varphi_{f,n}(x) \ge a$.
This is a bi~essentially saturated set, because
it is $\cW^s$-saturated and it coincides mod $0$ with
the $\cW^u$-saturated set $[\limsup \varphi_{f^{-1},n} \ge a]$.
Since $f$ is $\eps$-nearly ergodic,
$m(X)$ is either less than $\eps$ or greater than $1-\eps$.
So either $m[\varphi_f > a]$ or $m[\varphi_f \le a]$ is greater than $1-\eps$,
showing that $f \in \cG(\varphi,a,\eps)$.

It follows from Proposition~\ref{p.near ergodic} that the sets
$\cG (\varphi, a, \eps)$ are dense in $\PH^1_\omega(M)$,
while  Lemma~\ref{l.G delta} says they are $G_\delta$.
Thus to complete the proof of the theorem,
we need only to see that the set of ergodic diffeomorphisms is precisely
$$
\bigcap_{\varphi, a, \eps} \cG (\varphi, a, \eps) \, ,
$$
where $\varphi$ varies on a dense subset $\cD$ of $C^0(M,\R)$, and $a$ and $\eps$ take rational values.
Indeed, if $f$ is not ergodic then we can find $\varphi\in \cD$, $a<b$ and $0< \eps < 1/2$ such that
$[\varphi_f<a]$ and $[\varphi_f>b]$ both have measure greater than $\eps$.
Then $f$ cannot belong to $\cG (\varphi, a, \eps) \cap \cG (\varphi, b, \eps)$.
\end{proof}

\section{A Proof of Theorem~\ref{t.BDP}} \label{s.elliptic}

The following is the symplectic version of  Ma\~n\'e's
Ergodic Closing Lemma~\cite{Mane ergodic}, proved by Arnaud~\cite{Ar}.
If $f\in \Diff^1_\omega(M)$
and $x \in M$, we say that $x$ is \emph{$f$-closable} if for every $\eps>0$
there exists a $\eps$-perturbation $\tilde f \in \Diff^1_\omega(M)$
such that $x$ is periodic for $\tilde f$ and moreover
$d(\tilde f^i x, f^i x) < \eps$ for every $i$ between $0$ and
the $\tilde f$-period of $x$.

\begin{otherthm}[\cite {Ar}] \label{t.ECL}
For every $f \in \Diff_\omega^1(M)$, $m$-almost every point is $f$-closable.
\end{otherthm}

The first step to obtain Theorem~\ref{t.BDP} is to find  an ``almost elliptic'' periodic point,
that is a periodic point whose center eigenvalues are close to the unit circle:

\begin{lemma}\label{l.almost elliptic}
Let $f \in \PH_\omega^1(M)$ have unbreakable center.
Then for every $\eps>0$ there exists an $\eps$-perturbation $\tilde f$ and
a periodic point $x$ of period $p$ for $\tilde f$ such that
all eigenvalues $\mu_i$ of
$D\tilde f^p | E^c(x)$ satisfy $\left| \log |\mu_i| \right| \le \epsilon p$.
\end{lemma}

\begin{proof}
Let $f$ and $\eps$ be given.
Write $D^c f= Df | E^c$.
Since the eigenvalues of a symplectic map are symmetric,
to prove the lemma it suffices to find an $\eps$-perturbation $\tilde f$
with a periodic point $x$ of period $p$ such that:
\begin{equation} \label {e.bla}
\|D^c \tilde f^{m p} (x)\|<e^{\epsilon mp} \text { for some } m \geq 1.
\end{equation}

Due to Theorem~\ref{t.zero center}, we can assume that $\lambda^c(f)=0$.
Therefore there exists $k$ such that
$\frac{1}{k} \int_M \log \| D^c f^k \| \, dm < \epsilon$.
Hence, for all $x$ in a set of positive measure,
\begin{equation} \label {bla2}
\lim_{m \to \infty} \frac{1}{km}
\sum_{i=0}^{m-1} \log \|D^c f^k(f^{ik}(x))\| < \epsilon ,
\end{equation}
By Theorem~\ref{t.ECL}, we can take an $f$-closable point $x$ such that \eqref{bla2} holds.
If $x$ is periodic then \eqref{bla2} follows with $\tilde f=f$.
Otherwise, let $f_j \in \PH_\omega^1(M)$ be a sequence converging to
$f$ in the $C^1$ topology such that $x$ is periodic (of period $p_j$)
for $f_j$.  Then $p_j \to \infty$.  Let $m_j= \lfloor p_j/k \rfloor$.
We estimate:
$$
\frac{1}{p_j}\log \| D^c f^{p_j}_j(x) \| \le
\frac {1} {k m_j}
\sum_{i=0}^{m_j-1} \log \|D^c f_j^k(f_j^{ik}(x)) \| + \frac {1} {m_j}
\log \|D^c f_j\|_\infty
$$
As $j \to \infty$, the right hand side converges to the left hand side of \eqref{bla2}.
Thus the result follows with $\tilde f=f_j$ for $j$
sufficiently large.
\end{proof}

Next we see how the eigenvalues can be adjusted:

\begin{lemma}\label{l.nicolas}
Let $\eps>0$.  Let $A_1$, \ldots, $A_n$ be symplectic matrices
and let $2d$ be the number
of eigenvalues $\mu_i$ of $A_n \cdots A_1$ (counted with multiplicity) such
that $\frac {1} {n} \log |\mu_i| \le \epsilon$.  Then there exist symplectic matrices
$B_1$, \ldots, $B_n$ such that $\|B_i-\Id\|  \le e^\epsilon-1$ and
$A_n B_n \cdots A_1 B_1$ has exactly $2d$ eigenvalues
(counted with multiplicity) in the unit circle.
\end{lemma}

\begin{proof}
Assume the matrices have size $2N \times 2N$.
Let $\{p_1, \dots, p_N, q_1, \dots, q_N\}$ be the canonical symplectic and orthonormal basis of $\R^{2N}$.

Write $A^i=A_i \cdots A_1$.  Let $\lambda_1> \cdots > \lambda_t$ be the
Lyapunov exponents of $A^n$ and let $\{0\} =
F_0 \subsetneq F_1 \subsetneq \cdots \subsetneq F_t = \R^{2N}$ be the
Lyapunov filtration of $A^n$, that is, $A^n(F_i) = F_i$ and the action
of $A^n$ on $F_i/F_{i-1}$ has eigenvalues of modulus $e^{\lambda_i}$.
Let $r(i)$ be the dimension of $F_i$.  Let $m \geq 0$ be maximal with
$\lambda_m>0$, and let $0 \leq u \leq m$ be maximal with
$\lambda_u \geq \epsilon$.  Notice that $\dim F_{t-u}/F_u=2d$.

Let $F_i^k=A^k(F_i)$, $0 \leq k \leq n-1$.
There exists symplectic orthogonal matrices $C_0$, \ldots, $C_{n-1}$
such that if $i \leq m$ then $C_k F^k_i$ is spanned by
$p_1$, \ldots, $p_{r(i)}$.  It follows that if $i>m$ then
$C_k F^k_i$ is spanned by $p_1$, \dots, $p_d$, $q_d$, \dots, $q_{r(i)-d}$.

Let us consider a symplectic matrix $\Lambda$ such
that $\Lambda p_k=p_k$ and $\Lambda q_k=q_k$, unless
$r(i-1) < k \leq r(i)$ for some $u<i \leq m$,
in which case we let $\Lambda p_k=e^{-\lambda_i/n} p_k$,
$\Lambda q_k=e^{\lambda_i/n} q_k$.

Let $B_k=C_k^{-1} \Lambda C_k$.  Then
$T=A_n B_n \cdots A_1 B_1$ preserves the spaces
$F_i$, $0 \leq i \leq t$, and $T$ acts on $F_i/F_{i-1}$ as $A^n$,
unless $u<i \leq m$ or $k-m< i \leq k-u$,
in which case $T$ acts as $e^{-\lambda_i} A^n$.  It follows that the
action of $T$ on $F_{k-u}/F_u$ has only zero Lyapunov exponents.
The result follows.
\end{proof}

\begin{proof}[Proof of Theorem~\ref{t.BDP}]
By Lemma~\ref{l.almost elliptic} we can perturb $f$ and
create an ``almost elliptic'' periodic point.
Lemma~\ref{l.nicolas} says that the derivatives along this orbit can be perturbed
to become completely elliptic.
Using Lemma~\ref{l.pasting} we can realize this by a further  perturbation of
the diffeomorphism.
\end{proof}

The argument above could have been carried out by appealing to
the easier cocycle version of \cite{Bochi sympl} obtained in \cite{Bochi Viana}:
see the Appendix of this paper.


\section{Proof of Theorem~\ref{t.nonunifbw}}\label{s.amie}
We adopt as much as possible notation that is consistent with
the notation in \cite{Burns Wilk}, as the proof of Theorem~\ref{t.nonunifbw}
has many parallels with the proof of Theorem 3.1 there.  A few
statements are also adapted bearing in mind the needs of the proof of
Theorem \ref {t.nonunifjimmy} given in the appendix.

\subsection{Density}

If $\nu$ is a measure and
$A$ and $B$ are $\nu$-measurable sets with $\nu(B) >0$,
we define the {\em density of $A$ in $B$} by:
 $$
 \nu(A:B) = \frac{\nu(A\cap B)}{\nu(B)}.
 $$
A point $x\in M$ is a {\em Lebesgue density point}
of a measurable set $X\subseteq M$ if
$$\lim_{r\to 0} m(X: B_r(x)) = 1.$$
The Lebesgue Density Theorem implies that if $X$ is a measurable
set and $\widehat X$ is the set of Lebesgue density points of $X$,
then $m(X \vartriangle \widehat X) = 0$.

Lebesgue density points can be characterized using nested sequences
of measurable sets.
We say that a sequence of measurable sets $Y_n$ {\em nests} at point $x$
if $Y_0\supset Y_1 \supset Y_2 \supset \cdots \supset \{x\}$, and
$$\bigcap_n Y_n = \{x\}.$$
A nested
sequence of measurable sets $Y_n$ is {\em regular} if
there exists $\delta> 0$ such that, for all $n\geq 0$,
we have $m(Y_n)>0$, and
$$m(Y_{n+1})\geq \delta m(Y_n).$$

Two nested sequences of sets $Y_n$ and $Z_n$ are {\em internested} if there exists
a $k\geq 1$ such that, for all $n\geq 0$, we have
$$Y_{n+k} \subseteq Z_n,\qquad \hbox{ and } \qquad Z_{n+k}\subseteq Y_{n}.$$
The following lemma is a straightforward consequence of the definitions.

\begin{lemma}[\cite{Burns Wilk}, Lemma 2.1]\label{l=compreg}
Let $Y_n$ and $Z_n$ be internested sequences of measurable sets, with
$Y_n$ regular.
Then $Z_n$ is also regular.  If the sets $Y_n$ have positive measure, then
so do the $Z_n$, and, for any measurable set $X$,
$$\lim_{n\to\infty} m(X:Y_n) = 1 \quad\Longleftrightarrow \quad \lim_{n\to\infty} m(X:Z_n) = 1.$$
\end{lemma}


\subsection{Foliations and Absolute Continuity}

Let $\cF$ be a foliation with smooth $d$-dimensional leaves.
An open set $U\subset M$ is a  {\em foliation box for $\cF$} if it is the image  of
$\R^{n-d}\times \R^{d}$ under a homeomorphism that sends each
vertical $\R^d$-slice into a leaf of $\cF$. The images of the
vertical $\R^d$-slices are called {\em local leaves of $\cF$ in $U$}.

A {\em smooth transversal} to $\cF$ in $U$ is a smooth codimension-$d$
disk in $U$ that intersects each local leaf in $U$ exactly once and whose
tangent bundle is uniformly transverse to $T\cF$.
If $\tau_1$ and $\tau_2$ are two smooth transversals to $\cF$ in $U$,
we have the {\em holonomy map} $h_{\cF}: \tau_1 \to \tau_2$,
which takes a point in $\tau_1$ to the intersection  of its local leaf
in $U$ with $\tau_2$.

If $S \subseteq M$ is a smooth submanifold, we denote by
$m_S$ the volume of the
induced Riemannian metric on $S$. If $\cF$ is a foliation with smooth leaves,
and $A$ is contained in a single leaf of $\cF$ and is measurable in that
leaf, then we denote by $m_\cF(A)$ the induced Riemannian
volume of $A$ in that leaf.

A foliation $\cF$ with smooth leaves is {\em transversely
absolutely continuous with bounded
Jacobians} if  for every angle $\alpha\in(0,\pi/2]$,
there exists $C\geq 1$ and $R_0>0$ such that, for every foliation
box $U$ of diameter less than $R_0$, any two smooth
transversals $\tau_1, \tau_2$ to
$\cF$ in $U$ of angle at least $\alpha$ with $\cF$, and any
$m_{\tau_1}$--measurable set $A$ contained in $\tau_1$:
\begin{equation}\label{e=acholon}
 C^{-1} m_{\tau_1}(A) \leq m_{\tau_2}(h_\cF(A))\leq C m_{\tau_1}(A).
\end{equation}

The foliations $\cW^s$ and $\cW^u$ for a partially hyperbolic diffeomorphism
are transversely
absolutely continuous with bounded
Jacobians (see \cite{AV flavors}).

Let $\cF$ be an absolutely continuous foliation
and let $U$ be a foliation box for $\cF$.
Let  $\tau$ be a smooth transversal to $\cF$ in $U$.
Let $Y\subseteq U$ be a measurable set.
For a point $q \in \tau$, we define the
{\em fiber $Y(q)$ of $Y$ over $q$} to be the
intersection of $Y$ with the local leaf of $\cF$ in $U$ containing $q$.
The {\em base $\tau_Y$ of $Y$} is
the set of all $q\in \tau$ such that the fiber $Y(q)$ is
$m_\cF$-measurable and  $m_\cF(Y(q))>0$.
The absolute continuity of $\cF$ implies that $\tau_Y$ is
 $m_{\tau}$-measurable.
We say  ``$Y$ fibers over $Z$'' to indicate that $Z=\tau_Y$.

If, for some $c\geq 1$, the inequalities
$$
c^{-1} \leq \frac{m_\cF(Y(q))}{m_\cF(Y(q'))} \leq c
$$
hold for all $q,q'\in \tau_Y$, then we say that {\em $Y$ has $c$-uniform fibers}. A sequence of measurable sets $Y_n$ contained in $U$ has
{\em $c$-uniform fibers} if each set in the sequence has $c$-uniform fibers,
with $c$ independent of $n$.

\begin{prop}[\cite{Burns Wilk}, \S~2.3] \label{p=unifreg}\label{p=compmeas}\label{p=compmeas2}
Suppose that the foliation $\cF$ is absolutely continuous
with bounded Jacobians. Let $U$ be a foliation box for $\cF$,
and let  $\tau$ be a smooth transversal to $\cF$ in $U$.
Let $Y_n$ and $Z_n$ be sequences of measurable subsets of $U$
with $c$-uniform fibers.
\begin{enumerate}
\item Suppose that there exists $\delta>0$ such that:
\begin{enumerate}
\item for all $n \geq 0$,
$$
m_{\tau}(\tau_{Y_{n+1}}) \geq \delta m_{\tau}(\tau_{Y_n});
$$

\item for all $n \geq 0$,
there are points $z\in\tau_{Y_{n+1}}, z'\in \tau_{Y_n}$ with
$$
 m_{\cF}(Y_{n+1}(z)) \geq \delta m_{\cF}(Y_n(z')).
$$
\end{enumerate}
Then $Y_n$ is regular.
\item Suppose that $\tau_{Y_n} = \tau_{Z_n}$, for all $n$ and that
$Y_n$ and $Z_n$ both nest at a common point $x$.
Then, for any set $X\subseteq U$ that is  essentially $\cF$-saturated at $x$, we have the equivalence:
$$
\lim_{n\to\infty} m(X:Y_
n) = 1\,\Longleftrightarrow\,
\lim_{n\to\infty}m(X:Z_n) = 1.
$$
\item For every  measurable
set $X$ that is  $\cF$-saturated at $x$, we have the equivalence:
$$
\lim_{n\to\infty} m(X:Y_n) = 1\,\Longleftrightarrow\,
\lim_{n\to\infty}m_\tau(\tau_X:\tau_{Y_n}) = 1.
$$
\end{enumerate}
\end{prop}

\subsection{Construction of an Adapted Metric}

We begin with some notation.
Again fix the diffeomorphism $f\colon M\to M$.
For $x\in M$ and $j\in\Z$ we denote by $x_j$ the $j$th iterate $f^j(x)$.
If $\alpha$, $\beta$ are positive functions defined on the forward orbit
$\cO^+(p) = \{p_j  ; \;  j\geq 0\}$
of some $p\in M$,
we write $\alpha \prec \beta$ if there exists a positive constant $\lambda
< 1$ such that
for all $y\in \cO^+(p)$:
$$\frac{\alpha(y)}{\beta(y)} <\lambda.$$
Notice that if $\alpha, \beta$ happen to extend from $\cO^+(p)$ to
continuous functions on $M$ satisfying
the pointwise inequality $\alpha < \beta$, then compactness of $M$ implies
that $\alpha\prec \beta$.

If $\alpha$ is a positive function, and $j\geq 1$ is
an integer, let
$$
\alpha_j(x) = \alpha(x)\alpha(x_1)\cdots\alpha(x_{j-1}),
$$
and
$$
\alpha_{-j}(x) =
\alpha(x_{-j})^{-1}\alpha(x_{-j+1})^{-1}\cdots\alpha(x_{-1})^{-1}.
$$
We set $\alpha_0(x) = 1$.
Observe that $\alpha_j$ is a
multiplicative cocycle; in particular, we have
$\alpha_{-j}(x)^{-1}= \alpha_j(x_{-j})$.

\begin{lemma}

Let $f\colon M\to M$ be $C^{1}$ and partially hyperbolic, and let $p\in CB^+$.
Then there exist functions $B$, $\nu$, $\hat\nu$, $\gamma$, $\hat\gamma\colon\cO^+(p)\to
\R_+$, bounded from below, and
a Riemannian metric $\|\mathord{\cdot}\|_\star$ defined on $T_{\cO^+(p)}M$
with the following properties:

\begin{enumerate}
\item $\nu \prec \gamma\hat\gamma \le 1$ and $\hat\nu\prec 1$; 
\item for $y$ in $\cO^+(p)$,
$$
\|D_y f|_{E^s}\|_\star \prec \nu(y) \prec \gamma(y) \prec \m_\star(D_y
f|_{E^c}) \leq \|D_y f|_{E^c}\|_\star \prec \hat \gamma(y)^{-1} \prec \hat
\nu(y)^{-1} \prec \m_\star(D_y f|_{E^u});
$$
\item $\limsup_{j\to\infty} B(p_j)^{1/j} = 1$;
\item for all $v \in T_{p_j}M$ and $j\geq 0$:
\begin{equation}\label{e=metriccompare}
                \| v \|  \leq  \| v \|_\star \leq B(p_j) \| v \|.
\end{equation}
\end{enumerate}
\end{lemma}

\begin{proof}

Let $\| \mathord{\cdot} \|_1$ be a Riemannian metric on $T_{\cO^+(p)} M$ that
coincides with $\| \mathord{\cdot} \|$ on each of the three spaces $E^s(p_j)$,
$E^c(p_j)$, and $E^u(p_j)$, but with respect to which those three spaces are
orthogonal.  Notice that there exists a constant $C \geq 1$ such
that $C^{-1} \|v\| \leq \|v\|_1 \leq C \|v\|$ for every $v \in T_{\cO^+(p)}
M$.

Let us define another Riemannian metric $\| \mathord{\cdot} \|_2$ on $T_{\cO^+(p)} M$ as
follows.  Let $i_k$ be as in the definition of forward center bunching.
With respect to the inner product induced by $\| \mathord{\cdot} \|_1$,
the linear map $D_{p_{i_k}} f^{i_{k+1}-i_k}$
can be written in a unique way as $O_k P^{i_{k+1}-i_k}_k$
where $P_k:T_{p_{i_k}} M \to T_{p_{i_k}} M$ is selfadjoint positive and
$O_k:T_{p_{i_k}} M \to T_{p_{i_{k+1}}} M$ is an isometry: 
indeed $P_k^{2(i_{k+1}-i_k)} = 
(D_{p_{i_k}} f^{i_{k+1}-i_k})^* \cdot D_{p_{i_k}} f^{i_{k+1}-i_k}$.
Notice that $P_k$ preserves the spaces $E^s(p_{i_k})$, $E^c(p_{i_k})$, and $E^u(p_{i_k})$.
Define $\| \mathord{\cdot} \|_2$ on $T_{\cO^+(p)} M$ so that
for $i_k \leq j<i_{k+1}$, 
the map 
$$
D_{p_{i_k}} f^{j-i_k} \cdot P_k^{-(j-i_k)}: (T_{p_{i_k}} M, \| \mathord{\cdot} \|_1) \to 
(T_{p_j} M, \| \mathord{\cdot} \|_2)
$$
is an isometry.
By construction, for each $i_k \leq j<i_{k+1}$, and for each subbundle
$F=E^u$, $E^c$, $E^s$, we have
$\|D_{p_j} f|_F\|_2^{i_{k+1}-i_k}=\|D_{p_{i_k}}f^{i_{k+1}-i_k}|_F\|$ and
$\m_2(D_{p_j} f|_F)^{i_{k+1}-i_k}=\m(D_{p_{i_k}}f^{i_{k+1}-i_k}|_F)$.
The definitions of partial hyperbolicity and forward center bunching then
immediately imply that there exists $\rho<1$ such that
\begin{gather*}
\| D_y f|_{E^s}\|_2 \leq \rho^2 \m_2( D_y f|_{E^c}) \min \{1,
\|D_y f|_{E^c}\|_2^{-1}\}, \quad \text{and} \\
\max \{1,\| D_y f|_{E^c}\|_2\} \leq \rho^2 \m_2( D_y f|_{E^u} )
\end{gather*}
for every $y \in \cO^+(p)$.

Notice that $\| \mathord{\cdot} \|_2$ and $\| \mathord{\cdot} \|_1$ coincide for $T_{p_{i_k}} M$
for each $k$.  Let $C_j \geq 1$ be minimal such that
$C^{-1}_j \|v\| \leq \|v\|_2 \leq C_j \|v\|$ for every $v \in T_{p_j} M$.
The condition $i_{k+1}/i_k \to 1$ then implies that $C_j^{1/j} \to 1$.  Let
$D_j \geq C_j$ be a sequence such that 
$D_j \leq D_{j+1} \leq \rho^{-1} D_j$ and $D_j^{1/j} \to 1$.  
For every $j \geq 0$, let
$\| \mathord{\cdot} \|_\star= D_j \| \mathord{\cdot} \|_2$ over $T_{p_j} M$,
and $B(p_j)=D_j C_j$.
For $y \in \cO^+(p)$, we define 
$\nu(y)=\rho^{-1/4} \|D_y f|_{E^s}\|_\star$, 
$\gamma(y)=\rho^{1/4} \m_\star(D_y f|_{E^c})$,  
$\hat \gamma(y)= (\rho^{1/4} \|D_y f|_{E^c}\|_\star)^{-1}$,  and
$\hat \nu(y) = (\rho^{-1/4} \m_\star(D_y f |_{E^u}))^{-1}$.
All desired properties are straightforward to check.
\end{proof}

We next show that the sets $\mathit{CB}^+$ and $\mathit{CB}^-$ 
are respectively $\cW^s$ and $\cW^u$-saturated.

\begin{proof}[Proof of Proposition~\ref{p.CB is sat}]



We will use the previous lemma and its proof.
Let $p \in \mathit{CB}^+$ and $q \in \cW^s(p)$, and let
$p_j=f^j(p)$, $q_j=f^j(q)$.  Choose invertible
linear maps $A_j:T_{p_j} M \to T_{q_j} M$, bounded and with bounded
inverses with respect to $\| \mathord{\cdot} \|$,
that preserve the bundles $E^s$ and $E^c$,
and such that $A^{-1}_{j+1} D_{q_j} f A_j$ is
exponentially close to $D_{p_j}
f$ (here we use that $Df$ and the bundles $E^s$ and
$E^c$ are H\"older).  This implies that
$A^{-1}_{j+1} D_{q_j} f  A_j$
is also exponentially close to $D_{p_j} f$ with respect to
$\| \mathord{\cdot} \|_\star$.  It follows that there exists $\delta>0$ such that
$$\|A_{j+1}^{-1} D_{q_j}
f A_j|_{E^s}\|_\star \cdot \m_\star(A_{i_{j+1}}^{-1} D_{q_j}
f A_j|_{E^c})^{-1}  \cdot \|A_{j+1}^{-1} D_{q_j}
f A_j|_{E^c}\|_\star \leq 1-\delta$$ for every sufficiently large $j$.
Let $i_k$ be as in the definition of forward center bunching
for $p$.  By the proof of the previous lemma,
$\| \mathord{\cdot} \|_\star$ and $\| \mathord{\cdot} \|$ coincide 
modulo a constant factor over
$E^s(p_{i_k})$ and $E^c(p_{i_k})$, so
$$\|A_{i_{k+1}}^{-1} D_{q_{i_k}}
f^{i_{k+1}-i_k} A_{i_k}|_{E^s}\| \cdot
\m(A_{i_{k+1}}^{-1} D_{q_{i_k}}
f^{i_{k+1}-i_k} A_{i_k}|_{E^c})^{-1} \cdot
\|A_{i_{k+1}}^{-1} D_{q_{i_k}}
f^{i_{k+1}-i_k} A_{i_k}|_{E^c}\| \leq (1-\delta)^{i_{k+1}-i_k}$$ for
every $k$ sufficiently large.
Since the maps $A_j$, $A_j^{-1}$ are uniformly bounded with respect to
$\| \mathord{\cdot} \|$, and preserve $E^s$ and $E^c$,
we see that there exists $n \geq 1$
such that every $k \geq 0$,
$$
\|D_{q_{i_{nk}}}
f^{i_{nk+n}-i_{nk}}|_{E^s}\|^{-1} \geq
(1+\delta)^{i_{nk+n}-i_{nk}} \,
\frac {\|D_{q_{i_{nk}}} f^{i_{nk+n}-i_{nk}}|_{E^c}\|}
{\m(D_{q_{i_{nk}}} f^{i_{nk+n}-i_{nk}}|_{E^c})} \, .
$$
Since $i_{nk+n}/i_{nk} \to 1$, we conclude that $q \in \mathit{CB}^+$.

It follows by symmetry that $\mathit{CB}^-$ is $\cW^u$-saturated.
\end{proof}

Fix $R_0>0$ less than injectivity radius of $M$ in the original $\|\mathord{\cdot} \|$ metric.
Let $\exp$ denote the exponential map for the $\|\mathord{\cdot} \|$ metric.
Consider the neighborhood $\cN_{R_0}$ of $\cO^+(p)$ defined by
$$\cN_{R_0} = \bigsqcup_{j\geq 0} B(p_j,R_0),$$
where $B(x,r)$ denotes the ball of radius $r$ centered at $x$ in the original Riemannian metric.
The manifold $\cN_{R_0}$ carries the restriction of the original Riemannian metric.
When we speak of volumes and induced Riemannian volumes on submanifolds of $\cN_{R_0}$,
it will always be with respect to this metric.

We introduce two other metrics on $\cN_{R_0}$ that will be used in this proof, one of them
closely related to (and comparable to) the original metric.  The first metric is the
{\em flat $\|\mathord{\cdot} \|$ metric}, denoted $\|\mathord{\cdot} \|_\flat$,
which is the (locally) flat Riemannian metric defined as follows.  For $x\in B(p_j, R_0)$,
and $v,w\in T_{x}M$, we set
$$\langle v,w \rangle_\flat =  \langle D_x\exp_{p_j}^{-1}(v), D_x\exp_{p_j}^{-1}(w)\rangle_{p_j},$$
where we make the standard identification  $T_{u}(T_pM) \simeq (T_pM)$.
In the distance $d_\flat$ induced by this metric, we have, for $q,q'\in B(p_j, R_0)$,
$d_\flat(q, q') = \|\exp_{p_j}^{-1}(q)-\exp_{p_j}^{-1}(q')\|_{p_j}$.
Compactness of $M$ implies that $\|\mathord{\cdot} \|$ and $\|\mathord{\cdot} \|_\flat$ are comparable.

Next we extend the $\|\mathord{\cdot}\|_\star$ metric, which is defined on $T_{\cO^+(p)}M$, to a
flat metric  $\|\mathord{\cdot}\|_\star$ on $\cN_{R_0}$ using the same type of construction.
For $x\in B(p_j, R_0)$, and $v,w\in T_{p_j}M$, we set
$$
\langle v,w \rangle_\star =  \langle D_x\exp_{p_j}^{-1}(v), D_x\exp_{p_j}^{-1}(w) \rangle_{\star,p_j}.
$$
Denote by $d_\star$ the  distance induced by this Riemannian metric,
so that, for $q,q'\in B(p_j,R_0)$, we  have
$d_\star(q, q') = \|\exp_{p_j}^{-1}(q)-\exp_{p_j}^{-1}(q')\|_\star$.

The results of this section imply that on $B(p_j,R_0)$, we have
$K d_\flat \leq d_\star \leq B(p_j) d_\flat$.
Thus on any component $B(p_j,R_0)$, the $\star$ and $\flat$ metrics
are uniformly comparable.  The degree of comparability decays subexponentially
as $j\to \infty$. For $q\in \cN_{R_0}$ and $r>0$ sufficiently small,
we denote by $B_\star(q,r)$ the $d_\star$-ball of radius $r$ centered at $q$.

By uniformly rescaling the $\|\mathord{\cdot}\|_\flat$ and $\|\mathord{\cdot}\|_\star$ metrics by
the same constant factor,
we may assume that for some $R > 1$, and any $x \in M$,
the Riemannian balls $B_\star(x,R)$ and $B(x,R)$ are
contained in foliation boxes for both $\cW^s$ and $\cW^u$.
We assume both $R$ and $R_0$ are large enough so that
all the objects considered in the sequel
are small compared with $R$ and $R_0$.

\subsection{Fake Invariant Foliations}

Let ${\bf r}:\cO^+(p)\to \R_+$ be any positive function such that
$\sup_{j\geq 0} {\bf r}(p_j) \leq R_0$.
Denote by $\cN_{{\bf r}}$ the following neighborhood of $\cO^+(p)$:
$$ \cN_{{\bf r}} = \bigsqcup_{j\geq 0} B(p_j,{\bf r}(j)).$$

If $\cF$ is a foliation of $\cN_{\bf r}$, and  $B_\star(x,r)$ is contained
in a foliation box $U$ for $\cF$, then we will denote by
$\cF_\star(x,r)$ the intersection of the local leaf of $\cF$
at $x$ with $B_\star(x,r)$. Notice that
$\cF_\star(x,r) \subseteq \cF(x,K^{-1}r)$.

\begin{prop}\label{p=fakenu}
For every $\epsilon > 0$, there exist functions ${\bf r}$, ${\bf R}\colon\cO^+(p)\to \R$
satisfying:
$${\bf r}\prec {\bf R},\quad \sup_{y\in\cO^+(p)} {\bf R}(y)<R_0, \quad
\inf_{j\geq 0} \frac{{\bf r}(p_{j+1})}{{\bf r}(p_j)} > e^{-\epsilon},\quad\hbox{and}\quad \inf_{j\geq 0}\frac{{\bf R}(p_{j+1})}{{\bf R}(p_j)} > e^{-\epsilon},
$$
and such that the neighborhood $\cN_{\bf R}$ is foliated by
foliations $\hW^u$, $\hW^s$, $\hW^c$, $\hW^{cu}$ and $\hW^{cs}$ with the
following properties, for each $\beta \in \{u,s, c, cu, cs\}$:
\begin{enumerate}

\item {\bf Almost tangency to invariant distributions:}
For each $q\in \cN_{\bf R}$, the leaf $\hW^\beta(q)$ is $C^1$ and
the tangent space $T_q\hW^\beta(q)$ lies in a cone of $\| \mathord{\cdot}\|_\star$-angle
$\eps$ about $E^\beta(q)$ and also within a cone of $\| \mathord{\cdot}\|$-angle
$\eps$ about $E^\beta(q)$.

\item {\bf Local invariance:} for each $y\in \cO^+(p)$ and $q\in B(y,{\bf r}(y))$,
$$f(\hW^\beta(q,{\bf r}(y))) \subset
\hW^\beta(q_1),\,\hbox{ and }f^{-1}(\hW^\beta(q_1,{\bf r}(y_1))) \subset
\hW^\beta(q).$$

\item {\bf Exponential growth bounds at local scales:} The following hold for all $n\geq 0$ and $y\in\cO^+(p)$.
\begin{enumerate}
\item Suppose that $q_j\in B_\star(y_j,{\bf r}(y_j))$ for $0 \leq j \leq n-1$.

If
$q' \in \hW^{s}(q,{\bf r}(y))$,
then $q_n' \in \hW^{s}(q_n,{\bf r}(y_n))$, and
$$
d_\star(q_n,q_n') \leq \nu_n(y) d_\star(q,q').
$$
If $q_j' \in \hW^{cs}(q_j,{\bf r}(y_j))$ for $0 \leq j \leq n-1$, then
$q_n' \in \hW^{cs}(q_n)$, and
 $$d_\star(q_n,q_n') \leq \hat\gamma_n(y)^{-1} d_\star(q,q').$$

\item Suppose that $q_{-j}\in B_\star(y_{n-j},{\bf r})$ for $0 \leq j \leq n-1$.

If $q' \in \hW^{u}(q,{\bf r}(y_n))$,
then $q_{-n}' \in \hW^{u}(q_{-n},{\bf r}(y))$, and
$$d_\star(q_{-n},q_{-n}') \leq \hat\nu_{n}(y) d_\star(q,q').$$
If $q_{-j}' \in \hW^{cu}(q_{-j},{\bf r}(y_{n-j}))$ for $0 \leq j \leq n-1$, then
 $q_{-n}' \in \hW^{cu}(q_{-n})$, and
$$d_\star(q_{-n},q_{-n}') \leq \gamma_{n}(y)^{-1} d_\star(q,q').$$
\end{enumerate}

\item {\bf Coherence:}  $\hW^s$ and $\hW^c$ subfoliate  $\hW^{cs}$;
 $\hW^u$ and $\hW^c$ subfoliate  $\hW^{cu}$.

\item {\bf Uniqueness:}
$\hW^s(p)= \cW^s(p,{\bf R}(p))$, and $\hW^u(p)= \cW^u(p,{\bf R}(p))$.

\item {\bf Regularity:} The foliations $\hW^u$, $\hW^s$, $\hW^c$, $\hW^{cu}$ and $\hW^{cs}$ and their tangent distributions are uniformly H\"older continuous, in both the $d_\star$ and $d$ metrics.

\item{\bf Regularity of the strong foliation inside weak leaves:}
the restriction of the foliation $\hW^s$ to each leaf of $\hW^{cs}$ is absolutely continuous
with bounded jacobians, and the restriction of the foliation $\hW^u$ to each leaf of $\hW^{cu}$ is absolutely continuous with bounded jacobians (with respect to the standard Riemannian metric and volume).

There exists a constant $L>0$ such that for any $p'\in \cW^{s}(p)$,
the $\hW^s$-holonomy map $h^s\colon \hW^c(p)\to \hW^c(p')$ is
$L$-bi-Lipschitz at $p$.  That is, for all $q\in \hW^c(p)$, we have:
$$L^{-1} d_\star(p,q) \leq d_\star(h^s(p), h^s(q)) \leq L d_\star(p,q).$$
\end{enumerate}
\end{prop}

\begin{proof}[Proof of Proposition~\ref{p=fakenu}.]
The proof follows closely the proof of Proposition 3.1 in \cite{Burns Wilk}.
Our construction will be performed in two steps. In the first, we construct
foliations of each tangent space $T_yM$, $y\in\cO(p)$. In the second step, we use the
exponential map $\exp_y$ to project these foliations from a neighborhood of the origin in $T_yM$ to
a neighborhood of $y$.

The argument diverges slightly from the argument in \cite{Burns Wilk} in that, because
we are in the nonuniform setting, the H\"older continuity of $Df$ (in this case Lipschitz continuity)
must be used explicitly in the construction of the fake foliations.

\medskip

\noindent{\bf Step 1.} We extend the $\|\mathord{\cdot}\|_\star$-metric on $T_{\cO^+(p)}M$ to a metric on $T_{\cO(p)}M$,
which we also denote by $\|\mathord{\cdot}\|_\star$, by setting it equal to $\|\mathord{\cdot} \|$ on $\bigsqcup_{j\leq 0} T_{p_j}M$.
Extend the function $B$ to $\cO(p)$ by setting $B(p_j) = 1$ for $j\leq 0$.

Fix a constant $R_1<R_0$ such that the diameter of $f(B(x,R_1))$ is less than $R_0$,
for all $x\in M$.
For $v\in T_{p_j}M$, $\|v\|\leq R_1$, let $\tilde f_j(v) = \exp_{p_{j+1}}^{-1}\circ f \circ \exp_{p_j}(v)$.
Then $D_0\tilde f_j = D_{p_j}f$ and so, since $f$ is $C^2$:
\begin{equation}\label{e=c2}
\tilde{f}_j(v) = D_{p_j}f(v) + O(\|v\|^2), \qquad\hbox{and} \qquad \|D_{v}\tilde{f}_j - D_{p_j}f\| \leq	O(\|v\|),
\end{equation}
uniformly in $j$. Fix a family of smooth bump functions $\{\beta_r:\R\to [0,1], r>0\}$ with the
properties that $|\beta_r'|\leq 3r^{-1}$, $\beta_r(t) = 1$ for $|t|\leq r^2$, and $\beta_r(t) = 0$
for $|t|\geq 4r^2$.

For $r\in (0,R_1)$, define $F_{j,r}:T_{p_j}M\to T_{p_{j+1}}(M) $ by:
$$F_{j,r}(v) = \beta_r(\|v\|^2) \tilde{f}_j(v) + (1 - \beta_r(\|v\|^2)) D_{p_j} f(v).
$$
One easily checks using (\ref{e=c2}) that $d_{C^1}(F_{j,r}, D_{p_j} f) \leq O(r)$, uniformly in $j$
and that $F_{j,r}(v) = \tilde{f}(v)$ for $\|v\|\leq r$, and $F_{j,r}(v) = D_{p_j}f(v)$ for $\|v\|\geq 2r$.

For any function ${\bf r}\colon \cO(p)\to \R_+$
with $\sup_{y\in \cO(p)} {\bf r}(y) < R_1$, define a $C^2$ bundle map
$F_{\bf r}\colon T_{\cO(p)}M\to T_{\cO(p)}M$, by setting
$F_{\bf r} = F_{{\bf{r}(p_j)},j}$ on $T_{p_j}M$.
Then $F_{\bf r}$ covers
$f:\cO(p)\to \cO(p)$, and has the following properties:
\begin{enumerate}
\item $F_{\bf r}$ coincides with $\exp_{p_{j+1}}^{-1}\circ f \circ \exp_{p_j}$
on the $\|\mathord{\cdot} \|$-ball of radius ${\bf r}(p_j)$ in $T_{p_j}M$
and with $D_{p_j}f$ outside the ball of
radius $2{\bf r}(p_j)$
\item The $C^1$ distance from $F_{\bf r}$ to $Df$ on approaches $0$ uniformly as
$|{\bf r}|_\infty\to 0$. In particular, on $T_{p_j}M$, we have $d_{C^1}(F_{\bf r}, D_{p_j} f) \leq
O({\bf r}(p_j))$.
\end{enumerate}
When measured in the $\|\mathord{\cdot}\|_\star$-metric, the $C^1$ distance between two functions on $T_{p_j}(M)$
is multiplied by $B(p_j)$.  It follows that:
\begin{enumerate}
\item[3)] On $T_{p_j}M$, we have $d_{C^1}(F_{\bf r}, D_{p_j}f)_\star \leq O(B(p_j) {\bf r}(p_j)),$ uniformly in $j$;
that is, $d_{C^1}(F_{\bf r}, Df)_\star \leq O(B{\bf r})$.
\end{enumerate}

Let $\epsilon>0$ be given. Fix $\epsilon_1 < \epsilon$
such that
\begin{equation}\label{e=eps1choice}
e^{-2\epsilon} > \sup_{y\in\cO^+(p)} \max\left\{ \nu(y),\, \hat\nu(y),\, \frac{\nu(y)}{\gamma(y)},\,  \frac{\nu(y)}{\gamma(y)},\, \frac{\hat\nu(y)}{\hat\gamma(y)}, \, \frac{\nu(y)}{\gamma\hat\gamma(y)} \right\}.
\end{equation}
For $c>0$,
define a function ${\bf R}_{c}\colon \cO(p)\to \R_+$ by
$${\bf R}_{c}(p_j) =
\begin{cases}
c, & \hbox{if } j\leq 0\\
ce^{-j\epsilon'}, & \hbox{if } j>0.
\end{cases}
$$
Since $\limsup_{j\to\infty} B(p_j){\bf R}_{c}(p_j) = 0$,
the argument above shows that $d_{C^1}(F_{{\bf r}_{c}}, Df)_\star$ tends to
$0$ uniformly as $c\to 0$. This also implies that the $C^1$ distance in the original
Riemannian metric $\|\mathord{\cdot} \|$ tends to $0$ uniformly in $c$.

Since $Df$ is uniformly partially hyperbolic in both metrics, we may choose $c$ sufficiently small so that
$F = F_{{\bf R}_{c}}$ is uniformly partially hyperbolic in both $\| \mathord{\cdot}\|$ and $\|\mathord{\cdot} \|_\star$ metrics.
Note that $F$ is $C^{1+Lip}$ in the $\|\mathord{\cdot}\|_\star$ metric,
with Lipschitz constant of $DF$, $DF^{-1}$ on
$T_{p_j}M$ bounded by a constant $L(p_j)>0$
with the property $\limsup_{j\to\infty} L(p_j)^{1/j} = 1$.
$F$ is  uniformly $C^2$ in $\|\mathord{\cdot} \|$ metric.
Note also that $F$ is $C^{1-\epsilon}$ in the $\|\mathord{\cdot}\|_\star$ metric,
with H\"older constant of $DF$, $DF$ on
$T_{p_j}M$ bounded by a constant.
If $c$ is small enough, the equivalents  of inequalities  (3)--(6)
will hold for $TF$.

If $c$ is sufficiently small,
standard graph transform arguments give stable, unstable, center-stable,
and center-unstable foliations for
$F_r$ inside
each $T_pM$. These foliations are uniquely determined by
the extension $F$ and the requirement that their leaves be graphs of bounded functions.
We obtain a center foliation by intersecting the leaves of the
center-unstable and center-stable foliations.  While $TM$ is not compact, all of the relevant estimates for $F$ are uniform, and it is this,
not compactness, that counts.

The uniqueness of the stable and unstable foliations imply, via a standard
argument (see, e.g.\ \cite{HPS}, Theorem 6.1~(e)), that
the stable foliation subfoliates the center-stable, and the
unstable subfoliates the center-unstable.

We now discuss the regularity properties of these foliations
of $TM$. Our foliations of $TM$ have been constructed as
the unique fixed points of graph transform maps.
We can apply the above results to the $F$-invariant splittings of $TTM$ as
the sum of the stable and  center-unstable bundles for $F$
and as the sum of the center-stable and unstable bundles for $F$.
It follows from the pointwise H\"older section theorem (see \cite{Pugh Shub Wilk}, Theorem A)
 that both the center-unstable and unstable bundles and
the corresponding foliations are H\"older continuous as long as
$F$ is $C^{1+\delta}$ for some $\delta >0$.  Since $F$ is
$C^{1+\delta}$ uniformly in both $\|\mathord{\cdot} \|$ and $\|\mathord{\cdot}\|_\star$ metrics, it follows that the
bundles are uniformly H\"older in both metrics.

We obtain the H\"older continuity of the
center-stable and stable bundles for $F_r$ and the corresponding foliations
by thinking of the same splittings as $F_r^{-1}$-invariant.
H\"older regularity of the center bundle and foliation is obtained by
noticing the the center is the intersection of the center-stable and center-unstable.

The absolute continuity with bounded Jacobians of
the unstable foliation inside of the center-unstable foliation is a standard
result, using only partial hyperbolicity, dynamical coherence, the fact that $F$ is
uniformly $C^{1+\delta}$, and the H\"older continuity of the bundles in the partially hyperbolic
splitting.  Similarly, the stable foliation for $F$ is absolutely continuous with bounded
Jacobians when considered as a subfoliation of the center-stable.

The Lipschitz continuity of the stable inside of the center-stable is proved
in Lemma~\ref{l=lipschitz} below.

\medskip

\noindent{\bf Step 2.} \
We now have foliations of $T_yM$, for
each $y\in \cO(p)$.  We obtain the foliations $\hW^u, \hW^c, \hW^s, \hW^{cu}$, and $\hW^{cs}$ by applying the exponential map $\exp_y$ to the corresponding foliations of $T_yM$ inside the ball around the origin of radius ${\bf R}_c(y)$.

If $c$ is sufficiently small, then
the distribution $E^\beta_q$
lies within the angular $\eps/2$-cone about the parallel translate of
$E^\beta_y$, for every $\beta \in \{u,s, c, cu, cs\}$, $y\in \cO^+(p)$, and all $q\in B(y,{\bf R}_c(y))$.
Combining this fact with the preceding discussion, we obtain that property~(1) holds if $c$ is
sufficiently small.

Property~(2) --- local invariance --- follows from invariance under $F_r$ of the foliations of $TM$ and the fact that $\exp_{f(y)}(F(y,v)) = f(\exp_y(y,v))$ provided $\|v\| \leq {\bf R}_c(y)$.

Having chosen $c$, we now choose $c_1$ small enough so that, for all $y\in \cO^{+}(p)$,
$f(B(y,2{\bf R}_{c_1}(y))) \subset B(f(y),{\bf R}_c(y))$ and
$f^{-1}(B(y,2{\bf R}_{c_1}(y))) \subset B(f^{-1}(y),{\bf R}_c(y))$, and so that, for all
$q\in f(B(y,{\bf R}_{c_1}(y)))$,
\begin{align*}
q' \in \hW_p^s(q,{\bf R}_{c_1}(y))) \quad &\Longrightarrow \quad d_\star(f(q),f(q')) \leq \nu(y)\, d_\star(q,q'),\\
q' \in \hW_p^u(q,{\bf R}_{c_1}(y))) \quad &\Longrightarrow \quad d_\star(f^{-1}(q),f^{-1}(q')) \leq \hat\nu(f^{-1}(y))\, d_\star(q,q'),\\
q' \in \hW_p^{cs}(q,{\bf R}_{c_1}(y))) \quad &\Longrightarrow \quad  d_\star(f(q),f(q')) \leq \hat\gamma(y)^{-1} \,d_\star(q,q'),\\
q' \in \hW^{cu}(q,{\bf R}_{c_1}(y)))\quad &\Longrightarrow \quad d_\star(f^{-1}(q),f^{-1}(q')) \leq \gamma(f^{-1}(y))^{-1}\, d_\star(q,q').
\end{align*}
We set ${\bf R} = {\bf R}_c$ and ${\bf r} = {\bf R}_{c_1}$.

Property (3) --- exponential growth bounds at local scales --- is now
proved by a simple inductive argument.

Properties (4)--(7) --- coherence, uniqueness, regularity and regularity of the strong foliation inside weak leaves --- follow immediately from the corresponding properties of the foliations of $TM$ discussed above,
except for the Lipschitz continuity statement, which we now prove:

\begin{lemma}\label{l=lipschitz} The $\hW^s$ holonomy maps between $\hW^c$ manifolds are Lipschitz at $p$.
\end{lemma}

\begin{proof}[Proof of Lemma~\ref{l=lipschitz}]

Fix a function $\rho$ satisfying ${\nu}{\gamma}^{-1} \prec \rho \prec \min\{1,\hat\gamma\}$, and such that
$\kappa< e^{-\epsilon_1}$, where
$$\kappa = \sup_{y\in \cO^+(p)}\max\{({\nu}{\gamma}^{-1}\rho^{-1})(y), \, (\rho\hat\gamma^{-1})(y)\}.$$
Note that this is possible because (\ref{e=eps1choice}) implies that $\sup_{y\in \cO^+(p)} \max\{{\nu}{\gamma}^{-1}(y),  {\nu}\gamma^{-1}\hat\gamma^{-1}(y) \} < e^{-2\epsilon_1}$.
Fix a constant $\lambda \in (\kappa,e^{-\epsilon_1})$.
Observe that
\begin{equation}\label{e=kappa}
\sup_{y\in\cO^+(p)} (\nu\gamma^{-1}\hat\gamma^{-1}(y)) < \kappa,
\end{equation}
since $\rho \prec \min\{1,\hat\gamma\}$.

Since $B(p_{j})^{1/j}\to 1$ as $j\to\infty$, there exists a constant $C>0$ such that
$$\sup_{j\geq 0} B(p_j)(\kappa\lambda^{-1})^j < C.
$$

Let $\theta$ be the H\"older exponent of the partially hyperbolic splitting, in the $\star$-metric,
and let $H$ be the $\theta$-H\"older norm.
Choose $\delta >0 $ and $N>0$ such that:
\begin{itemize}
\item  $H\left((\delta\nu_j(p))^\theta) + (\rho_n\hat\gamma_j(p)^{-1})^\theta \right) < 1/2 - \epsilon$
for all $n\geq N$ and $j=0,\ldots, n$,
\item $\rho_N(p) < \delta/3$, and
\item $1- \lambda - 4\delta C \sup_{y\in \cO^+(p)}\gamma(y) > 0$.
\end{itemize}
Finally, choose $K>2\delta$ satisfying:
$$ K >  \sup_{j\in\N}\frac{8\delta B(p_{j+1})(\kappa\lambda^{-1})^{j+1}}{1- \lambda - 4\delta B(p_{j+1})\kappa^{j+1} \gamma(p_{j+1})},
$$
and let $L=3 + 2K$.

We will show that for each $p'\in\hW^{s}_\star(p,\delta/3)$, and for every
$q\in \hW^c_{loc}(p)$:
$$d_\star(p,q) \leq \rho_N(p) \implies  L^{-1}d_\star(p,q)\leq d_\star(h^s(p), h^s(q)) \leq L^{-1}d_\star(p,q),$$
where $h^s\colon \hW^c_{loc}(p)\to \hW^c(p')$ is the $\hW^s$-holonomy map.
We prove the righthand inequality; the proof of the lefthand inequality is given by switching the roles of
$p$ and $p'$.

Let $p'\in \hW^{s}_\star(p,\delta/3)$ be given, and let $q\in \hW^c_\star(p,\rho_N(p))$.
Denote by $q'$ the image of $q$ under $h^s$ (by definition $h^s(p)=p'$).
Fix $n\geq N$ such that $\rho_{n}(p) \leq  d_\star(p, q) < \rho_{n-1}(p)$.
Note that $d_\star(p,p')<\delta/3<\delta$ and $d_\star(q,q')<\delta$, by the triangle inequality.

\begin{lemma}\label{l=iteratebound} For $j=0,\ldots, n$, we have $\{p_j, p_j', q_j, q_j'\}\subset \cN_{\bf r}$.
Moreover:
\begin{enumerate}
\item $\rho_n(p)\gamma_j(p) \leq d_\star(p_j, q_j) \leq \rho_{n-1}(p) \hat\gamma_j(p)^{-1}$, and
\item $\max\{d_\star(p_j,p_j'), \, d_\star(q_j, q_j')\} < \delta \nu_j(p).$
\end{enumerate}
\end{lemma}
\begin{proof}
The proof is a simple inductive argument using
Part 3 of Proposition~\ref{p=fakenu}.
\end{proof}

We will work in $\|\mathord{\cdot} \|$-exponential coordinates in $\cN_{\bf r}$.
For $j\in \N$ and $x\in B_\star((p_j),{\bf r})$, denote by $\tilde{x}$ the point $\exp_{p_j}^{-1}(x)$.
Note that $\tilde{p_j} = 0$.
Let $v_j = \tilde{q_j} - \tilde{p_j}$, let $v_j' =  \tilde{q_j'} - \tilde{p_j'}$, and let $w_j = v_j'-v_j$.
Lemma~\ref{l=iteratebound} implies that for $j=0,\ldots,n$, we have
$(\rho_n\gamma_j)(p) \leq \|v_j\|_\star  \leq (\rho_{n-1}\hat\gamma_j^{-1})(p)$
and  $\|w_j\|_\star \leq d_\star(p_j,p_j') + d_\star(q_j,q_j') \leq 2\delta \nu_j(p)$.
Let $\pi^c_j:T_{p_j}M\to E^c_{p_j}$ be the linear projection with kernel $\left(E^u\oplus E^s\right)_{p_j}$,
and let $\pi^{us}_j:T_{p_j}M\to  \left(E^u\oplus E^s\right)_{p_j}$ be the linear projection with kernel $ E^c_{p_j}$.

The vectors $v_j$ and $v_j'$ lie in uniform cones about $E^c_{p_j}$ with respect to the splitting
$T_{p_j}M = E^c_{p_j}\oplus \left(E^u\oplus E^s\right)_{p_j}$:
\begin{lemma} For $j=0,\ldots n$, we have $\|\pi^{us}_j(v_j)\|_\star \leq \frac{1}{2}\| v_j \|_\star$,
$\|\pi^{us}_j(v_j')\|_\star \leq \frac{1}{2}\| v_j' \|_\star$,  $\|v_j\|_\star \leq \frac{3}{2} \|\pi^c_j(v_j)\|_\star$
and $\|v_j'\|_\star \leq \frac{3}{2} \|\pi^c_j(v_j')\|_\star$.
\end{lemma}

\begin{proof} $T_{p_j}\hW^c$ and $T_{q_j}\hW^c$ both lie in the $\epsilon$-cone about
$E^c(p_j)$, and the tangent distribution to $\hW^c$ is H\"older continuous.
Hence $\tan\angle_\star(T_{p_j}\hW^c, T_{p_j'}\hW^c) \leq H d_\star(p_j,p_j')^\theta \leq   H(\delta\nu_i(p))^\theta$, and
$\tan\angle_\star(T_{q_j}\hW^c, T_{q_j'}\hW^c) \leq H d_\star(q_j,q_j')^\theta \leq  H(\delta\nu_i(p))^\theta $.
Furthermore $\tan\angle_\star(T_{p_j}\hW^c, T_{q_j}\hW^c) \leq H d_\star(p_j,q_j)^\theta  \leq H (\rho_n(p)\hat\gamma_j(p)^{-1})^\theta$.
This implies that $$\tan\angle_\star(T_{p_j'}\hW^c, T_{q_j'}\hW^c) \leq H\left((\delta\nu_j(p))^\theta) + (\rho_n(p)\hat\gamma_j(p)^{-1})^\theta \right) < 1/2 - \epsilon$$ for $j=0,\ldots, n$,
by our choice of $\delta$.
\end{proof}

Since the points $\{p_j,p_j',q_j, q_j'\}$ all lie in $\cN_{\bf r}$,
in which $F$ coincides with $\tilde{f} = \exp^{-1}\circ f\circ \exp$,
we have that $\tilde{x}_j = F^j(\tilde x)$, for $x\in \{p,p',q,q'\}$.
The Mean Value Theorem implies that $v_{j-1} = \int_{0}^1 D_{\tilde{p_j} + tv_j}F (v_j)\, dt$ and
$v_{j-1}' = \int_{0}^1 D_{\tilde{p_j'} + tv_j'}F(v_j')\, dt$; subtracting
these expressions, we obtain:
$$w_{j-1} = \int_{0}^1\left( D_{\tilde{p_j} + tv_j}F^{-1}(v_j) - D_{\tilde{p_j'} + tv_j'}F^{-1}(v_j') \right)\, dt.$$
and
$$\pi^c_j(w_{j-1}) = \int_{0}^1\pi^c_j\left( D_{\tilde{p_j} + tv_j}F^{-1}(v_j) - D_{\tilde{p_j'} + tv_j'}F^{-1}(v_j') \right)\, dt.$$
Then $\|\pi^c_{j-1}(w_{j-1})\|_\star \leq \mathrm{(I)} + \mathrm{(II)}$ where
\begin{align*}
\mathrm{(I)} &= \int_{0}^1\left\| \pi^c_{j-1} D_{\tilde{p_j} + tv_j}F^{-1}\left(v_j-v_j'\right)\right\|_\star \, dt,
\\
\mathrm{(II)} &= \int_{0}^1\left\| \left(\pi^c_{j-1} D_{\tilde{p_j} + tv_j}F^{-1} - \pi^c_{j-1}D_{\tilde{p_j'} + tv_j'}F^{-1}\right)(v_j') \right\|_\star \, dt.
\end{align*}
We have
$$
\mathrm{(II)} \leq \int_{0}^1 B(p_j)\|v_j-v_j'\|_\star \|v_j'\|_\star \, dt
\leq B(p_j)\|w_j\|_\star\|v_j'\|_\star ,
$$
since $DF^{-1}$ is Lipschitz with norm $B(p_j)$ on $T_{p_j}M$.

We next estimate the expression $\mathrm{(I)}$.
Since $D_{p_j}F^{-1} = D_{p_j}f^{-1}$, which sends the splitting $\left(E^u\oplus E^c \oplus E^s\right)_{p_j}$ to
$\left(E^u\oplus E^c \oplus E^s\right)_{p_{j-1}}$ and has norm on $E^c$ bounded by $\gamma(p_j)^{-1}$, we have that:
$$
\int_{0}^1\left\| \pi^c_{j-1} D_{\tilde{p_j}}F^{-1}\left(w_j\right)\right\|_\star \, dt
\leq \gamma(p_j)^{-1}\|\pi^c_{j} w_{j}\|_\star.
$$
Hence
\begin{align*}
\mathrm{(I)}  &= \int_{0}^1\left\| \pi^c_{j-1} D_{\tilde{p_j} + tv_j}F^{-1}\left(w_j\right)\right\|_\star \, dt \\
&\leq \int_{0}^1\left\| \pi^c_{j-1} \left(D_{\tilde{p_j}}F^{-1} - D_{\tilde{p_j} + tv_j}F^{-1}\right) \left(w_j\right)\right\|_\star \, dt + \int_{0}^1\left\| \pi^c_{j-1} D_{\tilde{p_j}}F^{-1}\left(w_j\right)\right\|_\star \, dt  \\
&\leq \int_{0}^1\left\| \pi^c_{j-1} \left(D_{\tilde{p_j}}F^{-1} - D_{\tilde{p_j} + tv_j}F^{-1}\right) \left(w_j\right)\right\|_\star \, dt +
\gamma(p_j)^{-1}\|\pi^c_{j} w_{j}\|_\star\\
&\leq B(p_j)\|v_j\|_\star \|w_j\|_\star +
\gamma(p_j)^{-1}\|\pi^c_{j} w_{j}\|_\star,
\end{align*}
again using the Lipschitz continuity of $DF^{-1}$.  We conclude that
\begin{equation}\label{e=iterwj}
\begin{aligned}
\|\pi^c_{j-1}(w_{j-1})\|_\star &\leq
\gamma(p_j)^{-1}\|\pi^c_{j} w_{j}\|_\star + B(p_j)\left( \|v_j\|_\star \|w_j\|_\star +\|w_j\|_\star\|v_j'\|_\star\right) \\
&\leq \gamma(p_j)^{-1}\|\pi^c_{j} w_{j}\|_\star + 2\delta B(p_j)\nu_j(p)\left( \|v_j\|_\star + \|v_j'\|_\star\right),
\end{aligned}
\end{equation}
using the bound $\|w_j\|_\star  \leq 2\delta \nu_j(p)$.

\begin{claim} For $j = 0,\ldots,n$, we have $\|\pi^c_j w_j\|_\star \leq K \lambda^j \gamma_j(p) \|v_0\|_\star $
and $  \|v_j'\|_\star \leq (3 + 2K\lambda^j) \|v_j\|_\star$
\end{claim}

\begin{proof}
We prove it by backwards induction on $n$.
The base case is $j=n$. Observe that:
$$
\|\pi^c_n w_n\|_\star \leq  \|w_n\|_\star \leq 2\delta \nu_n(p) = 2\delta \frac{\nu_n(p)}{(\rho\gamma)_n(p)} \gamma_n(p)\rho_n(p)  < 2 \delta \lambda^n \gamma_n(p) \|v_0\|_\star < K\lambda^n \gamma_n(p) \|v_0\|_\star.
$$
Since $\|w_n\|\ast \leq 2\delta \nu_n(p) \leq 2\delta \lambda^n (\rho\gamma)_n(p) \leq  2\delta \lambda^n \|v_n\|_\star$,
we also obtain that
$$
\|v_n'\|_\star \leq  \|v_n\|_\star +  \|v_n- v_n'\|_\star
=  \|v_n\|_\star + \|w_n\|_\star \leq  \|v_n\|_\star(1 +  2\delta \lambda^n) \leq  \|v_n\|_\star(3 +  2K\lambda^n).
$$

Now suppose that the claim holds for some $(j+1)\leq n$. Then, by (\ref{e=iterwj}):
\begin{align*}
\|\pi^c_{j}(w_{j})\|_\star &\leq  \gamma(p_{j+1})^{-1}\|\pi^c_{j+1} w_{j+1}\|_\star + 2\delta B(p_{j+1})\nu_{j+1}(p)\left( \|v_{j+1}\|_\star + \|v_{j+1}'\|_\star\right)\\
&\leq \gamma(p_{j+1})^{-1} K \lambda^{j+1} \gamma_{j+1}(p) \|v_0\|_\star + 2\delta B(p_{j+1})\nu_{j+1}(p)(4 + 2K\lambda^{j+1} )\|v_{j+1}\|_\star\\
&\leq K \lambda^{j+1} \gamma_{j}(p) \|v_0\|_\star + 2\delta B(p_{j+1})\|v_{0}\|_\star(\nu\hat\gamma^{-1})_{j+1}(p)(4 + 2K\lambda^{j+1} )\\
&\leq K \eta \lambda^{j} \gamma_{j}(p)\|v_0\|_\star,
\end{align*}
where
\begin{align*}
\eta &= \lambda + 8\delta B(p_{j+1}) \frac{(\nu\gamma^{-1}\hat\gamma^{-1})_{j+1}(p)}{K\lambda^{j+1}} \gamma(p_{j+1})
+4\delta B(p_{j+1})(\nu\gamma^{-1}\hat\gamma^{-1})_{j+1}(p) \gamma(p_{j+1})\\
&\leq \lambda + 8\delta B(p_{j+1}) \frac{\kappa^{j+1}}{K\lambda^{j+1}} \gamma(p_{j+1}) +4\delta B(p_{j+1})\kappa^{j+1} \gamma(p_{j+1}),
\end{align*}
by (\ref{e=kappa}).
Then $\eta <1$, since
$$ K >  \frac{8\delta B(p_{j+1})(\kappa\lambda^{-1})^{j+1}}{1- \lambda - 4\delta B(p_{j+1})\kappa^{j+1} \gamma(p_{j+1})}.
$$
This implies that $\|\pi^c_{j}(w_{j})\|_\star \leq K \lambda^{j} \gamma_{j}(p)\|v_0\|_\star$, completing the
inductive step for the first assertion of the claim.

Finally, to prove the inductive step for the second part of the claim, we have:
\begin{align*}
\|v_j'\|_\star &\leq  \|v_j\|_\star +  \|v_j- v_j'\|_\star\\
&\leq  \|v_j\|_\star +  \|\pi^c_j (v_j- v_j')\|_\star +\|\pi^{us}_j (v_j- v_j')\|_\star\\
&\leq  \|v_j\|_\star +  \|\pi^c_j (w_j)\|_\star +\|\pi^{us}_j (v_j)\|_\star +\|\pi^{us}_j(v_j')\|_\star\\
&\leq  \|v_j\|_\star + K \lambda^j \gamma_j(p) \|v_0\|_\star+ .5\|v_j\|_\star + .5\|v_j'\|_\star\\
&\leq  \|v_j\|_\star + K \lambda^j \|v_j\|_\star + .5\|v_j\|_\star + .5 \|v_j'\|_\star;\\
\end{align*}
Solving for $\|v_j'\|_\star$, we obtain that $\|v_j'\|_\star \leq (3 + 2K \lambda^j) \|v_j\|_\star$, as desired.
\end{proof}

The claim finishes the proof of Lemma~\ref{l=lipschitz}; setting, $j=0$ we see
that
$$d_\star(h^s(p),h^s(q)) = \|v_0'\|_\star \leq (3+2 K)\|v_0\|_\star = L d_\star(p,q).$$
\end{proof}
\end{proof}

Given this proposition, the proof now proceeds  as the proof of Theorem~5.1 in \cite{Burns Wilk},
with a few modifications, which we will describe in the sequel.

\subsection{Distortion Estimates in Thin Neighborhoods} 

Fix $p\in M$ satisfying the bunching hypotheses of Theorem~\ref{t.nonunifbw}.
Henceforth the entire analysis will take place in a neighborhood of the forward orbit of $p$.

We choose $\epsilon>0$:
\begin{itemize}
\item much smaller than $\pi/2$, which is the $\star$-angle between the bundles
of the partially hyperbolic splitting over $\cO^+(p)$.
\item small enough that
\begin{equation}\label{e=epschoice}
e^{-\epsilon} > \sup_{y\in\cO^+(p)} \max\left\{ \nu(y),\, \hat\nu(y),\, \frac{\nu(y)}{\gamma(y)},\, \frac{\hat\nu(y)}{\hat\gamma(y)}, \, \frac{\nu(y)}{\gamma\hat\gamma(y)} \right\}.
\end{equation}
\end{itemize}

Let ${\bf r}$, ${\bf R}\colon \cO^+(p)\to \R_+$ and foliations $\hW^u$, $\hW^s$, $\hW^c$, $\hW^{cu}$ and $\hW^{cs}$ be given by Proposition~\ref{p=fakenu}, using this value of $\epsilon$.  By uniformly
rescaling the $\|\mathord{\cdot}\|_\star$ metric on $\cN_{\bf R}$, we may assume that
$$\inf_{y\in\cO^+(p)} {\bf r}(y) \gg 1.$$ We may also assume that if $x,y\in B_\star(p_j,{\bf r})$, then
$\hW^{cs}(x) \cap \hW^u(y)$,
$\hW^{cs}(x) \cap \cW^u_{loc}(y)$, $\hW^{cu}(x) \cap \hW^s(y)$
and  $\hW^{cu}(x) \cap \cW^s_{loc}(y)$
are single points. We denote by $\hm_a$ the measure $m_{\hW^a}$ induced by the volume form $\|\mathord{\cdot}\|$.

We next choose functions $\sigma, \tau\colon \cO^+(p)\to \R_+$ satisfying
\begin{equation}\label{e=sigmatauchoice}
\sigma \prec \min\{1, \hat\gamma\},\quad\hbox{ and }
\qquad \nu \prec \tau  \prec \sigma\gamma,
\end{equation}
and such that $\kappa = \sup_{y\in\cO^+(p)}\sigma\hat\gamma^{-1}(y) < e^{-\epsilon}$ (this
is possible because of (\ref{e=epschoice})).
Note that these inequalities  also imply that
$$\tau\hat\nu \prec \sigma\gamma\hat\nu \prec \sigma\gamma\hat\gamma \le \sigma.$$
For the rest of the proof, except where we indicate otherwise, cocycles will
be evaluated at the point $p$.  We will also drop the dependence on
$p$ from the notation; thus, if $\alpha$ is a cocycle, then
$\alpha_n(p)$ will be abbreviated to $\alpha_n$.

Using these functions and the fake foliations,
we next define a sequence of thin neighborhoods $T_{n}$ of $\cW^s_\star(p,1)$.
We first define a neighborhood $S_{n}$ in $\hW^{cs}(p)$ by:
$$S_{n} = \bigcup_{x\in \cW^s_\star(p,1)} \hW^c_\star(x,\sigma_n),$$
and then define the neighborhood $T_n$ by:
\begin{equation}\label{e=Tn}
T_{n} = f^{-n}\left(\bigcup_{z\in f^n(S_{n})}\hW^u_\star(z,\tau_n)\cup\cW^u_\star(z,\tau_n)\right).
\end{equation}

\begin{lemma}[c.f.\ \cite{Burns Wilk}, Lemma 4.3] \label{l=kappa}
The set $T_n$ is well-defined. There exists $C>0$ and $0<\kappa<1$ such that, for every $n\geq 0$,
$$f^j(T_{n}) \subset B_\star(p_j, C\kappa^j),$$
for $j=0,\ldots, n$.
\end{lemma}

\begin{proof} Suppose first that $x\in \cW^s_\star(p,1)$
and  $y\in \hW^c(x,\sigma_n)$.  By part 3(a) of Proposition~\ref{p=fakenu}, we then have
$$y_j\in \hW^c_\star(x_j, \hat\gamma_{j}^{-1}\sigma_{n})
\subset \hW^c_\star(x_j, 1) \subset B_\star(p_j,2),$$
for $0\leq j\leq n$.  In fact, since
$\sigma \prec \min\{\hat\gamma, 1\}$,
the quantity $\hat\gamma_{j}^{-1}\sigma_{n} <
\hat\gamma_{j}^{-1}\sigma_{j} \leq \kappa^j$ is exponentially small in $j$,
as is the $\star$-diameter of $f^j(\cW^s(p,1))$. This implies
that for some $C>0$ and for every $n\geq 0$,
\begin{equation}\label{e=sliceest}
f^j(S_{n}) \subset B_\star(p_j, C\kappa^j),
\quad\text{for $j=0,\ldots, n$.}
\end{equation}

For every $x\in S_{n}$, we have that $B_\star(x_n, \tau_n) \subset B_\star(p_n,{\bf r}(p_n))$, and
so the set $T_n$ is well-defined by \eqref{e=Tn}.
Proposition~\ref{p=fakenu} 
implies
that the leaves of $\hW^u_{p_j}$ and $\cW^u_{loc}$ are uniformly contracted
by $f^{-1}$ as long as they stay near the orbit of $p$.
Because $\kappa<e^{-\epsilon}$, the image of $f^n(T_n)$, for $n$ sufficiently
large, remains in the neighborhood  $\cN_{\bf r}$ of $\cO^+(p)$
in which the fake foliations are defined and the expansion and contraction estimates hold.

Combining these facts with~\eqref{e=sliceest}, we obtain the conclusion.
\end{proof}

\begin{lemma}[c.f.\ \cite{Burns Wilk}, Lemma 4.4]\label{l=constalpha}
Let $\alpha:M\to \R$ be a positive, uniformly H\"older continuous function.
Then there is a constant $C\geq 1$ such that,
for all $n\geq0$ and all $x, y\in T_{n}$,
$$C^{-1}\leq \frac{\alpha_n(y)}{\alpha_n(x)} \leq C.$$
\end{lemma}

\begin{proof} Since $d\leq K^{-1}d_\star$, Lemma~\ref{l=kappa} implies that the diameter of $f^j(T_n)$ remains exponentially small in the $d$ metric, for $j=0,\ldots,n$.  Since $f$ is $C^{1+\delta}$, the lemma follows from
the following elementary distortion estimate:

\begin{lemma}[\cite{Burns Wilk}, Lemma 4.1]\label{l=dist}
Let $\alpha:M\to \R$ be a positive H\"older continuous function, with exponent $\theta>0$.  Then there exists a constant $H>0$ such that the following holds, for all $p,q\in M$, $B>0$ and $n\geq 1$:
$$\sum_{i=0}^{n-1} d(p_i, q_i)^\theta \leq B\quad\Longrightarrow\quad e^{-{H}B} \leq \frac{\alpha_n(p)}{\alpha_n(q)} \leq e^{{H}B},$$
and
$$\sum_{i=1}^{n} d(p_{-i}, q_{-i})^\theta \leq B \quad\Longrightarrow\quad e^{-{H}B} \leq \frac{\alpha_{-n}(p)}{\alpha_{-n}(q)} \leq e^{{H}B}.$$
\end{lemma}
\end{proof}

\subsection{Juliennes}

The next step is to define juliennes. For each $x\in \cW^s_\star(p,1)$
one defines a sequence $\{\hJ^{cu}_n(x)\}_{n\geq 0}$ of center-unstable juliennes, which
lie in the fake center-unstable manifold $\hW^{cu}(x)$ and shrink exponentially
as $n \to \infty$ while becoming increasingly thin in the $\hW^u$-direction.

Define, for all $x\in \cW^s(p,1)$,$$
\hB^c_n(x) = \hW^c_\star(x,\sigma_n).
$$
Note that
$$S_n = \bigcup_{x \in \cW^s(p,1)} \hB^c_{n}(x).$$
For $y\in S_n$, we may then define two types of {\em unstable juliennes}:
$$
\hJ^u_n(y) = f^{-n}(\hW^u_\star(y_n,\tau_n))
$$
and
$$
J^u_n(y) = f^{-n}(\cW^u_\star(y_n,\tau_n)).
$$
Observe that for all $y\in S_n$, the sets
$\hJ^u_n(y)$ and $J^u_n(y)$ are contained in $T_n$.

For each $x\in  \cW^s(p,1)$ and $n\geq 0$, we then
define the {\em center-unstable julienne} centered at $x$ of order $n$:
$$
\hJ^{cu}_n(x) = \bigcup_{q \in \hB^c_n(x)} \hJ^u_n(q).
$$
Note that, by their construction, the sets
$\hJ^{cu}_n(x)$ are contained in $T_n$, for all $n\geq 0$ and
$x\in \cW^s(p,1)$.

The crucial properties of center unstable juliennes are summarized in
the next three Propositions.  We state them in a slightly
more general form than we will need for the proof of
Theorem~\ref{t.nonunifbw}; the more general formulation
will be used in the proof of Theorem~\ref{t.nonunifjimmy}.

\begin{prop}[c.f.\ \cite{Burns Wilk}, Proposition 5.3]\label{p=inclusion}
Let $x,x'\in \cW^s(p,1)$, and let
$h^s: \hW^{cu}(x) \to \hW^{cu}(x')$  be the holonomy map
induced by the stable foliation
$\cW^s$.
Then the sequences $h^s(\hJ^{cu}_n(x))$ and $\hJ^{cu}_{n}(x')$
are internested.
\end{prop}

\begin{prop}[c.f.\ \cite{Burns Wilk}, Proposition 5.4]\label{p=julmeas}  There exist $\delta > 0$ and
$c \geq 1$ such that, for all $x\in \cW^s(p,1)$, and all $q,q'\in S_n$,
the following hold, for all $n\geq 0$:
$$
c^{-1}\leq \frac{\hm_{u}(\hJ^u_n(q))}{\hm_{u}(\hJ^u_n(q'))}\leq c,
$$
$$
c^{-1}\leq \frac{m_{u}(J^u_n(q))}{m_{u}(J^u_n(q'))}\leq c,
$$
$${\hm_{u}(\hJ^{u}_{n+1}(q))}\geq \delta {\hm_{u}(\hJ^{u}_{n}(q))},$$
and
$${\hm_{cu}(\hJ^{cu}_{n+1}(x))}\geq \delta {\hm_{cu}(\hJ^{cu}_{n}(x))}.$$
\end{prop}


\begin{prop}[c.f.\ \cite{Burns Wilk}, Proposition 5.5]\label{p=sliceden}
Let $X$ be a measurable set that is both $\cW^s$-saturated and
essentially $\cW^u$-saturated at some point $x\in \cW^s(p)$.
Then $x$ is a Lebesgue density point of $X$ if and only if:
$$
\lim_{n\to\infty} \hm_{cu}(X:\hJ^{cu}_n(x)) = 1.
$$
\end{prop}

Assuming these propositions, we conclude the:

\begin{proof}[Proof of Theorem~\ref{t.nonunifbw}]
Let $X$ be a bi~essentially saturated set, and let
$X^s$ be an essential $\cW^s$-saturate of $X$.
Since $m(X \vartriangle X^s) = 0$,
the Lebesgue density points of $X$ are precisely the same as those of $X^s$.
Suppose that $x\in \cW^s(p,1)$ is a
Lebesgue density point of $X^s$.  Proposition~\ref{p=sliceden}
implies that $x$ is a $cu$-julienne density point of $X^s$.

To finish the proof, we show that every $x'\in \cW^s(p,1)$ is
a  $cu$-julienne density point of $X^s$.  Then
by Proposition~\ref{p=sliceden}, every $x'\in \cW^s(p,1)$ is
a Lebesgue density point of $X^s$, and so $\cW^s(p,1)\subset \hat X$.
Notice that if $p$ satisfies the hypotheses of Theorem~\ref{t.nonunifbw},
thes so does every $p'\in\cW^s(p)$.  Hence if $\cW^s(p)\cap \hat X\ne \emptyset$,
then $\cW^s(p)\subset \hat X$, completing the proof.

Let $h^s: \hW^{cu}(x) \to \hW^{cu}(x')$  be the holonomy map
induced by the stable foliation $\cW^s$.
The sequence $h^s(\hJ^{cu}_n(x))\subset \hW^{cu}(x')$ nests at $x'$.

Transverse absolute continuity of $h^s$ with bounded
Jacobians implies that
$$\lim_{n\to\infty} \hm_{cu}(X^s:\hJ^{cu}_n(x)) = 1\quad\Longleftrightarrow\quad
\lim_{n\to\infty} \hm_{cu}(h^s(X^s):h^s(\hJ^{cu}_n(x))) = 1.$$
Since $X^s$ is $s$-saturated, we then have:
$$\lim_{n\to\infty} \hm_{cu}(X^s:\hJ^{cu}_n(x)) = 1\quad\Longleftrightarrow\quad
\lim_{n\to\infty} \hm_{cu}(X^s:h^s(\hJ^{cu}_n(x))) = 1.$$
Since we are assuming that $x$ is a $cu$-julienne density point of $X^s$, we thus have
$$\lim_{n\to\infty} \hm_{cu}(X^s:h^s(\hJ^{cu}_n(x))) = 1.$$

Working inside of $\hW^{cu}(x')$, we will apply Lemma~\ref{l=compreg}
to the sequences $h^s(\hJ^{cu}_n(x))$ and
$\hJ^{cu}_n(x')$, which both nest at $x'$.
Proposition~\ref{p=inclusion} implies that these sequences are internested.
Proposition~\ref{p=julmeas} implies that $\hJ^{cu}_n(x')$ is regular
with respect to the induced
Riemannian measure $\hm_{cu}$ on $\hW^{cu}(x')$.
Lemma~\ref{l=compreg} now implies that
$$\lim_{n\to\infty} \hm_{cu}(X^s:h^s(\hJ^{cu}_n(x))) = 1\quad\Longleftrightarrow\quad\lim_{n\to\infty} \hm_{cu}(X^s:\hJ^{cu}_n(x')) = 1,$$
and so $x'$ is a $cu$-julienne density point of $X^s$. It follows from Proposition~\ref{p=sliceden} that $x'$ is a Lebesgue density point of $X^s$, and thus of $X$. 
\end{proof}

We extract from this proof a proposition that will be used
in the proof of Theorem~\ref{t.nonunifjimmy}:

\begin{prop}\label{p.holonomyinv} Let $Y\subset \hW^{cu}(p)$ be a measurable subset, let
$x\in \cW^s(p,1)$, and let $Y'$ be the image of $Y$ under $\cW^s$-holonomy.
Then $p$ is a  $cu$-julienne density point of $Y$ if and only if $x$ is a  $cu$-julienne density point of $Y'$.

\end{prop}

\subsection{Julienne Quasiconformality}
Here we prove Proposition~\ref{p=inclusion}.
The proof is taken {\em mutatis mutandis} from \cite{Burns Wilk}.

By a simple argument reversing the roles of $x$ and $x'$,
it will suffice to show that $k$ can be chosen so that
\begin{equation}\label{e=inc}
h^s(\hJ^{cu}_n(x))\subseteq \hJ^{cu}_{n-k}(x'),
\end{equation}
for all $n\geq k$, whenever $x$ and $x'$ satisfy the hypotheses of the
proposition.

In order to prove that $k$ can be chosen so that (\ref{e=inc}) holds,
we need two lemmas.
\begin{lemma}\label{l=centhol} There
exists a positive integer $k_1$ such that, for all $x, x'\in \cW^s(p)$,
$$
\hat{h}^s(\hB^c_n(x)) \subseteq \hB^c_{n-k_1}(x'),
$$
for all $n\geq k_1$, where $\hat{h}^s:\hW^{cu}_{loc}(x)\to \hW^{cu}(x')$
is the local $\hW^s$ holonomy.
\end{lemma}

\begin{proof} 
Proposition~\ref{p=fakenu} implies that $\hat{h}^s$ is
$L$-Lipschitz at $x$, for some $L\geq 1$. Therefore the
image of $\hW^c_\star(x,\sigma_n)$ under $\hat{h}^s$
is contained in $\hW^c_\star(x',L\sigma_n)\subseteq \hW^c_\star(x',\sigma_{n-k_1})$,
 for any
$k_1$ large enough so that $\sigma_{-k_1} > L$. 
\end{proof}

\begin{lemma}\label{l=holon} There exists a positive integer
$k_2$ such that the
following holds for every integer
$n\geq k_2$. Suppose $q,q'\in S_n$, with $q' \in \hW^s(q)$.
Let $y\in \hJ^u_n(q)$, and let $y'$ be the image of $y$ under $\cW^s$ holonomy
from $\hW^{cu}_{loc}(q)$ to $\hW^{cu}(q')$.  Then
$$
y' \in \hJ^{u}_{n-k_2}(z'),
$$
for some $z'\in \hW^c_\star(q', \sigma_{n-k_2})$.
\end{lemma}

\begin{proof}
Let $z'$ be the unique point
in $\hW^u(y')\cap \hW^c(q')$.  It is not hard to see
that $z'_j\in \cN_r$, for $j=0,\ldots, n-1$ and that
$z'_{n}$ is the unique point
in $\hW^u(y'_{n})\cap \hW^c(q'_{n})$.
It will suffice to prove that $d_\star(y_{n}',z_{n}')
= O(\tau_n)$ and
$d_\star(q',z') = O(\sigma_n)$.

We have $d_\star(q_{n}, y_{n}) \leq \tau_n$ because
$y\in f^{-n}(\cW^u_\star(q_{n},\tau_n))$. By Proposition~\ref{p=fakenu}, 3(a),
we also have
that
$d_\star(q_{n},q'_{n}) = O(\nu_n)$ and  $d_\star(y_{n},y'_{n}) = O(\nu_n)$, since
$d_\star(q,q')$ and $d_\star(y,y')$ are both $O(1)$. Note that
$q_n$ and $z_n'$ are, respectively, the images of $y_n$ and $y_n'$
under $\hW^u$-hononomy between $\hW^{cs}_{loc}(y_n)$ and
$\hW^{cs}(q_n)$. Uniform transversality of the foliations
$\hW^u$ and $\hW^{cs}$ implies that
$$d_\star(y_n',z_n') = O(\max\{d_\star(q_n,y_n), d_\star(y_n,y_n')\}) = O(\tau_n),$$
since $\nu < \tau$.

We next show that $d_\star(q',z') = O(\sigma_n)$.
By the triangle inequality,
$$d_\star(q_{n}',z_{n}') \leq d_\star(q_n',q_n) + d_\star(q_n,y_n) + d_\star(y_n, y_n') + d_\star(y_n', z_n').$$
All four of the quantities on the right-hand side are easily seen to be
$O(\tau_n)$.
Since $q_{n}'$ and $z_{n}'$ lie in the same $\hW^c$-leaf at
$d_\star$-distance $O(\tau_n)$, Proposition~\ref{p=fakenu}
now implies that
$d_\star(q',z') = O((\gamma_{n})^{-1}\tau_n)$.
But $\tau$ and $\sigma$ were chosen so that $\tau \prec \gamma\sigma$.
Hence $(\gamma_{n})^{-1}\tau_n< \sigma_n$ and
$d_\star(q',z') = O(\sigma_n)$, as desired. 
\end{proof}

\begin{proof}[Proof of Proposition~\ref{p=inclusion}]
As noted above, it suffices to prove the inclusion~(\ref{e=inc}).
For $q \in \hat{B}^c_n(x)$, let $q' = \hat{h}^s(q)$. Then $q' \in \hat{B}^c_{n-k_1}(x')$
by Lemma~\ref{l=centhol}. Hence $q,q'\in S_{n-k_1}$
 and we can apply Lemma~\ref{l=holon} to obtain
$$
h^s(\hat{J}^{cu}_n(x)) \subseteq \bigcup_{z \in Q} \hat{J}^{u}_{n-k_2}(z),
$$
where
$$Q=\bigcup_{q' \in \hat{B}^c_{n-k_1}(x')} \hat{B}^{c}_{n-k_2}(q').$$
For $k\geq k_2$, we have:
$$\bigcup_{z \in Q} \hat{J}^{u}_{n-k_2}(z)\subseteq
\bigcup_{z \in Q} \hat{J}^{u}_{n-k}(z).$$
It therefore suffices to find $k\geq k_2$ such that $Q\subseteq \hat{B}^c_{n-k}(x')$.
This latter inclusion holds if:
$$
\sigma_{n-k_1} + \sigma_{n-k_2} \leq \sigma_{n-k},$$
which is obviously true for all $n\geq k$, if $k$ is sufficiently large.
\end{proof}

\subsection{Julienne Measure}
Next we give the:
\begin{proof}[Proof of Proposition~\ref{p=julmeas}]
Recall that we are using the standard Riemannian volumes and induced
Riemannian volumes on submanfiolds (not the $\|\mathord{\cdot} \|_\star$-volumes).

\begin{lemma}[cf.\ inequalities (21), \cite{Burns Wilk}]\label{l=eqn21}
There exists a constant $C_1>1$ such that, for all $n\geq 0$:
\begin{equation}\label{e=init}
C_1^{-1} \leq
   \frac{\hm_{u}\big( \hW^u_\star(q_{n},\tau_n(p)) \big)}{\hm_{u}\big( \hW^u_\star(q_{n}',\tau_n(p)) \big)}
             \leq C_1,
\end{equation}
for all $q,q'\in S_{n}$, where $\hW^u_\star(x,r) = \hW^u(x)\cap B_\star(x,r)$,
and $\hm_{u}$ is the induced Riemannian metric on $\hW^u$-leaves.
\end{lemma}

\begin{proof}
Recall that the flat $\|\mathord{\cdot} \|_{\flat}$ metric on $T_{p_n}M$
and the Riemannian metric in a neighborhood of $p_n$ (viewed in exponential coordinates at $p_n$)
are uniformly comparable. We will estimate the ratio in (\ref{e=init}) using
the volume on $T_{p_n}M$ induced by $\|\mathord{\cdot} \|_{\flat}$.

On $T_{p_n}M$, the $\|\mathord{\cdot} \|_\star$-metric is
also flat: the ball of radius
$\tau_n$ at $q_n$ is just a translate by $q_n-q_n'$ of a  $\|\mathord{\cdot}\|_\star$-ball of radius
$\tau_n$ at $q_n'$.  Viewed in the $\|\mathord{\cdot} \|_{\flat}$ metric,
a $\star$-ball of radius $\tau_n$ in $T_{p_n}M$ is an ellipsoid with
eccentricity bounded by $K^{-1} B_n$. The intersection of such a ball centered
at $q_n$ of radius $\tau_n(p)$ with $\hW^u(q_n)$ gives the set $\hW^u(q_{n},\tau_n(p))$.
Since the leaves of $\hW^u(q_n)$ are tangent
to a uniformly H\"older continuous distribution, the volumes of these sets
are uniformly comparable to the intersection of $T_{q_n}\hW^u(q_n)$
with $B_\star(q_n,\tau_n(p))$.  This is also a ($u$-dimensional) ellipsoid, call it
$\cE(q_n)$.  Similarly we have an ellipsoid $\cE(q_n')$ centered at $q_n'$.

The distance between the spaces $T_{q_n}\hW^u(q_n)$ and $T_{q_n'}\hW^u(q_n)$
(translated by $q_n-q_n'$) is of the order of $d(q_n,q_n')^\theta$, for some
$\theta\in(0,1]$,  and so is bounded by $c\beta^n$,
where $\beta = \kappa^\theta <1$.  The bound on the eccentricity of $B_\star(q_n,\tau_n)$ then implies
that the ratio between the $u$-dimensional volumes of $\cE(q_n)$ and $\cE(q_n')$ is
bounded above by $D_n = C'(1 + c K^{-1} B_n \beta^n)^u$ and below by $D_n^{-1}$, for
some constant $C'$.
Since $\limsup_{n\to\infty} B_n^{1/n} = 1$, there exists a constant $D$ such that
$D_n\leq D$ for all $n$. We conclude that there exists a constant $C$ satisfying (\ref{e=init}), for all
$n$ and all $q,q'\in S_n$.
\end{proof}

Let $\hE^s, \hE^c$, and $\hE^u$ be the tangent distributions to the
leaves of $\hW^s, \hW^c$, and $\hW^u$, respectively. They are H\"older
continuous by Proposition~\ref{p=fakenu}, part~6.   Furthermore,
the restrictions of
these distributions to $T_n$ are invariant under $Df^j$, for $j=1,\ldots n$.
We next observe that the Jacobian $\mathrm{Jac}(Df^{n}\vert_{\hE^u})$
is nearly constant when
restricted to the set $T_n$.
More precisely, we have:
\begin{lemma}\label{l=jacdist}
There exists $C_2\geq 1$ such that, for all $n\geq 1$, and
all $y,y'\in T_n$,
\begin{equation*}
C_2^{-1}\leq \frac{\mathrm{Jac}(Df^{n}\vert_{\hE^u})(y)}
                  {\mathrm{Jac}(Df^{n}\vert_{\hE^u})(y')} \leq C_2.
\end{equation*}
\end{lemma}

\begin{proof}
By the Chain Rule, these inequalities follow from Lemma~\ref{l=constalpha} with
$\alpha = \hbox{Jac}(Df\vert_{\hE^u})$.
\end{proof}

Let $q\in  S_n$, and let $X\subseteq \hJ^u_n(q)$ be a measurable
set (such as $\hJ^u_n(q)$ itself). Then:
$$
\hm_u(f^n(X)) = \int_{X}
       \hbox{\rm Jac}(Tf^{n}\vert_{\hE^u})(x)\,d\hm_u(x).
$$
From this and Lemma~\ref{l=jacdist} we then obtain:
\begin{lemma}\label{l=generalcompare} There exists
$C_3>0$ such that, for all $n\geq 0$, for
any $q,q'\in S_n$, and any
measurable sets $X\subset \hJ^u_n(q), X'\subset \hJ^u_n(q')$,
we have:
$$C_3^{-1}\frac{\hm_u(f^n(X))}{\hm_u(f^{n}(X'))} \leq \frac{\hm_u(X)}{\hm_u(X')}
\leq C_3 \frac{\hm_u(f^n(X))}{\hm_u(f^{n}(X'))}.$$
\end{lemma}

Recall that $f^n(\hJ^u_n(q)) = \hW^u_\star(q_n,\tau_n)$, for
$q\in S_n$.
The first conclusion of Proposition~\ref{p=julmeas} now
follows from (\ref{e=init}) and Lemma~\ref{l=generalcompare} with
$X= \hJ^u_n(q)$ and $X' = \hJ^u_n(q')$.

The second conclusion is proved similarly.

We next show that there exists $\delta>0$ such that
\begin{equation}\label{e=juratio}
\frac{\hm_{u}(\hJ^{u}_{n+1}(q))}{\hm_{u}(\hJ^{u}_{n}(q))} \geq \delta,
\end{equation}
for all $n\geq 0$ and all $q\in S_n$.
To obtain (\ref{e=juratio}), we will apply
Lemma~\ref{l=generalcompare} with $q=q'$, $X=\hJ^u_{n+1}(q)$,
and $X' = \hJ^u_n(q)$.  This gives us:
$$
\frac{\hm_u(\hJ^{u}_{n+1}(q))}{\hm_u(\hJ^{u}_{n}(q))} \geq
C_3^{-1} \frac{\hm_u(f^{n}(\hJ^{u}_{n+1}(q)))}{\hm_u(f^{n}(\hJ^{u}_{n}(q)))}.
$$
But $f^{n}(\hJ^{u}_{n+1}(q)) =
f^{-1} (\hW^{u}_\star(q_{n+1},\tau_{n+1}))$
and  $f^{n}(\hJ^{u}_{n}(q))=
\hW^u_\star(q_n, \tau_n)$,
and hence:
$$
\frac{\hm_u(f^{n}(\hJ^{u}_{n+1}(q)))}{\hm_u(f^{n}(\hJ^{u}_{n}(q)))} =
\frac{\hm_u(f^{-1} (\hW^{u}_\star(q_{n+1},\tau_{n+1})))}{\hm_u(\hW^u_\star(q_n, \tau_n))}.\\
$$
We show that this ratio is uniformly bounded below away from $0$.
Since $\tau_{n+1}/\tau_n$ is uniformly bounded, and $f$ is uniformly
$C^1$ in the $\star$-metric on $\cN_{\bf r}$, there exists a constant
$\mu<1$ (independent of $n$) such that $f^{-1} (\hW^{u}_\star(q_{n+1},\tau_{n+1}))$ contains the set
$\hW^u_\star(q_n, \mu\tau_n)$. Since $\|\mathord{\cdot}\|_\star$ is a locally flat metric, the set $B_\star(q_n, \mu\tau_n)$
is just the set  $B_\star(q_n, \tau_n)$ dilated (in exponential coordinates at $p_j$)
from $q_n$ by a factor of $\mu$. Since the leaves of the $\cW^u$ foliation are uniformly smooth,
the volumes of  $\hW^u_\star(q_n, C\tau_n)$ and $\hW^u_\star(q_n, \tau_n)$
in the $\|\mathord{\cdot}\|$-metric are therefore uniformly comparable.  This implies
that their Riemannian volumes are comparable.

To prove the final claim, we begin by observing that,
considered as a subset of $\hW^{cu}(x)$, the set $\hJ^{cu}_n(x)$ fibers over $\hB^c_{n}(x)$ with $\hW^u$-fibers $\hJ^u_n(q)$.  We have just proved that these fibers are $c$-uniform.
Since  $\sigma_{n+1}/\sigma_{n} = \sigma(p_n)$ is uniformly bounded away from $0$, the ratio
$$
\frac{\hm_c(\hB^c_{n+1}(x))}{\hm_c(\hB^c_{n}(x))} = 
\frac{\hm_c(\hW^{c}(x,\sigma_{n+1}))}{\hm_c(\hW^{c}(x,\sigma_{n}))} 
$$
is bounded away from $0$, uniformly in $x$ and $n$.
Thus the sequence of bases $\hB^c_{n}(x)$ of $\hJ^{cu}_n(x)$
is regular in the induced Riemannian volume $\hm_c$.
Proposition~\ref{p=fakenu}, part 7.
implies that, considered as a subfoliation of $\hW^{cu}(x)$,
$\hW^u$ is absolutely continuous with bounded
Jacobians. Proposition~\ref{p=unifreg} implies
that the sequence $\hJ^{cu}_n(x)$ is regular, with respect to the
induced Riemannian measure $\hm_{cu}$.  This proves the final claim
of Proposition~\ref{p=julmeas}.
\end{proof}

\subsection{Julienne Density}
We now come to the:
\begin{proof}[Proof of Proposition~\ref{p=sliceden}] 
We must show that if
a measurable set $X$ is both  $\cW^s$-saturated and
 essentially $\cW^u$-saturated at a point
 $x\in \cW^s_\star(p,1)$, then $x$ is a Lebesgue density point of $X$ if and only if
$$
\lim_{n\to\infty} \hm_{cu}(X:\hJ^{cu}_n(x)) = 1.
$$

As in \cite{Burns Wilk}, we will establish the following chain of equivalences:
\begin{alignat*}{3}
 \text{$x$ is a Lebesgue density point of $X$}
  \quad \Longleftrightarrow \quad
              &\lim_{n\to\infty} m(X:B_n(x)) &&= 1 \\
  \quad \Longleftrightarrow \quad
              &\lim_{n\to\infty} m(X:C_n(x)) &&= 1 \\
  \quad \Longleftrightarrow \quad
              &\lim_{n\to\infty} m(X:D_n(x)) &&= 1 \\
  \quad \Longleftrightarrow \quad
              &\lim_{n\to\infty} m(X:E_n(x)) &&= 1 \\
  \quad \Longleftrightarrow \quad
              &\lim_{n\to\infty} m(X:F_n(x)) &&= 1 \\
  \quad \Longleftrightarrow \quad
              &\lim_{n\to\infty} m(X:G_n(x)) &&= 1 \\
  \quad \Longleftrightarrow \quad
              &\lim_{n\to\infty} \hm_{cu}(X:\hJ^{cu}_n(x)) &&= 1.
\end{alignat*}
The sets $B_n(x)$ through $G_n(x)$ are defined as follows.
The set $B_n(x)$ is a $\star$-Riemannian ball in $M$:
$$
B_n(x) = B_\star(x,\sigma_n).
$$
The sets  $C_n(x)$, $D_n(x)$ and $E_n(x)$
will fiber over the same base $D^{cs}_n(x)$, where
$$
 D^{cs}_n(x) = \bigcup_{x'\in \hW^s_\star(x,\sigma_n)}\hB^c_n(x').
 $$
Proposition~\ref{p=fakenu}, part (4) implies that
$D^{cs}_n(x)$ is contained in the $C^1$ submanifold $\hW^{cs}(x)$;
the sequences $D^{cs}_n(x)$ and  $\hW^{cs}_\star(x,\sigma_n)$ are internested.
Let
$$C_n(x)
=\bigcup_{q\in D^{cs}_n(x)} \cW^u_\star(q,\sigma_n),$$
and let
$$
D_n(x) = \bigcup_{q\in D^{cs}_n(x)} J^u_n(q).$$
The set $E_n(x)$ is nearly identical to $D_n(x)$, with the difference
that the $J^u_n$-fibers are replaced with
$\hJ^u_n$-fibers:
$$
E_n(x) = \bigcup_{q\in D^{cs}_n(x)} \hJ^u_n(q) =
\bigcup_{x' \in \hW^s_\star(x,\sigma_n)} \hJ^{cu}_n(x')
= \bigcup_{x' \in \cW^s_\star(x,\sigma_n)} \hJ^{cu}_n(x').$$
The rightmost equality follows
from the fact that $\hW^s_\star(x,\sigma_n) =\cW^s_\star(x,\sigma_n)$,
for all $x\in \cW^s(p,1)$ (Proposition~\ref{p=fakenu}, part (5))

We define $F_n(x)$ to be the foliation product of $\hJ^{cu}_n(x)$ and $\cW^s_\star(x,\sigma_n)$:
$$
F_n(x) = \bigcup_{q\in \hJ^{cu}_n(x), \,\, q'\in \cW^s_\star(x,\sigma_n)} \cW^s(q)\cap
\hW^{cu}(q').
$$
This definition makes sense since the
foliations $\hW^{cu}$ and $\cW^s$ are transverse.
  Finally, let
$$
G_n(x) = \bigcup_{q \in \hJ^{cu}_n(x)} \cW^s_\star(q,\sigma_n).$$
We now prove these equivalences, following the outline described above.

First, recall that $B_n(x)$ is a round $d_\star$-ball about $x$ of radius $\sigma_n$.
The forward implication in the first equivalence
is obvious from the definition of $B_n(x)$.
The backward implication follows from this definition and the fact that
the ratio $\sigma_{n+1}/\sigma_{n} = \sigma(p_n)$
of successive radii is less than $1$, and is
bounded away from both $0$ and $1$ independently
of $n$. From this we also see that $B_n(x)$ is regular.

The set $C_n(x)$ fibers over $D^{cs}_n(x)$, with fiber
$\cW^u_\star(x',\sigma_n)$ over $x'\in D^{cs}_n(x)$.
The sequence $D^{cs}_n(x)$ internests with the
sequence of disks $\hW^{cs}_\star(x, \sigma_n)$, by continuity
and transversality of the foliations $\hW^c$ and $\hW^s$.
Continuity and transversality of the foliations $\cW^u$ and
$\hW^{cs}$  then imply
that $C_n(x)$ and $B_n(x)$ are internested.

To prove the equivalence
$$\lim_{n\to\infty} m(X:C_n(x)) = 1 \Longleftrightarrow
              \lim_{n\to\infty} m(X:D_n(x)) = 1,$$
we note that $C_n(x)$ and $D_n(x)$ both fiber over $D^{cs}_n(x)$,
with $\cW^u$-fibers.  Since $X$ is  essentially $\cW^u$-saturated at $x$,
Proposition~\ref{p=compmeas} implies that it suffices to show that the
fibers of $C_n(x)$ and $D_n(x)$ are both $c$-uniform.  The fibers
of of $C_n(x)$ are easily seen to be uniform, because they are
all comparable to balls in $\cW^u$ of fixed radius $\sigma_n$.
The fibers of $D_n(x)$ are the unstable juliennes $J^u_n(x')$,
for $x'\in D^{cs}_n(x)$.  Uniformity of these fibers follows
from Proposition~\ref{p=julmeas}.

We next prove:

\begin{lemma} The sequences $D_n(x)$ and $E_n(x)$ are internested.
\end{lemma}

\begin{proof} 
Recall that
$$D_n(x) = \bigcup_{q\in D^{cs}_n(x)} J^u_n(q),\quad \hbox{ and }\quad
E_n(x) = \bigcup_{q\in D^{cs}_n(x)} \hJ^u_n(q).$$
Internesting of the sequences $D_n(x)$ and $E_n(x)$ means that there is
a $k\geq 0$ such that, for all $n\geq k$,
$$D_n(x)\subseteq E_{n-k}(x) \quad \hbox{ and }\quad E_n(x)\subseteq D_{n-k}(x).$$
We will show that there is a $k$ for which the first inclusion holds.  Reversing the roles of of $\cW^u$ and $\hW^u$ in the proof gives the second
inclusion.

Suppose $y\in D_n(x)$. Then $y\in  J^u_n(q) = f^{-n}(\cW^u_\star(q_n,\tau_n))$, for
some $q\in D^{cs}_n(x)$; in particular,
\begin{equation}\label{e=firstau}
d_\star(y_n,q_n) = O(\tau_n).
\end{equation}
  Let $\hat q$
be the unique point of intersection of
$\hW^u(y)$ with $\hW^{cs}(x)$.   We will show that $y\in E_{n-k}(x)$,
for some $k$ that is independent of $n$.  In order to do this, it
suffices to show that $\hat q\in D^{cs}_{n-k}(x)$ and
$y\in \hJ^u_{n-k}(\hat q) =   f^{-(n-k)}(\hW^u_\star(\hat q_{n-k},\tau_{n-k}))$.

In order to prove that $\hat q\in D^{cs}_{n-k}(x)$ it will suffice to show that
\begin{equation}\label{e=sigman}
d_\star(q,\hat q) = o(\sigma_n)
\end{equation}
(in fact, $O(\sigma_n)$ would suffice, but the argument gives $o(\sigma_n)$).
In order to prove that $y\in \hJ^u_{n-k}(\hat q)$ it will suffice to show
that
\begin{equation}\label{e=taun}
d_\star(y_n,\hat q_n) = O(\tau_n).
\end{equation}

Equation (\ref{e=sigman}) follows easily from (\ref{e=taun}).
Since $y_n$ and $\hat q_n$ lie in the same $\hW^u$
leaf, Proposition~\ref{p=fakenu} and (\ref{e=taun}) imply
that
\begin{equation}\label{e=taun2}
d_\star(y, \hat q) = O(\hat\nu_n \tau_n) = o(\sigma_n),
\end{equation}
since $\hat\nu \tau \prec \sigma$. Similarly,  Proposition~\ref{p=fakenu} and
(\ref{e=firstau}) imply that
\begin{equation}\label{e=taun3}
d_\star(y, q) = o(\sigma_n).
\end{equation}
Applying the triangle inequality to (\ref{e=taun2}) and
(\ref{e=taun3}) gives (\ref{e=sigman}).

It remains to prove (\ref{e=taun}).
Recall from the construction of
the fake foliations in Proposition~\ref{p=fakenu} that,
 at any point $z$ in the neighborhood $\cN_{\bf r}$ of the orbit
of $p$ in which the fake foliations are defined, the
tangent space  $T_z\hW^u(z)$ lies in the
$\eps$-cone about $T_z \cW^u(z) = E^u(z)$.  Furthermore,
the angle between $T_z\hW^{cs}(z)$ and  either $T_z\hW^u(z)$
or $T_z\cW^u(z)$ is uniformly bounded away from $0$.  Note that
$\hat q_n$ is the unique point in $\hW^u(y_n)\cap \hW^{cs}(x_n)$ and
$q_n$ is the unique point in $\cW^u(y_n) \cap \hW^{cs}(x_n)$;
combining this with (\ref{e=firstau}) gives:
$$d_\star(y_n, \hat q_n) = O (d_\star(y_n, q_n)) = O(\tau_n).$$
This completes the proof.
\end{proof}

We next show:
\begin{lemma} $E_n(x)$ and $F_n(x)$ are internested, as
are $F_n(x)$ and $G_n(x)$.
\end{lemma}

\begin{proof} 
The sets $E_n(x)$ and $F_n(x)$ both fiber over the
same base $\hW^s_\star(x,\sigma_n)$.  The fibers of $E_n(x)$ are the
$cu$-juliennes$\hJ^{cu}_n(x')$, for $x'\in \hW^s(x,\sigma_n)$.
the fibers of $F_n(x)$ are images of $\hJ^{cu}_n(x)$ under
$\cW^s$-holonomy from $\hW^{cu}(x)$ to $\hW^{cu}(x')$, for
$x'\in \hW^s_\star(x,\sigma_n)$.  It follows immediately from
Proposition~\ref{p=inclusion}
that the sequences $E_n(x)$ and
$F_n(x)$ are internested.

To see that $F_n(x)$ and
$G_n(x)$ are internested, suppose that
$q'$ lies in the boundary of the fiber of $F_n(x)$ that lies in
$\cW^s(q)$
for some $q \in \hJ^{cu}_n(x)$.
Then $q' \in \hJ^{cu}_n(x')$ for a point $x'$ that lies in
the boundary of $\cW^s_\star(x,\sigma_n)$. The diameters of
$\hJ^{cu}_n(x)$ and $\hJ^{cu}_n(x')$
are both $O(\sigma_n)$, and $d_\star(x,x') = \sigma_n$.
Hence, if $k$ is large enough, we will have
 $$
\sigma_{n+k} \leq d_\star(q,q') \leq {\sigma_{n-k}}.
$$
Thus all points on the boundary of the fiber of $F_n(x)$  in $\cW^s_{loc}(q)$
lie outside $\cW^s_\star(q,\sigma_{n+k})$ and inside $\cW^s_\star(q,\sigma_{n-k})$.
\end{proof}

We now know that any two of
$D_n(x), E_n(x), F_n(x)$ and $G_n(x)$ are internested.
As discussed above, to prove the fourth through sixth
equivalences, it now suffices to show:

\begin{lemma}\label{l=gnreg}
The sequence $G_n(x)$ is regular for each $x \in \cW^s(p,1)$.
\end{lemma}

\begin{proof}
The set
$$G_n(x) = \bigcup_{q\in \hJ^{cu}_n(x)} \cW^s_\star(q,\sigma_n)$$
fibers over $\hJ^{cu}_n(x)$, with $\cW^s$-fibers
$\cW^s_\star(q,\sigma_n)$.  Since $\cW^s$ is absolutely continuous,
Proposition~\ref{p=unifreg} implies that regularity of $G_n(x)$ follows from
regularity of the base sequence and fiber sequence.
Proposition~\ref{p=julmeas} implies that the sequence
$\hJ^{cu}_n(x)$ is regular in the induced measure $\hm_{cu}$.
As we remarked above, the ratio
$\sigma_{n+1}/\sigma_{n} = \sigma(p_n)$ is uniformly bounded below away
from $0$.  Consequently, the ratio
$$\frac{m_s(\cW^s_\star(q,\sigma_{n+1}))}{m_s(\cW^s_\star(q,\sigma_{n}))}$$
is bounded away $0$, uniformly in $x, q$, and $n$.
The regularity of $G_n(x)$ now follows from Proposition~\ref{p=unifreg}.
\end{proof}

To prove the final equivalence, we use the fact that
$G_n(x)$ fibers over $\hJ^{cu}_n(x)$ with $c$-uniform fibers
and apply Proposition~\ref{p=compmeas2}.  Here we use the
fact that $X$ is $\cW^s$-saturated.
This completes the proof of Proposition~\ref{p=sliceden}. 
\end{proof}


\section{Cocycle Saturation} \label {s.jimmy}

We now explain a generalization of Theorem~\ref{t.nonunifbw} involving
saturation properties of sections.  This brings the results of
\cite{ASV} into the nonuniform setting. We review the notations
from \cite{ASV}.  In this discussion $M$ denotes a closed manifold
and $f\colon M\to M$ a partially hyperbolic diffeomorphism.

A Hausdorff topological space $P$ is {\em refinable} if there exists
an increasing sequence of countable partitions $\cQ_1\prec \cQ_2\prec\cdots\prec \cQ_n\prec \cdots$
into measurable sets such that any sequence $(Q_n)_{n\in\N}$ with $Q_n\in\cQ_n$
and $\bigcap Q_n \neq \emptyset$ converges to a point $\eta\in P$ in the sense
that every neighborhood of $\eta$ contains all $Q_n$ for $n$ sufficiently large.
Every separable metric space is refinable.

We shall consider continuous fiber bundles $\cX$ over $M$ with fiber a Hausdorff
topological space $P$. Such a fiber bundle is {\em refinable} if
$P$ is refinable.

A fiber bundle $\pi\colon \cX\to M$ {\em has stable and unstable holonomies} if,
for every $x,y\in M$ with $y\in \cW^\ast(x)$ and $\ast\in\{u,s\}$, there exists
a homeomorphism $h^\ast_{x,y}\colon \pi^{-1}(x)\to \pi^{-1}(y)$ with the following
properties:
\begin{enumerate}
\item $h^\ast_{x,x} = Id_{\pi^{-1}(x)}$, and $h^\ast_{y,z}\circ h^\ast_{x,y} = h^\ast_{x,z}$;
\item the map $(x,y,\eta,d_\ast(x,y))\mapsto h^\ast_{x,y}(\eta)$ is continuous
on its domain (a subset of $M \times M \times \cX \times [0,\infty)$), where
$d_\ast(x,y)$ stands for the distance between $x$ and $y$ in
$\cW^\ast(x)$.\footnote {This can be reformulated, in view of (1), as
requiring that $(x,y,\eta) \mapsto h^*_{x,y}(\eta)$ is continuous
when we restrict $x$ and $y$ to belong to {\it local} $\cW^*$ leaves.}
\end{enumerate}
Our main result concerns the saturation properties of sections
of refinable bundles with stable and unstable holonomies.  In
analogy with the definition of stable saturated set, we say that
a section $\Psi\colon M\to\cX$ is {\em $h^s$-saturated} if,
for every $x\in M$ and $y\in \cW^s(x)$:
$$\Psi(y)  = h^s_{x,y}(\Psi(x)).$$
We similarly define {\em $h^u$-saturated} sections (the terms
{\em $s$-invariant} and {\em $u$-invariant} are used in \cite{ASV}). A section
is {\em bisaturated} if it is both $h^s$- and $h^u$-saturated.
A section $\Psi$ is {\em bi~essentially saturated} if there exist
an $h^s$-saturated section $\Psi^s$ and a $h^u$-saturated section
$\Psi^u$ such that $\Psi = \Psi^s = \Psi^u$ almost everywhere with respect
to volume on~$M$.

\medskip

\noindent{\bf Examples:}
\begin{enumerate}
\item Let $\cX = M\times\{0,1\}$ and set $h^\ast_{x,y}(\eta) = \eta$. In this trivial example,
if $A\subset M$ is
a ($\cW^s$/$\cW^u$/bi) - saturated set, then $x\mapsto (x,1_A(x))$ is an ($h^s$/$h^u$/bi) - saturated
section. If $A$ is bi~essentially saturated, then so is the associated section.

\item (cf.~\cite{Wi}, Proposition 4.7) Every H\"older-continuous function $\psi\colon M\to \R$ determines
stable and unstable holonomy maps on the bundle $M\times \R$,
invariant under the skew product $(x,\eta)\mapsto (f(x), \eta + \psi(x))$.

If $\Psi\colon M\to \R$ is a 
continuous solution to the cohomological equation
\begin{eqnarray}\label{e.cohom}
\psi = \Psi\circ f - \Psi,
\end{eqnarray}
then $\Psi$ is a bisaturated section.
Moreover, if  $f$ is $C^2$, volume-preserving and ergodic,
$\Psi\colon M\to \R$ is measurable, and the
equation (\ref{e.cohom}) holds almost everywhere with respect to volume,
then $\Psi$ is a bi~essentially saturated section.

\item (cf.~\cite{ASV}) Let $A\colon M\to SL(n,\R)$ be a H\"older-continuous
matrix-valued cocycle.  If this cocycle is dominated (in the
sense of \cite{ASV}), then it determines 
in a natural way stable and unstable holonomies on the refinable fiber bundle
$\cX = M\times \cM(\R P^{n-1})$, where $\cM(\R P^{n-1})$ is the space
of probability measures on the projective space $\R P^{n-1}$.

Suppose that  the Lyapunov exponents
of $A_n(f) = (A\circ f^{n-1}) (A\circ f^{n-1})\cdots A$ vanish almost everywhere.  
Then $A$ determines a bi~essentially saturated
section of the bundle $\cX$. These results are proved in \cite{ASV}
and used to show that the generic such cocycle over
an accessible, center-bunched partially hyperbolic
diffeomorphism has a nonvanishing exponent.
\end{enumerate}

Our main result expands Theorem~\ref{t.nonunifbw} to include bi~esentially saturated
sections. Following \cite{ASV}, we introduce an analogue for measurable
sections of the notion of density point for measurable sets.

Let $\pi\colon \cX\to M$ be a refinable bundle.
We say that $p\in M$ is a {\em point of measurable continuity} for
a section $\Psi\colon M\to \cX$, if there exists $\eta \in \cX$
such that $p$ is a Lebesgue density point of $\Psi^{-1}(V)$, for
every open neighborhood $V$ of $\eta$ in $\cX$.  If such an $\eta$ exists,
it is unique, and is called the {\em density value} of $\Psi$
at $p$.

Let $MC(\Psi)$ be the set of points of measurable continuity
of $\Psi$.  We define a measurable section $\tilde\Psi\colon MC(\Psi)\to \cX$
by setting $\tilde\Psi(p)$ to be the density value of $\Psi$ at $p$.
Then $MC(\Psi)$ has full volume in $M$, and
$\tilde\Psi = \Psi$ almost everywhere, with respect to volume (see
Lemma 7.10, \cite{ASV}).

%
%


\begin{thm}[cf. Theorem 7.6, \cite{ASV}]\label{t.nonunifjimmy} 
Let $f$ be $C^2$
and partially hyperbolic, and let $\cX$ be a refinable fiber bundle
with stable and unstable holonomies.

Then, for any bi~essentially saturated
section $\Psi\colon M\to \cX$:
\begin{enumerate}
\item $ MC(\Psi)\cap \CB^{+}$ is $\cW^s$-saturated, and  the restriction of $\tilde\Psi$ to $MC(\Psi)\cap CB^{+}$ 
is $h^s$-saturated;
\item $MC(\Psi)\cap \CB^{-}$ is $\cW^u$-saturated, and  the restriction of $\tilde\Psi$ to $MC(\Psi)\cap \CB^{-}$ 
is $h^u$-saturated;
\end{enumerate}

\end{thm}

\begin{proof} The proof follows the same lines as
Theorem 7.6 in \cite{ASV}.  The proof there adapts the
proof of the main result in \cite{Burns Wilk}, and we correspondingly
adapt the proof of Theorem~\ref{t.nonunifbw} here.

We first prove the theorem under the assumption that the bundle
$\cX$ has stable and unstable holonomies.
We prove the first part of the theorem; the second part follows from the first,
replacing $f$ by $f^{-1}$.
Let $\pi\colon \cX\to M$ be a refinable bundle with stable
and unstable holonomies.
The holonomy  maps $h^s$ and $h^u$ define foliations 
$\cF^s$ and $\cF^u$ of $\cX$; the leaf of $\cF^\ast$ through
a point $\eta\in \cX$ is:
$$
\cF^\ast(\eta) = \{ h^\ast_{\pi(\eta), y}(\eta)\colon y\in \cW^\ast(\pi(\eta))\}.
$$
We similarly define for $r>0$ the local leaf: 
$$\cF^\ast(\eta,r)= \{ h^\ast_{\pi(\eta), y}(\eta)\colon y\in \cW^\ast(\pi(\eta),r)\}.
$$
Observe that a section $\Phi$ is
$\ast$-saturated if and only if its image $\Phi(M)\subset \cX$ is a union of whole leaves
of $\cF^\ast$.


Fix a bi~essentially saturated section $\Psi\colon M\to \cX$.
Recall that bi~essential saturation of $\Psi$ means that there
exist an $h^s$-saturated section $\Psi^s$ and a
$h^u$-saturated section $\Psi^u$ such that
$\Psi^s=\Psi^u=\Psi$ almost everywhere.  

Fix $x\in MC(\Psi)\cap CB^{+}$, and let $\eta = \tilde\Psi(x)$ be 
the density value of $\Psi$ at $x$.
Note that $\eta$ is also a density value for $\Psi^s$ and $\Psi^u$.
We will show that for every $y\in \cW^s(x,1)$, $h^s_{x,y}(\eta)$
is a (the) density value of $\Psi$ at $y$. Since
$CB^+$ is $\cW^s$-saturated, this will simultaneously establish
that $MC(\Psi)\cap CB^{+}$
is $\cW^s$-saturated and that the restriction of $\tilde\Psi$
to $MC(\Psi)\cap CB^+$ is $h^s$-saturated.

To this end, fix $y\in \cW^s(x,1)$, and 
let $V$ be a neighborhood of $h^s_{x,y}(\eta)$ in $\cX$.
Note that $h^s_{x,y}(\eta)$ lies on the local leaf $\cF^s(\eta, 1)$.
To show that $h^s_{x,y}(\eta)$ is a density value for 
$\Psi$ at $y$, we must show that $y$ is a density point of $\Psi^{-1}(V)$.

Continuity of the stable holonomy maps in $\cX$ and
stable saturation of $\Psi^s$ together imply that $(\Psi^s)^{-1}(V)$ is  
$\cW^s$-saturated at $y$; recall this means that there exist $0<\delta_0<\delta_1$
such that for any $z\in B(y,\delta_0) \cap (\Psi^s)^{-1}(V)$, we have $\cW^s(z,\delta_1)\subset \Psi_s^{-1}(V)$.
Similarly, $(\Psi^u)^{-1}(V)$ is   $\cW^u$-saturated at $y$,
and so  $\Psi^{-1}(V)$ is  bi~essentially saturated at $y$.

Fix $\eps>0$ and $\delta>0$ such that 
$\pi^{-1}(B(y,\eps)) \cap N_\delta \subset V$,
where
$$N_\delta = \bigcup_{z\in B(\eta,\delta)} \cF^s(z,1)$$
is the union of the local $\cF^s$ leaves through
$B(\eta, \delta )$ in $\cX$.  Since $N_\delta$
is saturated by local $\cF^s(\mathord{\cdot},1)$ leaves, and the section
$\Psi^s$ is $h^s$-saturated, it follows that the set $(\Psi^s)^{-1}(N_\delta)$
is saturated by local $\cW^s(\mathord{\cdot},1)$-leaves.
The set $\Psi^{-1}(B(\eta, \delta ))$ is  bi~essentially saturated
at $x$
and coincides mod~$0$ with the set $(\Psi^s)^{-1}(B(\eta, \delta))$,
which is  $h^s$-saturated at $x$.
Since $x\in MD(\Psi)$, it is a Lebesgue density point of  $\Psi^{-1}(B(\eta, \delta ))$.
But $x$ is also an element of $CB^+$, and so
Proposition~\ref{p=sliceden} implies that $x$ is a $cu$-julienne density point of
$(\Psi^s)^{-1}(B(\eta, \delta ))$, and hence of $(\Psi^s)^{-1}(N_\delta)$ as well.

Now, since  $(\Psi^s)^{-1}(N_\delta)$
is saturated by local $\cW^s(\mathord{\cdot},1)$-leaves, and $x\in CB^+(f)$,
Proposition~\ref{p.holonomyinv}
implies that $y$ is also a $cu$-julienne density point of 
$(\Psi^s)^{-1}(N_\delta)$.  Thus $y$ is a $cu$-julienne
density point of $B(y,\eps)\cap (\Psi^s)^{-1}(N_\delta)$.
But 
$$B(y,\eps)\cap (\Psi^s)^{-1}(N_\delta) = (\Psi^s)^{-1}(\pi^{-1}(B(y,\eps)) \cap N_\delta) \subset (\Psi^s)^{-1}(V);$$ 
since the latter set is  $\cW^s$-saturated and  essentially $\cW^u$-saturated at $y$,
and since $y\in CB^+(f)$, Proposition~\ref{p=sliceden} implies
that $y$ is a Lebesgue density point of $(\Psi^s)^{-1}(V)$.
Finally, since $(\Psi^s)^{-1}(V) = \Psi^{-1}(V) \,\mod 0$, 
we obtain that $y$ is a Lebesgue density point of $\Psi^{-1}(V)$.
This completes the proof of Theorem~\ref{t.nonunifjimmy}.
\end{proof}

\section{Examples} \label{s.examples}

Here we will be interested first in the construction of a $C^2$-open class of
maps which are not uniformly center bunched, but display
{\it nonuniform center bunching} in the
sense that the set $\CB$ of center bunched points has full Lebesgue measure.  We then
show, using Corollary \ref {c.weird},
that this class contains $C^2$-stably ergodic maps, and describe an application
of Theorem \ref {t.nonunifjimmy} to the cohomological equation.


All of the following constructions can be carried out in the volume-preserving
setting.  We do it in the symplectic setting, as the arguments are slightly more
subtle.

\subsection{A nonuniformly, but not uniformly, center bunched example}

Let $P$, $Q$ and $S$ be compact symplectic manifolds, and let $F\colon P\to P$,
$G\colon Q\to Q$
and $H\colon S\to S$ be symplectic $C^2$ diffeomorphisms with the following properties:
\begin{enumerate}
\item $F$ and $G$ are Anosov diffeomorphisms.

\item We have
$$
\sup_Q \|DG \vert_{E^s_G}\| < \inf_S \m(DH)^2 \le \sup_S \| DH \|^2 < \inf_Q \m \big( DG\vert_{E^u_G} \big) , 
$$
so that $G\times H\colon Q\times S\to Q\times S$ is partially hyperbolic
and center bunched, with center bundle tangent to the $S$ factor.

\item We have
$$
\sup_P \|DF\vert_{E^s_F}\| < \inf_S \m(DH) \le \sup_S \| DH \| < \inf_P \m \big( DF\vert_{E^u_F} \big) ,
$$
so that $F\times G\times H$ is partially hyperbolic on $M=P\times Q\times S$, 
with center bundle tangent to the 
$S$ factor.

\item  Indicating by $m_P$ the normalized volume measure induced by the
symplectic form on $P$, we have 
$$
\int \log \|DF\vert_{E^s_F}\| \, dm_P<  2 \inf_S \log \m(DH) \le 
2 \sup_S \log \|DH\| < \int \log \m(DF\vert_{E^u_F}) \, dm_P \, .
$$

\item There exists a point $p \in P$ of period $k$ under $F$ such that:
$$ 
\|D_p F^k\vert_{E^s_F}\| < \m(D G\vert_{E^s_G})^k  < 
\|D G\vert_{E^u_G}\|^k < \m(D_p F^k\vert_{E^u_F}) \, ,
$$
which implies that $\bigcup_{j=0}^{k-1} \{F^j(p)\}\times Q\times S$ 
is normally hyperbolic and contained in $\CB$.
\end{enumerate}
Let $\omega$ be the symplectic form in $M = P \times Q \times S$ given by the sum of
the forms on $P$, $Q$ and $S$.  Then $f_0 = F \times G \times H:M \to M$ is
symplectic.


\begin{lemma} \label {product1}

If $f$ is a $C^2$ volume-preserving ($C^2$-small) perturbation of $f_0$,
then $f$ is nonuniformly center bunched.

\end{lemma}

\begin{proof}

To show that almost every orbit is forward center bunched,
it is enough to prove that for any $f$-invariant set  $W$ of positive Lebesgue
measure, we have
$$
\frac {1} {m(W)} \int_W \log \|Df|_{E^s_f}\| \, dm < 
2 \inf_M \log \m (Df|_{E^c_f}).
$$
We notice that $E^s_f$ is close to $E^s_F \oplus E^s_G$ and $E^c_f$ is close to
$TS$ everywhere.  Thus the right hand side is close to $2 \inf_S \log {\bf m}(DH)$,
while the left hand side is bounded, up to small error, by the maximum of
$\sup_Q \|DG|_{E^s_G}\|$ and $\int \log \|DF|_{E^s_F}\| \,  d\pi_* \mu$, where
$\mu$ is the normalized restriction of the
Lebesgue measure $m$ to $W$ and $\pi:M \to P$ is the
coordinate projection.  By (2) and (4), we are reduced to showing that $\pi_* \mu$
is weak-$*$ close to $m_P$.

An $f$-invariant probability measure which is absolutely continuous with
respect to the unstable foliation
$\cW^u_f$ will be called an $u$-state for $f$.  One defines
$s$-states analogously.  Let $\cU(f)$ be the set of $u$-states for $f$ and
$\cS(f)$ be the set of $s$-states for $f$.  An $u$-state that is also an
$s$-state will be called an $su$-state.
Since $f$ is $C^2$, the $\cW^u_f$ and $\cW^s_f$ foliations are absolutely
continuous, thus $\mu$ is an $su$-state.  We are going to
show that this already implies that $\pi_* \mu$ is close to $m_P$.

The uniform expansion in the unstable direction as we iterate forward
has a regularization effect which implies that there is an
apriori bound on the densities of the disintegration of an $u$-state for $f$
along $\cW^u_f$: the quotient between the densities at
different points in the same unstable leaf is bounded by $K^d$ where $K$
is a constant (uniform in a $C^2$ neighborhood of $f_0$) and
$d$ is the distance between the points inside the leaf.  (Recall that the
density is defined, in each leaf, only up to scaling but the quotient is
well defined and given by the Anosov-Sinai cocycle; see formula (11.4) in \cite{BDV livro}.)

This bound has the important consequence that $\cU(f)$
is closed (and hence compact) in the weak-$*$ topology.  Moreover, in a $C^2$
neighborhood $\cV$ of $f_0$, the set $\bigcup_{f \in \cV}
\{f\} \times \cU(f)$ is also closed.  We call this fact the
upper-semicontinuity in $f$ 
of the set of $u$-states, see \cite {BDPP} for a detailed proof.

Analogous considerations show that the set of
$s$-states is upper-semicontinuous in $f$.  Thus
$\bigcup_{f \in \cV} \{f\} \times (\cS(f) \cap \cU(f))$ is closed as well,
so the set of $su$-states also depends upper-semicontinuously on $f$.

The product structure of the foliations implies
that an $su$-state for $f_0$
projects onto an $su$-state for $F \times G$, which is $C^2$ Anosov, and
the absence of a central
direction for $F \times G$
implies that the projection is absolutely continuous.  Since $F
\times G$ is Anosov, it is ergodic so the projection is Lebesgue on $P
\times Q$.  Projecting again, we conclude that $\pi_* \nu=m_P$ whenever $\nu$
is an $su$-state for $f_0$ (in fact, an $su$-state for $f_0$ is just the product
of Lebesgue on $P \times Q$ by an arbitrary invariant probability measure on
$S$).  By upper-semicontinuity, if $f$ is
close to $f_0$, the projection of any $su$-state for $f$ is weak-$*$ close
to $m_P$.  The result follows.
\end{proof}

Notice also that we may construct the map $f_0$ so that 
no $f$ nearby is center bunched.
For example, one can arrange that the conditions above hold
and in addition 
there are hyperbolic periodic points $p' = F^\ell(p')$, $q=H^m(q)$ such that 
$$ 
\rho( D_{p'} F^\ell \vert_{E^s_F}) ^{1/\ell} >  \rho( D_{q} H^m )^{-2/m} \, ,
$$ 
where $\rho$ denotes spectral radius.
Note that the main theorem in \cite{Burns Wilk} does not apply to such an example, nor to its
perturbations.

\subsection{Stable ergodicity}

Condition (5) implies that for any $C^1$ perturbation
$f$ of $f_0$, there exists a normally hyperbolic manifold $N_f$,
$C^1$-close to
$\bigcup_{j=0}^{k-1} \{F^j(p)\} \times Q\times S$, whose connected components are
permuted under $f$.


Let us say that $N_f$ is
accessible if for any $x$ and $y$ in the same connected component
of $N_f$, there is an $su$-path in $N_f$ connecting $x$ and $y$.  We say
that $N_f$ is stably accessible if $N_g$ is accessible for every $g$ in a
neighborhood of $f$ in $\Diff^1_\omega(M\times N\times P)$.
These properties are non-void:

\begin{lemma} \label {product2}
For any neighborhood $\cZ$ of $f_0$ in $\Diff^\infty_\omega(M)$
there exists $f\in \cZ$ such that $N_f$ is stably accessible.
\end{lemma}

\begin{proof} In \cite{Shub Wilk} it is shown that for every neighborhood $\cV$ of the identity
in $\Diff^\infty_{\omega_{Q\times S}}(Q \times S)$ there exists $\Phi\in \cV$ such that
$\Phi \circ (G\times H)$ is stably accessible. For such a $\Phi$,
define $\phi\in \Diff^\infty_{\omega}(M\times N\times P)$ by
$\phi = \Id_{M}\times  \Phi$. Then $\phi\circ f_0$ is close to $f_0$ and
satisfies the desired properties.
\end{proof}

\begin{lemma}\label{l.CB access}
If $f$ is $C^1$ near $f_0$ and $N_f$ is accessible, then we can join any two points in
$\CB$ by an $su$-path with corners in $\CB$.
\end{lemma}

\begin{proof}
Fix $x \in N_f$.
Obviously $\cW^c(x) \subset N_f \subset \CB$,
and since $N_f$ is stably accessible, any two points
in $\cW^c(x)$ can be joined by an $su$-path with corners
in $N_f$ and hence in $\CB$.  Thus it is enough to show that any $y \in \CB$ can be
joined to some point in $\cW^c(x)$ through an $su$-path
with corners in $\CB$.  The action of $f$ on $M/\cW^c_f$ is topologically conjugated
to the Anosov map $F \times G$, and under this identification, the projection of any
unstable or stable leaf of $f$ is an unstable or stable leaf of $F \times G$.  Obviously,
for $F \times G$ any two points can be connected by an $su$-path with $2$ legs.
We conclude that for every $y \in M$ there exists $z \in \cW^c(x)$ such that
$\cW^u(y) \cap \cW^s(z) \neq \emptyset$.
When $y \in \CB$, $\cW^u(y) \cap \cW^s(z) \subset \CB$ (since $y,z \in \CB$ and
$\CB^+$ is $\cW^s$-saturated while $\CB^-$ is $\cW^u$-saturated), showing that
$y$ is connected to $\cW^c(x)$ by a $2$-legged $su$-path with corners in $\CB$.
\end{proof}

Putting together Lemmas~\ref{product2}, \ref{l.CB access} and 
Corollary \ref {c.weird} we conclude:

\begin{thm}
If $f$ is $C^2$-close to $f_0$ and $N_f$ is
accessible then $f$ is ergodic
(and in fact, a K-system).
\end{thm}

\subsection{Continuity of bi saturated sections and the cohomological equation}

\begin{lemma}
	
	Let $f:M \to M$ be a $C^2$ volume-preserving partially hyperbolic diffeomorphism
and let $Z$ be a bi essentially saturated set of positive Lebesgue measure.
	If $x \in \supp(m|Z)$ then any $su$-path starting at $x$ can be approximated
by an $su$-path with corners in $Z$.
	
\end{lemma}

\begin{proof}
	
	Let us say that $z \in Z$ is $k$-pretty if almost every
$w \in \cW^u(z) \cup \cW^s(z)$ is $k-1$-pretty, where all points
of $Z$ are declared to be $0$-pretty.  Since $\cW^u$ and $\cW^s$ are
absolutely continuous, it follows by induction
that almost every $z \in Z$ is $k$-pretty for every $k$.
	
	Consider an $su$-path connecting $x_0$ to $x_n$ through
$x_1,\dots,x_{n-1}$.  Now just approximate $x_0$ by an $n$-pretty point $z_0$, and then
successively $x_i$ by an $n-i$-pretty point $z_i \in \cW^*(z_{i-1})$.
\end{proof}

\begin{thm}
	
Let $f$ be $C^2$-close to $f_0$ and let $\cX$ and $\Psi$ be as in
Theorem \ref {t.nonunifjimmy}.  If $N_f$ is accessible then
$\Psi$ coincides almost everywhere with a continuous bi invariant
section.
	
\end{thm}

\begin{proof}

	Since any two points of $\CB$ can be joined by an $su$-path with corners
in $\CB$, and $\CB$ has positive Lebesgue measure, it follows that $MC(\Psi)$
contains $\CB$.

	Let us show that we can define a bi saturated section that coincides
with $\tilde \Psi$ on $\CB$.  By the argument of Section~8.2 of \cite {ASV} (where center
bunching does not play a role), the accessibility of $f$ implies that such a section is
necessarily continuous, and since $m(\CB)=1$, $\Psi$ must coincide almost
everywhere with it.

We notice that, restricting the above considerations to $N_f \subset \CB$,
and using that $N_f$ is accessible, we can already conclude that $\tilde \Psi|N_f$
is continuous.

Let $x \in N_f$.  We are going to show that, for any $su$-path starting and
ending at $x$,
the composition of holonomies along the $su$-path fixes $\tilde \Psi(x)$.
Since $f$ is accessible, this allows us to define a bi saturated section:
join $x$ to any $y \in M$ by any $su$-path and apply the holonomy to $\tilde
\Psi(x)$. 
If well defined, this new section automatically will coincide
with $\tilde \Psi(x)$ on $\CB$ by Theorem \ref {t.nonunifjimmy} (since,
by Lemma~\ref{l.CB access}, $x$ can be joined to any $y \in \CB$ through an $su$-path
with corners in $\CB$).

Let us consider thus an $su$-path starting and ending in $x_0$, and its
composed holonomy map $h$.  Assume that $h(\tilde \Psi(x)) \neq \tilde
\Psi(x)$.  By the
previous lemma, it is approximated by an $su$-path with corners in $\CB$. 
Apriori, the extremes of the latter path do not belong to $N_f$, but by
adding at most $4$ short legs to the latter (two at the beginning and two at the
end), we obtain an $su$-path starting and ending at points $y,z \in N_f$. 
Since the corners of this path all belong to $\CB$, the corresponding
composed holonomy map $\tilde h$ takes $\tilde \Psi(y)$ to $\tilde
\Psi(z)$.  Since $y,z \in N_f$ are close to $x$, we can use the continuity
of holonomy maps, and of $\tilde \Psi|N_f$, to conclude that $\tilde
h(\tilde \Psi(y))$ is close to $h(\tilde \Psi(x))$ and $\tilde \Psi(z)$ is
close to $\tilde \Psi(x)$.  Since we assumed that $h(\tilde \Psi(x)) \neq
\tilde \Psi(x)$, this implies that $h(\tilde \Psi(y)) \neq \tilde \Psi(z)$,
contradiction.
\end{proof}

One particular interesting application is the case of the cohomological equation
(see Example 2 in Section \ref {s.jimmy}): if $\psi : M \to \R$ is a H\"older continuous
function then a measurable solution of the cohomological equation
$\psi=\Psi \circ f-\Psi$ coincides almost everywhere with a continuous one.

One can also deduce non-degeneracy of the Lyapunov spectrum of generic
bunched cocycles over $f$ (see Example 3 of Section \ref {s.jimmy}).
However, the application of those ideas to the analysis of the central
Lyapunov exponents of $f$ themselves is more subtle since this cocycle is not
bunched (but only nonuniformly bunched), and will be carried out elsewhere: we will
show for instance that in the case that $S$ is a surface then stably Bernoulli,
nonuniformly hyperbolic examples like above can be obtained.

\subsection{Further examples}

The mechanism for ergodicity implemented above can be abstracted somewhat to a
criterion for ergodicity, which we quickly describe.

Let $f:M \to M$ be a $C^2$ accessible partially hyperbolic volume preserving
diffeomorphism.  Let $N \subset M$ be a normally hyperbolic compact (not necessarily connected)
submanifold.\footnote {Our arguments would also work by taking $N$ as a
(non-compact) leaf of a normally hyperbolic lamination.}
It is easy to see that $T_x N=(E^s(x)
\cap T_x N) \oplus (E^c(x) \cap T_x N) \oplus (E^u(x) \cap T_x N)$ at every
$x \in N$.  If those three subbundles are non-trivial, then this splitting
is partially hyperbolic.
We are interested on the case that $N$ is
$c$-saturated in the sense that $T_x N \supset E^c(x)$ for every $x \in N$.
Assume that $f|N$ is center bunched and has some open accessibility
class.  We will show that
\emph{the restriction of Lebesgue measure to the set $\CB$
is either null or ergodic.}

Let us first note that, since $N$ is normally hyperbolic, the condition that
$f|N$ is center bunched implies that $N \subset \CB$.

For a set $U \subset N$, let $\tilde U$ be the set of all $x \in M$ such
that there exists $y \in U$ such that
$\cW^u(x) \cap \cW^s(y) \neq \emptyset$.  Notice that
since $N$ is $c$-saturated, it is
clear that $\interior \tilde U \supset \interior U$.

We claim that if $\CB$ is dense, and if $U$ is an
open accessibility class for $f|N$, then any two points in $\CB \cap
\bigcup_{k \in \Z} f^{k}(\interior \tilde U)$ can be joined by an $su$-path with
all corners in $\CB$.  

Since $f$ is accessible, almost every orbit is dense
(by Theorem \ref {t.Brin}); hence the claim implies that almost every pair in $\CB$
can be joined by an $su$-path with corners in $\CB$, which gives the
conclusion, by Corollary \ref {c.weird}; that is, if $\CB$ has positive measure
(which, by Theorem \ref {t.Brin}, implies that it is dense), then the restriction
of $f$ to $\CB$ is ergodic.

To prove the claim, note first that if $x \in \CB$, then $\cW^u(x) \cap \cW^s(y)
\subset \CB$,
for any $y\in N$.  Since $U$ is an accessibility class of $f|N$,
and $N\subset CB$, it follows that
any two points in $\CB \cap \tilde U$ can be connected
by an $su$-path with all corners in $\CB$.

Since $f$ is accessible, 
so is $f\times f$; Theorem \ref {t.Brin} implies that $f\times f$ is topologically transitive.
This implies that for any three open sets
$V_1,V_2,V \subset M$
there exists $n\in\Z$ such that
$V_j\cap f^n(V)\ne\emptyset$, for $j=1,2$. 
In particular, for any pair of integers $k_1, k_2$, there exists $n \in \Z$ such that
$f^{k_j}(\interior \tilde U)\cap f^n(\interior \tilde U)\ne\emptyset$, for $j=1,2$. 
Since $\CB$ is dense,
we can find points $x'_j \in \CB \cap f^n(\interior \tilde U)
\cap f^{k_j}(\interior \tilde U)$.  Then we can join $x_1$ to
$x_2$ by an $su$-path with corners in $\CB$ by going through $x_1'$ and
$x_2'$: $x_j$ and $x_j'$ can be joined since
$f^{-k_j}(x_j),f^{-k_j}(x_j') \in \CB \cap \tilde U$, while
$x_1'$ and $x_2'$ can be joined since
$f^{-n}(x_1'),f^{-n}(x_2') \in \CB \cap \tilde U$.  This proves the claim.


One can also apply the argument of the previous section to conclude, for instance, that if $\psi:M \to \R$ is a H\"older continuous function then any measurable solution of the cohomological equation $\psi=\Psi \circ f-\Psi$ defined over $\CB$ coincides almost everywhere with a continuous solution defined in the whole $M$.  We notice that here it is only needed to
assume that $m(\CB)>0$, and apriori the system could even be non-ergodic as far
as the current theory goes.

%
%
%
%
%
%

\appendix

\section{Reobtaining Some Results from~\cite{Bon Diaz Pujals}}


A variation of the method presented in Section~\ref{s.elliptic}
allows one to obtain various \cite{Bon Diaz Pujals}-like (topological)
conclusions from \cite{Bochi Viana}-like (ergodic) results.
To illustrate, we will reobtain the following:

\begin{otherthm}[\cite{Bon Diaz Pujals}] \label{t.original BDP}
A diffeomorphism that has a non-dominated homoclinic class
can be perturbed to display a nearby
periodic orbit with all eigenvalues of the same modulus.
\end{otherthm}

Let us explain the result from \cite{Bochi Viana} that we need.
Let $(X,\mu)$ be a non-atomic probability space, and let $f$ be an ergodic automorphism of it.
Fix a positive integer $d$ and let $L^\infty$
be the set of measurable maps (called \emph{cocycles})
$A: X \to \GL(d,\R)$
such that $\|A^{\pm 1}\|$ are $\mu$-essentially bounded,
where maps that differ on zero sets are identified.
Notice $L^\infty$ is a Baire space.

Given a cocycle $A \in L^\infty$,
asymptotic information about the products
$$
A^n_f(x) = A(f^{n-1} (x)) \cdots A(x)
$$
is given by Oseledets Theorem.
So let $\R^d = E^1(x) \oplus \cdots \oplus E^k(x)$
be the Oseledets splitting, defined for $\mu$-a.e.\ $x\in X$,
and let $\lambda_1(A) \ge \cdots \ge \lambda_d(A)$ the Lyapunov exponents repeated according to multiplicity.
(Notice that $k$ and the Lyapunov exponents are constant $\mu$-almost everywhere
by ergodicity.)
We also write $L_i(A) = \sum_{j=1}^i \lambda_j(A)$.
We have
$$
L_i(A)  = \inf_{n \ge 1} \frac{1}{n} \int \log \| \mathord{\wedge}^i A^n_f(x) \| \, d\mu(x) \, .
$$
As a consequence of this formula,
the function $A \in L^\infty \mapsto L_i (A)$ is upper-semicontinous,
and hence its points of continuity form a residual set.
Another semi-continuity property that follows easily from the formula is:

\begin{lemma}\label{l.semicont}
Given $A\in L^\infty$, $C> \|A^{\pm 1}\|_\infty$, and $\eps>0$,
there exists $\delta>0$ such that if $B\in L^\infty$
is such that if $\|B^{\pm 1}\|_\infty < C$ and $\mu [B \neq A] < \delta$
then $L_i(B) < L_i(A) + \eps$.
\end{lemma}

Let $\R^d = E(x) \oplus F(x)$ be a splitting defined for $\mu$-a.e.\ $x$
and invariant under a cocycle $A\in L^\infty$.
Also assume that $\dim E$ is constant (called the \emph{index} of the splitting).
We say that the splitting is \emph{dominated}
(or, more precisely, that \emph{$E$ dominates~$F$})
in the case that there exists
$m \in \N$ such that
\begin{equation}\label{e.def dom}
\frac{\left\|A^m(x)|_{F(x)}\right\|}
{\m \left( A^m(x)|_{E(x)} \right)} \le \frac{1}{2}
\quad \text{for $\mu$-a.e.\ $x\in X$.}
\end{equation}
It is not hard to check the 
following elementary properties\footnote{Or see e.g.~\cite[Appendix~B]{BDV livro}.}:
\begin{enumerate}
\item The angle between $E$ and $F$ is essentially bounded from below.
\item For a fixed index, the dominated splitting is unique over the points where it exists.
\item In the case that the space $X$ is compact Hausdorff and $A$ is a continuous map,
then the splitting can be defined over each point of $\supp \mu$,
and varies continuously.
\end{enumerate}

We say that the Oseledets splitting of $A$ is \emph{trivial} if $k=1$,
and \emph{dominated} if $k>1$ and
$E^1 \oplus \cdots \oplus E^i$ dominates $E^{i+1} \oplus \cdots \oplus E^k$
for all $i \in \{1,\ldots, k-1\}$.

\begin{otherthm}[\cite{Bochi Viana}]\label{t.BV}
A cocycle $A \in L^\infty$ is a point of continuity of all $L_i$'s
if and only if
the Oseledets splitting is trivial or dominated.
\end{otherthm}

\begin{rem}\label{r.groups}
As shown in \cite{Bochi Viana},
the statement of Theorem~\ref{t.BV} remains true if $\GL(d,\R)$ is replaced by
any Lie group of matrices that acts transitively on the projective space,
for example the symplectic group.
\end{rem}

We will deduce from Theorem~\ref{t.BV} the following:

\begin{prop}\label{p.our BV}
If $A \in L^\infty$ has no dominated splitting then
there exists $B \in L^\infty$ arbitrarily close to $A$
whose Oseledets splitting is trivial.
\end{prop}

The proof of the proposition requires a few preliminaries.

Given a cocycle $A \in L^\infty$, 
we define $\mu_A$ as a probability measure on $\GL(d,\R)^\Z$ by
taking the push-forward of $\mu$ under the map $x \mapsto (A(f^n(x)))_n$.
Notice $\mu_A$ is invariant under the shift.
Let $\mathrm{Hull}(A) = \supp \mu_A$; this is a compact Hausdorff space.
Let $\hat{A} : \mathrm{Hull}(A) \to \GL(d,\R)$ be the projection on the zeroth coordinate,
considered as a cocycle over the shift on $\mathrm{Hull}(A)$.
This new cocycle has the advantage of being continuous.
Using the elementary properties listed above, it is easy to see
that a cocycle $A\in L^\infty$ has a dominated splitting if and only if $\hat{A}$
has one.  This means that the existence of a dominated splitting for $A$ depends only on $\mathrm{Hull}(A)$;
in particular, if $B$ has a dominated splitting and $\mathrm{Hull}(A) \subset \mathrm{Hull}(B)$,
then $A$ has a dominated splitting.

Let $\cN$ indicate the set of cocycles $A\in L^\infty$
that have no dominated splitting.
Then $\cN$ is a $G_\delta$ subset\footnote{More precisely, $\cN$ is a closed set, but we will not need this.}
of $L^\infty$, and thus a Baire space.
Indeed, the set of $A \in L^\infty$ that have a dominated splitting with fixed index and fixed $m$ as in \eqref{e.def dom}
is easily seen to be a closed set.

\begin{lemma}\label{l.L infty}
If $A\in \cN$ is a point of continuity of $L_i|\cN$ then $A$ is a point of continuity of~$L_i$.
\end{lemma}

\begin{proof}
Assume that $L_i$ is \emph{not} continuous at some $A \in \cN$;
we will show that neither is $L_i|\cN$.


Let $a^k=(a^k_n)_n$, $k \geq 0$ be a dense sequence in $\mathrm{Hull}(A)$.
For $j \geq 0$, let
$U_{k,j} \subset \GL(d,\R)^\Z$ be the set of all sequences $(x_n)_n$
with $\|x_n-a^k_n\|<2^{-j}$ for every $|n| \leq j$.  Then each
$U_{k,j}$ is open in $\GL(d,\R)^\Z$ and for each $k \geq 0$,
$\{U_{k,j}\}_{j \geq 0}$
is a fundamental system of neighborhoods of $a^k$.
Let $D_{k,j} \subset X$ be the set of all $x$ such that
$(A(f^n(x)))_n \in U_{k,j}$.  Since $a^k \in \supp \mu_A$, we have
$\mu(D_{k,j})>0$, and since $\mu$ is non-atomic, for every $l \geq 0$
we can choose a subset $D_{k,j,l} \subset D_{k,j}$
with $0<\mu(D_{k,j,l})<2^{-k-j-l}$.  Let $Z_l=\bigcup_{k,j \geq 0}
\bigcup_{|n| \leq j} f^n(D_{k,j,l})$.  Then $\mu(Z_l) \to 0$ as $l \to
\infty$.  Moreover, if $B \in L^\infty$ is any cocycle that coincides with
$A$ on some $Z_l$, then for every $x \in
D_{k,j,l}$, and every $|n|\leq j$, we have $B(f^n(x)) = A(f^n(x))$; the definition of
$U_{k,j}$ then gives that $(B(f^n(x)))_n \in U_{k,j}$.  
This implies successively that $\mu_B(U_{k,j}) \geq \mu(D_{k,j,l})>0$ for every
$k,j \geq 0$, $a^k \in \mathrm{Hull}(B)$ for every $k \geq 0$,
$\mathrm{Hull}(B) \supset \mathrm{Hull}(A)$, and $B \in \cN$.


Since $L_i$ is upper-semicontinuous and not continuous at $A$,
there exists a sequence $A_n \in L^\infty$ converging to $A$
and $\eps>0$ such that $L_i(A_n) < L_i(A) - \eps$ for each~$n$.
Let $B_{n,l}$ be the cocycle equal to $A$ on $Z_l$ and equal to $A_n$ elsewhere.
By Lemma~\ref{l.semicont}, for each $n$ there exists $l_n$ such that
$L_i(B_{n,l_n}) < L_i(A_n) + \eps/2$.
Thus the sequence
$B_{n,l_n}$ is in $\cN$, converges to $A$,
and satisfies $L_i(B_{n,l_n}) < L_i(A) - \eps/2$.
This shows that $L_i|\cN$ is not continuous at $A$, as desired.
\end{proof}

Now we can give the:

\begin{proof}[Proof of Proposition~\ref{p.our BV}]
Let $A$ be a element of $\cN$, that is, a cocycle without dominated splitting.
Since $\cN$ is a Baire space and the functions $L_i$ are upper-semicontinuous,
we can find a point $B$ of continuity of all functions $L_i|\cN$ that is as close to $A$ as desired.
By Lemma~\ref{l.L infty}, $B$ is a point of continuity of all $L_i$'s,
and thus, by Theorem~\ref{t.BV}, its Oseledets splitting is either dominated or trivial.
Since $B \in \cN$, the former alternative is forbidden
and thus all Lyapunov exponents of $B$ are equal.
\end{proof}

\bigskip

Now let us use these results to prove Theorem~\ref{t.original BDP}.
Our approach to needs a suitable measure to start with:

\begin{lemma}\label{l.erg full supp}
For every homoclinic class $H$, there exists an ergodic invariant
probability measure whose support is $H$.
\end{lemma}

\begin{proof}
This is a simple consequence of the fact that
any non-trivial homoclinic class $H$ is contained in the closure of a
countable union of {\it horseshoes} $H_1 \subset H_2 \subset \cdots$
(by a horseshoe we mean an invariant compact set restricted to which the
dynamics is topologically conjugate to a transitive subshift of
finite type).
This allows one to construct a wealth of invariant
measures with support $H$ (for instance, with positive entropy),
as suitable ``infinite Markovian'' measures, but below we will proceed by a
somewhat less direct argument.

Given a compact invariant set $X \subset M$, let $\cM(X)$ be
the set of invariant probability measures $\mu$ with $\supp \mu \subset X$, endowed with
the weak-star topology.  Let $\cM_e(X) \subset \cM(X)$ be the set of ergodic
measures, and for any compact subset $Y \subset X$, let $\cM(X,Y)$ be the
set of invariant measures whose support contains $Y$.  It is easy to see
that $\cM_e(X)$ and $\cM(X,Y)$ are always $G_\delta$ subsets of $\cM(X)$.

Since $H_i$ is a horseshoe, both $\cM_e(H_i)$ and
$\cM(H_i,H_i)$ are dense (and hence residual) in $\cM(H_i)$.
Let $G_i = \cM_e(H_i) \cap \cM(H_i,H_i)$.
Let $W \subset \cM(H)$ be the closure of the union of the $\cM(H_i)$.
Let $W_i = W \cap \cM_e(H) \cap \cM(H,H_i)$, which is a $G_\delta$-subset of $W$.
Notice that $W_i$ contains $G_j$ for each $j \ge i$.
Since $G_j$ is a $G_\delta$-dense subset of $\cM(H_j)$,
it follows that $W_i$ is dense in $W = \overline {\bigcup_{j \geq i} \cM(H_j)}$ for every $i$.
Now, $W$ is a compact Hausdorff and hence Baire space,
and we conclude that $\bigcap W_i$ is a dense subset of $W$.
Since $H = \overline {\bigcup H_i}$, the set
$\bigcap W_i$ is precisely $W \cap \cM_e(H) \cap \cM(H,H)$.
In particular, $\cM_e(H) \cap \cM(H,H)$ is non-empty, as desired.
\end{proof}

\begin{proof}[Proof of Theorem~\ref{t.original BDP}]
Let $f$ be a diffeomorphism and let $H$ be a homoclinic class that has no dominated splitting.
Choose an ergodic probability measure $\mu$ whose support is $H$, using Lemma~\ref{l.erg full supp}.

We will consider $L^\infty$-perturbations $A$ of the derivative of $f$ restricted to $H$.
Such an object $A$ is the assignment
for $\mu$-a.e.\ $x\in H$ of a linear map
$A(x) : T_x M \to T_{f(x)} M$ that is close to $Df(x)$, and varies measurably.
Now, using Proposition~\ref{p.our BV} 
we can find such a perturbation $A$ of the derivative whose Lyapunov exponents coincide
$\mu$-almost everywhere.
Using Lusin's Theorem, we may alter $A$ on a set of arbitrarily small $\mu$-measure,
while keeping it uniformly close to $Df$, to obtain a \emph{continuous} perturbation $B$.
It follows from Lemma~\ref{l.semicont} that
the Lyapunov exponents of $B$ are all close to each other $\mu$-almost everywhere.
In other words,
there is a small number $\eps>0$ such that
$$
\lim_{n \to \infty} \frac{1}{n} \log \frac{\|B_f^n(x)\|}{\m(B_f^n(x))} \le \frac{\eps}{2}
\quad \text{for $\mu$-a.e.\ $x\in H$,}
$$
where we indicate $B_f^n(x) = B(f^{n-1}(x)) \cdots B(x)$.
Next we apply the Ergodic Closing Lemma (imitating the proof of Lemma~\ref{l.almost elliptic})
and find a $C^1$-perturbation $\tilde f$ of $f$ that has a periodic point $x$ of period $p$
such that
$$
\frac{\|B_{\tilde f}^{mp}(x)\|}{\m(B_{\tilde f}^{mp}(x))} < e^{\eps m p}
\quad \text{for some $m \ge 1$.}
$$
This implies that the moduli of the eigenvalues of $B_{\tilde f}^{p}(x)$
are all close to each other.
By means of an (easier) dissipative analogue of Lemma~\ref{l.nicolas},
we can perturb $B$ along the $\tilde f$-orbit of $x$
to make the eigenvalues of $B_{\tilde f}^{p}(x)$ of the same moduli.
By Franks' Lemma one can perturb the diffeomorphism again,
keeping the periodic orbit and inserting the desired derivatives.
This concludes the proof.
\end{proof}

\end{document}